\documentclass[11pt,reqno]{amsart}
\usepackage{fullpage}
\usepackage{comment}
\usepackage{footmisc}
\usepackage{amssymb}
\usepackage{tikz-cd}
\usepackage{tikz}
\usepackage{amsthm}
\usepackage{lscape}
\usepackage{array}
\usepackage{colonequals}
\usepackage{graphicx}
\usepackage{caption}
\usepackage{subcaption}
\usepackage[pagebackref=true, colorlinks = true, allcolors=magenta]{hyperref}
\usepackage{cleveref} 
\usepackage{longtable}
\usepackage{booktabs}
\usepackage{float}
\usepackage{makecell}
\usepackage{array}
\usepackage{footnote}

\renewcommand*{\backrefalt}[4]{%
  \ifcase #1 %
    \relax
  \or
    $\uparrow$#2.%
  \else
    $\uparrow$#2.%
  \fi%
}

\DeclareMathOperator{\Ns}{Ns}
\DeclareMathOperator{\ns}{ns}

\DeclareMathOperator{\ag}{ag}
\DeclareMathOperator{\cyc}{cyc}

\DeclareMathOperator{\sgn}{sgn}
\DeclareMathOperator{\Aut}{Aut}

\DeclareMathOperator{\Gal}{Gal}

\DeclareMathOperator{\lcm}{lcm}

\DeclareMathOperator{\PSL}{PSL}
\DeclareMathOperator{\GL}{GL}
\DeclareMathOperator{\PGL}{PGL}

\DeclareMathOperator{\SL}{SL}

\DeclareMathOperator{\Hom}{Hom}

\DeclareMathOperator{\red}{red}
\DeclareMathOperator{\gon}{gon}

\theoremstyle{plain}\newtheorem{ithm}{Theorem}
\theoremstyle{plain}
\theoremstyle{plain}
\theoremstyle{plain}
\theoremstyle{plain}\newtheorem{iconj}[ithm]{Conjecture}
\theoremstyle{plain}\newtheorem{iex}[ithm]{Example}

\newtheorem{theorem}{Theorem}[section]
\newtheorem*{theorem*}{Theorem}
\newtheorem{lemma}[theorem]{Lemma}
\newtheorem{conjecture}[theorem]{Conjecture}
\newtheorem{proposition}[theorem]{Proposition}
\newtheorem*{proposition*}{Proposition}
\newtheorem{corollary}[theorem]{Corollary}
\newtheorem*{corollary*}{Corollary}

\newtheorem{question}[theorem]{Question}
\theoremstyle{definition}
\newtheorem{remark}[theorem]{Remark}
\newtheorem{definition}[theorem]{Definition}
\newtheorem{example}[theorem]{Example}

\numberwithin{equation}{section}

\newcommand{\Q}{\mathbb{Q}}
\newcommand{\Z}{\mathbb{Z}}

\newcommand{\Zhat}{\widehat{\mathbb{Z}}}

\newcommand{\C}{\mathbb{C}}
\newcommand{\F}{\mathbb{F}}
\newcommand{\PP}{\mathbb{P}}

\newcommand{\Qbar}{\overline \Q}

\def\torz#1{\mathbb Z /#1 \mathbb Z}

\newcommand{\diamondop}[1]{\langle #1 \rangle} 
\newcommand{\set}[1]{\left\lbrace #1 \right\rbrace}

\newcommand{\githubbare}[1]{\href{https://github.com/nt-lib/twist-parametrized/blob/main/#1}{\path{#1}}}

\newcommand{\lmfdbmc}[1]{\href{https://beta.lmfdb.org/ModularCurve/Q/#1}{#1}}
\newcommand{\gitlink}[2]{\href{https://github.com/nt-lib/twist-parametrized/blob/main/#1}{#2}}

\title{Rational points on modular curves: parameterization and geometric explanations}

 \author[Derickx]{Maarten Derickx}
\address{Maarten Derickx, University of Zagreb, Faculty of Science, Department of Mathematics, Bijeni\v{c}ka Cesta 30, 10000 Zagreb, Croatia}
\email{\url{maarten@mderickx.nl}}
\urladdr{\url{https://www.maartenderickx.nl/}}

\author{Sachi Hashimoto}
 \address{Sachi Hashimoto, Department of Mathematics, Brown University, Box 1917, 151 Thayer St., Providence, RI, 02912, USA} 
 \email{\url{sachi\_hashimoto@brown.edu}}
 \urladdr{\url{https://sachihashimoto.github.io}}

\author[Najman]{Filip Najman}
\address{Filip Najman, University of Zagreb, Faculty of Science, Department of Mathematics, Bijeni\v{c}ka Cesta 30, 10000 Zagreb, Croatia}
\email{\url{fnajman@math.hr}}
\urladdr{\url{https://web.math.pmf.unizg.hr/~fnajman/}}
 
\author{Ari Shnidman}
 \address{Ari Shnidman, Department of Mathematics, Temple University, Philadelphia, USA} 
 \email{\url{ari.shnidman@gmail.com}}

\begin{document}
\begin{abstract}
We show that, conditional on Zywina's effective version of the Serre uniformity conjecture, there is a natural way to parameterize non-CM $\Q$-rational points on all modular curves in terms of the rational points on finitely many modular curves.  
Our proof refines Zywina's work to give a (conditional) parameterization of the  images of adelic Galois representations of elliptic curves. In particular, we show that there are $41$ $j$-invariants of elliptic curves whose associated Galois image does not vary in an infinite family.  
Using our explicit parameterization, we show that all rational points on all modular curves arise from the geometry of modular curves in a formal sense, confirming a philosophy of Mazur and Ogg. 
\end{abstract}

\maketitle

\section{Introduction}

In 1977, Mazur classified all possible rational torsion subgroups of elliptic curves over $\Q$ \cite{Mazur77RatPoints}. 
More precisely, he showed that given a finite abelian  group $H$, the open modular curve classifying elliptic curves $E$ over $\Q$ with $E(\Q)_{\mathrm{tors}} \simeq H$ has rational points if and only if it has genus $0$.

While announcing his result, Mazur formulated a more general question he called ``Program B".  It asks, for any number field $K$ and for any open subgroup $G \subset \GL_2(\Zhat)$, to classify the elliptic curves $E/K$ whose Galois representation 
\[\rho_E \colon \Gal_K \longrightarrow \Aut(\varprojlim_N E[N]) \simeq \GL_2(\Zhat)\]
has image $G_E \colonequals \rho_E(\Gal_K)$ which is conjugate to a subgroup of $G$. 
Such an elliptic curve $E$ gives rise to a $K$-rational point on the modular curve $X_G$ (see Section \ref{sec: definitions}). Thus, informally, Mazur's Program B asks for a classification of the  $K$-rational points on all modular curves, or alternatively a classification of the images of Galois representations of elliptic curves over $K$. 

Mazur's Program B is a broad aspirational goal rather than a single precise question. 
However, in the last 50 years, number theorists have proved and conjectured many precise results in this direction.  The main goal of this paper is to reformulate Mazur's Program B as three different precise questions, and to present conditional solutions to two of them, in the case $K = \Q$. One of these  questions was implicitly answered by Zywina (see \cite[Section 14]{zywina_explicit_images}); here we reformulate, systematize, and give the sharpest possible refinement that comes out of his work. 
We thereby give a (conditional) classification of the groups $G_E$, as $E$ varies over elliptic curves over $\Q$, by parameterizing them in terms of the rational points on $160$ explicit modular curves. We then use this parameterization to show that, in a certain rigorous sense, all $\Q$-rational points on all modular curves ``come from geometry''.

\subsection{Three interpretations of Mazur's Program B}

\begin{enumerate}
    \item (Algorithmic question) \label{programb:alg} The most literal interpretation of Program B is that it asks for a description of $X_G(K)$  given both $G$ and $K$. We formulate this interpretation as follows: determine an algorithm that, given a $K$ and a $G$, returns a description of $X_G(K)$. 

    \item (Parameterization question) \label{programb:param} One may also interpret Program B as asking for a generalization of Mazur's classification of the rational points on the curves $X_1(N)$ to all modular curves and all number fields. More precisely: given $K$, is there a simple classification or parameterization of all $K$-rational points on all modular curves $X_G$? Alternatively, is there a classification  of the pairs $(E,G_E)$, where $E$ is an elliptic curve over $K$. 

    \item (Geometric characterization) \label{programb:geom} Finally, one may interpret Program B as asking for a geometric characterization of the rational points on all modular curves, in the spirit of Mazur's characterization of the non-cuspidal rational points on the curves $X_1(N)$, i.e.\ that they exist if and only if the genus of $X_1(N)$ is $0$. Ogg similarly observed that the rational points on the curves $X_0(N)$ are also ``explained by geometry'' \cite{OggDiophantine}.  
\end{enumerate}

For certain subclasses of modular curves, an answer to questions (\ref{programb:param}) or (\ref{programb:geom}) also gives an answer to (\ref{programb:alg}). For example, there is now no need for an algorithm to compute $\Q$-rational points on the curves $X_1(N)$, because Mazur's theorem  classifies all the rational points on these curves.  Of course, to {\it prove} such classifications, one typically needs to explicitly compute the rational points on finitely many modular curves. Part of the appeal of ``Program B'' is this interplay between algorithms and sporadic examples on the one hand, and uniform theoretical statements on the other hand.  

Mazur's Program B itself concerns the class of {\it all} modular curves, which includes  infinitely many curves $X_G$ of genus $g > 1$ with non-cuspidal, non-CM rational points. It is a priori not clear what a ``parameterization'' or ``geometric characterization'' of the rational points on all curves $X_G$ might look like. However, Zywina \cite{zywina_explicit_images} recently discovered an underlying structure that makes such a parameterization possible, and using his work we will give conditional solutions to questions (\ref{programb:param}) and (\ref{programb:geom}) over $\Q$, by grouping modular curves into finitely many twist families.

\subsection{Parameterization}
To formulate our results precisely requires some definitions (see Section \ref{sec: definitions} for details). For each open $G \leq \GL_2(\Zhat)$, there is a
Deligne--Mumford $\Q$-stack $\mathcal X_G$ whose non-cuspidal part parameterizes pairs
$(E,\alpha)$, with $E$ an elliptic curve and $\alpha$ a $G$-orbit
of isomorphisms $\Zhat^2 \simeq \varprojlim_N E[N]$. Its coarse space $X_G$ is a smooth projective algebraic curve over $\mathbb Q$. The map $\mathcal{X}_G \to X_G$ induces a bijection on rational points\footnote{That is, it induces a bijection between the set of isomorphism classes of the groupoid $\mathcal{X}_G(\Q)$ and $X_G(\Q)$.}. We use both the stack and the coarse space in what follows, and we refer to both as ``modular curves''.
We say a $K$-rational point on a modular curve is \emph{special} if it is a cusp or a CM point.
If $G_E$ is conjugate to a subgroup of $G$, then $E$ gives rise to a rational point on $\mathcal{X}_G$. Moreover, $G_E$ is an open subgroup of $\GL_2(\Zhat)$ if and only if $E$ is non-CM, and in the CM case we understand the groups $G_E$ via the theory of complex multiplication.  Thus, version (\ref{programb:param}) of Mazur's Program B is asking for a parameterization of all non-special $K$-rational points on the modular curves $\mathcal{X}_{G_E}$.\footnote{To classify the groups $\pm G_E$, it is enough to consider the coarse curves $X_G$. This would simplify some of the stacky technicalities that follow, but such a classification does not formally imply a classification of the groups $G_E$.} 

 For each open $G$, there is a finite abelian group $T_G \leq \Aut(\mathcal{X}_G)$ of {\it toric operators}, generalizing the classical diamond operators.  The quotient $\mathcal{X}_G/T_G$ is the modular curve $\mathcal{X}_{G^{\mathrm{ag}}}$, where $G \leq G^\mathrm{ag}$ is Zywina's agreeable closure, 
 and so $T_G \simeq G^\mathrm{ag}/G$ as groups. See Section \ref{sec: definitions} for precise definitions.

The map $\mathcal{X}_G \to \mathcal{X}_{G^{\ag}}$ induced by the inclusion $G \leq G^\mathrm{ag}$ has the interesting property that all of its twists are modular curves. 
That is, for each group homomorphism $\chi \colon \Gal_\Q \to T_G$, the twisted curves $\mathcal{X}_G^\chi$ and $X_G^\chi$ (in the sense of \cite[III.1.3]{Serre-GaloisCohomology}) are respectively identified with modular curves $\mathcal{X}_{G^\chi}$ and $X_{G^\chi}$, for a certain open subgroup $G^\chi \leq \GL_2(\Zhat)$ of the same index as $G$ (Proposition \ref{prop:gps-curves}).
 The group $T_G$ acts on each curve $\mathcal{X}_G^\chi = \mathcal{X}_{G^\chi}$ in the twist family. 
The induced map on coarse spaces $\pi_G \colon X_{G} \to X_{G^\mathrm{ag}}$ is a $\overline{T}_G$-cover, where $\overline{T}_G \simeq G^{\mathrm{ag}}/{\pm}{G}$. We have $T_G = \overline{T}_G$ if and only if $-I \in G$.

 A simple but crucial observation is that the non-special rational points on the coarse curves $X_{G^\chi}$ (for fixed $G$ but varying $\chi$) are parameterized by the non-special rational points on the {\it single} modular curve $X_{G^\mathrm{ag}}$. Indeed, if $x \in X_{G^{\mathrm{ag}}}(\Q)$ is non-special, we may identify $\pi_G^{-1}(x)$ with $\overline{T}_G$, and then  $\Gal_\Q$ acts on the fiber $\pi_G^{-1}(x)$  via some character $\chi_x \colon \Gal_\Q \to \overline{T}_G$. The rational point $x$ then lifts to a unique $\overline{T}_G$-orbit of rational points on $X_G^{\chi_x}$ (which is identified with $X_{G^\chi}$, for any $\chi\colon \Gal_\Q \to T_G$ lifting $\chi_x$). 
Thus, the data of just $X_G \to X_{G^{\mathrm{ag}}}$ and $X_{G^\mathrm{ag}}(\Q)$  allows for a parameterization of the rational points on all the curves $ X_{G^\chi}$.

With this structure in mind, the following is a precise formulation of question (\ref{programb:param}) above.  
\begin{iconj}[Galois Images Parametrization Conjecture]\label{conj: mazur program B revised}
    There exist finitely many explicit open subgroups $G_1,\ldots, G_n$ of $\GL_2(\Zhat)$ with the following property. If $E$ is a non-CM elliptic curve over $\Q$, then there is an index $i$ and a point $x \in X_{G_i^\mathrm{ag}}(\Q)$ with $j$-invariant $j(x) = j(E)$ such that  
    $G_E$ is conjugate to  $G^\chi_i$, where $\chi \colon \Gal_\Q \to T_{G_i}$ is a lift of the character $\chi_x \colon \Gal_\Q \to \overline{T}_{G_i}$ describing the $\overline{T}_{G_i}$-action on $\pi_{G_i}^{-1}(x)$. 
\end{iconj}

The conjecture says that the groups $G_E$, as $E$ varies, are ``parameterized'' by the rational points on the finitely many modular curves $X_{G_i^{\mathrm{ag}}}$. Of course, one must remember the original groups $G_i$ because the parameterization uses the maps $X_{G_i} \to X_{G_i^\mathrm{ag}}$ and the groups $G_i^\chi$. We note that $\chi$ is uniquely determined from $E$ and $x$, even if $\overline{T}_{G_i} \neq T_{G_i}$. In that case, a quadratic twist $E^d$ of $E$ gives rise to the same point $x$, and its corresponding character $\chi_d$ is the ``quadratic twist'' of $\chi$. Thus, strictly speaking, the sets $X_{G_i^\mathrm{ag}}(\Q)$  parameterize the groups $G_E$ only up to quadratic twist.

We note that the {\it existence}  of finitely many curves $G_i$ as in Conjecture \ref{conj: mazur program B revised} is equivalent to the Serre uniformity conjecture since 
there are only finitely many agreeable modular curves of bounded index (Proposition \ref{prop: bounded index}), 
but even this is an open question. However, a crucial part of the conjecture is to determine the $G_i$ {\it explicitly}.  Provably determining the set $X_{G_i^{\mathrm{ag}}}(\Q)$ for each $G_i$ should of course also be considered part of the conjecture, and this can itself be a difficult task (see e.g. \cite{Balakrishnan}).

In \cite{zywina_explicit_images}, Zywina 
gives a list of 524 groups\footnote{See the two paragraphs after \cite[Theorem 1.9]{zywina_explicit_images} for the genera of these groups.}  $G_i$ which he proves satisfy Conjecture \ref{conj: mazur program B revised} assuming Conjecture \ref{conj: GIC}.
Our first result is the following reinterpretation and refinement of Zywina's results, giving an optimal conjectural form of the solution to Conjecture \ref{conj: mazur program B revised}.
\begin{ithm}\label{thm 1+2}
    Assume Conjecture \ref{conj: GIC}. For $i \in \{1,\ldots, 160\}$, let $(K_i,M_i)$ be the groups listed in Tables~\ref{table:twist isolated} and \ref{table:bigtableofgroups}, corresponding to the 
    $T_i$-cover $\mathcal{X}_{K_i} \to \mathcal{X}_{M_i}$, where $T_i = M_i/K_i$ and the $\overline{T}_i$-cover $\pi_i \colon X_{K_i} \to X_{M_i}$, where $\overline{T}_i = M_i/{\pm}{K_i}$.  Then for any  non-CM elliptic curve $E$ over $\Q$, there exists an index $i$ and a rational point $x\in X_{M_i}(\Q)$ with $j(x)=j(E)$ such that $G_E$ is conjugate to $K_i^{\chi}$, where $\chi \colon \Gal_\Q \to T_i$ is a lift of the character $\chi_x \colon \Gal_\Q \to \overline{T}_i$ describing the   $~\overline{T}_i$-action on $\pi_i^{-1}(x)$. 
\end{ithm}

Theorem \ref{thm 1+2} states that the rational points on the 160 curves $X_{M_i}$ parameterize the groups $G_E$ arising from elliptic curves, up to quadratic twist. 
Moreover, it is not possible to choose fewer modular curves (or equivalently fewer pairs $(K_i,M_i)$), to parametrize all possible $G_E$.
Among the $X_{M_i}$, the first 138 have infinitely many rational points and the remaining 22 have finitely many rational points.  For $i \in \{1,\ldots, 138\}$, we call the curves $X_{K_i}^\chi$ {\it twist parameterized}, since their rational points are collectively parameterized by the infinitely many rational points on the single curve $X_{M_i}$. See Section \ref{sec:examples} for examples of how this parameterization looks in practice.  For $i \in \{139,\ldots, 160\}$, we call the curves $X_{H_i}$ {\it twist isolated}. See Section \ref{sec:twistisolatedparam} for the proper definition of both of these terms.

A non-CM $j$-invariant $j_0 \in \Q$ is \textit{twist isolated} if there exists some open $G\leq \GL_2(\Zhat)$ and $x\in \mathcal X_G(\Q)$ with $j(x)=j_0$ and $\mathcal X_G$  twist isolated.
As a byproduct of Theorem \ref{thm 1+2}, we prove:
\begin{ithm}\label{thm: 41 twist isolated points}
    Assuming Conjecture \ref{conj: GIC}, the twist-isolated $j$-invariants are the $41$ listed in  Table \ref{table:twist isolated}.
\end{ithm}
In other words, there are $41$ elliptic curves over $\Q$ (up to quadratic twist) whose Galois representations $\rho_E$ do not vary algebraically in an infinite family. 
These $41$ twist-isolated $j$-invariants correspond to rational points on the $22$ modular curves $X_{M_i}$ with $\#X_{M_i}(\Q) < \infty$.

\subsection{Geometric characterization}

Our second main result is a formulation and conditional proof of question (\ref{programb:geom}), which asks for a geometric characterization of all rational points on all modular curves $X_G$, in the spirit of Ogg (see for example \cite{BalakrishnanMazur}).
The starting point of Ogg's philosophy is Mazur's theorem that the only rational points on the curves $X_1(N)$ are those that can be generated by the cusps using the geometry of the modular curve. Indeed, if $X_1(N)$ has non-cuspidal rational points, then it is a plane conic and all rational points are obtained by intersecting the conic with rational lines through one of the rational cusps. 

The pattern seemed to break down with Mazur and Kenku's classification of the rational points on $X_0(N)$, because some of these curves have positive genus and non-special rational points.  However Ogg  pointed out that each of these rational points is also explained by geometry \cite{OggDiophantine}.  For example, the two non-special rational points on  $X_0(37)$ are obtained from the two rational cusps by  applying the hyperelliptic involution. In Ogg's words \cite{Ogg74}, ``we are certainly interested in knowing when this sort of thing is going on, and in putting a stop to it if at all possible.''

The parameterization of non-special rational points on the curves $X_{G_E}$ in Theorem \ref{thm 1+2} implies a similar parameterization for non-special rational points on {\it all} modular curves, since the former are cofinal among the latter. It is therefore natural to ask whether all rational points on all modular curves are explained by geometry, as Ogg would have liked.   To formalize this notion, we declare
\begin{enumerate}
    \item All special points are ``explained'',
\end{enumerate} 
and then specify four geometric operations that propagate explained rational points from one modular curve to another. Roughly speaking, the four basic operations are: \begin{itemize}
    \item[(2)] ``push-forward'': the image of a rational explained point under a morphism is explained,
    \item[(3)] ``abelian lifts'': any lift of a rational explained point along an abelian cover is explained, 
    \item[(4)] ``collinearity'': if all but one rational point $y$ of a rational hyperplane section of $X_H$ is explained, then $y$ is explained as well, and
    \item[(5)] ``fiber splicing'': given two modular parameterizations $f_1 \colon X_H \to E_1$ and $f_2 \colon X_H \to E_2$ of elliptic curves $E_1,E_2$ over $\Q$, and two explained points $x_1\in E_1(\Q)$ and $x_2 \in E_2(\Q)$, if the intersection of the fibers $f_1^{-1}(x_1) \cap f_2^{-1}(x_2)$ is a single point $y$, then $y$ is explained.
\end{itemize}

We refer to Section \ref{sec: geometric points} for the precise definition of rational points $y \in X_H(\Q)$ {\it explained by geometry}, as $X_H$ varies over all modular curves. (In particular, we impose additional assumptions on the data to assure that our ``explanations'' are truly geometric and avoid certain tautologies.)

We can now state our second main result.
\begin{ithm}\label{thm: B}
Assuming Conjecture \ref{conj: GIC}, all rational points on all modular curves are explained by geometry. More precisely, for any finite index subgroup $H \leq \GL_2(\widehat{\Z})$, and for any $y \in X_H(\Q)$, $y$ is explained by geometry in the sense of Section $\ref{sec: geometric points}$.  
\end{ithm}

This realizes the vision of Mazur and Ogg that rational points on modular curves exist if and only if there is a good geometric reason for them to exist. The result is conditional, but we may formulate an unconditional result by replacing ``all rational points'' with ``all known rational points''.

The proof of Theorem \ref{thm: B} proceeds in several steps. First, we use the explicit classification in Theorem \ref{thm 1+2} to prove that all rational points on modular curves with infinitely many rational points are explained. Axiom (3) then shows that any unexplained rational point comes from an unexplained twist-isolated rational point. Thus, it is enough to show that all rational points having one of the $41$ twist isolated $j$-invariants of Theorem \ref{thm: 41 twist isolated points} are explained. We then reduce to checking that the rational points on 22 modular curves $X_{M_i}$ are explained, which we check ``by hand''.

To give a flavor of how these geometric explanations work in practice, we summarize our explanation of the rational points on the genus 3 curve $X_{S_4}(13)$ with label $\href{https://beta.lmfdb.org/ModularCurve/Q/13.91.3.a.1/}{13.91.3.a.1}$. 

\begin{iex}
{\em 
    $X = X_{S_4}(13)$ embeds canonically in $\PP^2$ as a smooth plane quartic curve. 
    It has three non-special rational points $y_0,y_1,y_2$.  Since $X$ has no distinguished maps to lower genus curves  and $\Aut(X)$ is trivial, it is not clear where $y_0,y_1,y_2$ ``come from''.  Searching among cubic CM points, we find  three conjugate  points with CM by $\Z[2i]$ that are $\Q$-collinear. 
    By Bezout's theorem, the fourth point on $X$ in their span is necessarily rational, which explains $y_0$. 
    We similarly check that $y_0$ is $\Q$-collinear with a pair of conjugate points with CM by $\Z[i]$. The fourth point in their span is rational, explaining $y_1$.  At this point, the collinearity relations on $X$ run dry. The point $y_2$ is explained by fiber splicing on a genus 10 double cover  $X_H \to X$; see Section \ref{subsec:genus3} for details. }
\end{iex}

\subsection{Discussion}
To summarize, we have presented Mazur's Program B as the union of three different points of view. We explained how recent work of Zywina can be thought of as formulating and conditionally proving a precise version of the parameterization question (\ref{programb:param}) over $\Q$. We then presented our refined version of Zywina's solution in Theorem \ref{thm 1+2}. Finally, we formulated and conditionally proved a precise version of the geometric characterization question (\ref{programb:geom}) over $\Q$ in Theorem \ref{thm: B}, whose proof relies on Theorem \ref{thm 1+2}. 

It is natural to ask about the implications of these results to the algorithmic question (\ref{programb:alg}). Our classification results for all modular curves do not render question (\ref{programb:alg}) moot. For example,  consider the genus 2 modular curve $X = X_{K_{71}}$ with label \href{https://beta.lmfdb.org/ModularCurve/Q/18.54.2.a.1/}{18.54.2.a.1} and Weierstrass model
    \[y^2 =  6x^6 - 9x^5 - 18x^4 + 33x^3 + 9x^2 - 36x - 12 =: f(x)\]
    We have $T_{71} = \Aut(X) = \Z/2\Z$, so $X$ has the property that all of its quadratic twists 
    \[X_d \colon dy^2 = f(x)\] 
    are also modular curves. 
    Our results do not answer the algorithmic question: given $d$, is there a non-special rational point in $X_d(\Q)$? Instead, they say that the set of such $d$ is the projection of $\{f(t) \colon t \in \Q\}$ to $\Q^\times/\Q^{\times 2}$ since $t \in \Q$ gives rise to the rational points $(t,\pm 1)$ on $X_{f(t)}$. In other words, we parameterize the rational points on the curves $X_d$ by $\PP^1_t$ in one fell swoop, without specifying for each individual $d$ whether $X_d(\Q)$ has a rational point.\footnote{On one hand, this is similar to how we parameterize elliptic curves $E$ over $\Q$ with a subgroup of order $N \leq 10$, by specifying that $j(E)$ is of the form $g_N(t)$, for some explicit rational function $g_N(x)$ and some $t \in \Q$. On the other hand, verifying whether $j(E)$ is of this form is much easier than verifying whether a given $X_d$ has a rational point.} Thus, our results do not resolve question (\ref{programb:alg}), even conditionally.  However, we do expect that we have given the best possible answer to question (\ref{programb:param}), since it is hard to fathom a better way to describe the projection of $\{f(t) \colon t \in \Q\}$ to $\Q^\times/\Q^{\times 2}$. It seems that the algorithmic question (restricted to finitely many specific twist families) is an indispensable aspect of any ``complete'' solution to Mazur's Program~B.  

Of course, nowadays question (\ref{programb:alg}) is a special case of one of the fundamental open questions in arithmetic geometry: find an efficient algorithm to effectively determine the set of rational points on a curve of genus $g > 1$ over a number field. Since modular curves have rich structure, one might hope to find an algorithm tailored to them (and even if a general algorithm would be discovered, it would be interesting to find a more efficient algorithm for modular curves). 
Many of the modern developments towards an ``effective Faltings theorem'' come out of the study of rational points on modular curves; see e.g. \cite{Balakrishnan}. 

\subsection{Relation to recent work} 
The basis for this paper, especially Theorem \ref{thm 1+2}, is the recent work of Zywina on the possible adelic images of elliptic curves over $\Q$ \cite{zywina_possible_indices,zywina_explicit_images, zywina_open_image_computations}. We note that Zywina's proof of Conjecture \ref{conj: mazur program B revised} assuming Conjecture \ref{conj: GIC} allowed him to construct in \cite{zywina_explicit_images} an algorithm that computes $G_E$ for a given non-CM elliptic curve $E/\Q$. The algorithm first verifies that it does not give a counterexample to Conjecture \ref{conj: GIC}.  Then, it determines analogues of $i$, $x$, and $\chi$ in the notation as in Conjecture \ref{conj: mazur program B revised} which gives $G_E$ up to conjugacy. 

There has been considerable work in recent years on mod $\ell$ images \cite{zywina_possible_images, sutherland, BP11, BPR13, Balakrishnan, LeFournLemos, FurioLombardo251} and $\ell$-adic images of Galois \cite{RouseDZB, RSZB22, FurioLombardo252, Baletal25}, all of which constitutes progress towards Mazur's Program B. Another relevant result is Rakvi's determination of all modular curves of genus 0 with rational points \cite{Rakvi24}.

One of the original motivations for this paper was to study how sporadic and isolated points on modular curves could be explained systematically. Sporadic and isolated points on $X_0(N)$ and $X_1(N)$ have received a lot of attention recently. 
In \cite{mazur77, KM88, DerickxNajman25}, it is shown that $X_1(N)$ and $X_1(M,N)$ have no non-cuspidal sporadic points of degree $1,2$ and $4$. Najman found degree 3 sporadic points on $X_1(21)$ \cite{najman16}, while van Hoeij found sporadic points of all degrees $5\leq d \leq  13$ (see \cite{vanHoeij} for the points, and \cite[Appendix A]{DerickxNajman25} for a proof that points of these degrees are indeed sporadic). In \cite[Appendix A]{DerickxNajman25} Derickx and Najman prove that $X_0(N)$ has sporadic points of all degrees $d\leq 2166$, and conjecture that $X_0(N)$ has sporadic points of all degrees. In \cite{BELOV} the authors prove that there are only finitely many rational $j$-invariants associated to isolated points on $X_1(N)$, assuming Serre's uniformity conjecture, while in \cite{BHKKLMNS25} the authors conjecturally find all such rational isolated $j$-invariants on $X_1(N)$. Bourdon and Ejder \cite{BourdonEjder25} unconditionally determined all such rational isolated $j$-invariants on $X_1(\ell^n)$ and $X_0(\ell^n)$. Terao \cite[Theorem 1]{Terao24} proved that an isolated non-CM point $x\in X_G(\overline \Q)$ with $j(x)\in \Q$ necessarily induces a $\textit{rational}$ isolated point on some (possibly different) modular curve. Our results show that there exist infinitely many rational isolated $j$-invariants, but they come in finitely many ``families".

One might consider extending these results to arbitrary number fields. However, significant obstacles appear immediately. Notably, the twists of modular curves are not necessarily modular curves themselves, essentially because the Kronecker--Weber theorem does not hold beyond $\mathbb{Q}$.

\subsection{Outline}
In Section \ref{sec: definitions} we fix our definitions and conventions. We introduce the key concepts of twist-parameterized and twist-isolated curves, points, and $j$-invariants in Section \ref{sec:twistisolatedparam}.

In Section \ref{sec: twist isolated points} we develop the theory of twist-parameterized points, and show that, assuming Serre's uniformity conjecture, there are finitely many twist-isolated rational points on all modular curves. We develop an effective algorithm (Section \ref{sec:algorithm}) that can check whether a given $j$-invariant is twist parameterized. We show that any twist-isolated $j$-invariant would have to be an  ``exceptional" in Zywina's terminology (Remark \ref{rem:exc_j}). Zywina found $77$ such exceptional points, and conjectured that this list is complete. We run our algorithm through this list and obtain that there are, assuming Zywina's conjecture, only $41$ twist-isolated $j$-invariants.

In Section \ref{sec: geometric points}, we define carefully what we mean by ``rational points explained by geometry''. In the next section we (conditionally) show that all rational points on modular curves are explained by geometry. 
In Section \ref{sec:nonserre}, we use Theorem \ref{thm 1+2} to show that non-Serre curves come in 20 families and we discuss their moduli interpretations.
In Section \ref{sec: super-sporadic} we show that most of the twist families $X_{K_i}^\chi$ produce \textit{super-sporadic points}, points that lift to sporadic points on any modular cover, answering a version of a question considered in \cite[Section 7]{BELOV}.

Many theorems and examples in this work are based on code which is available at 
\begin{center}
\url{https://github.com/nt-lib/twist-parametrized/blob/main/}.
\end{center}

\section*{Acknowledgements}
We are grateful to David Zywina whose work is the inspiration for this paper. Parts of this project began at the ICERM Algebraic Points on Curves workshop and we are thankful to ICERM and the organizers of the workshop.
We thank Jacob Mayle and Jeremy Rouse for providing extremely useful code to study maps of modular curves to elliptic curves and for their help with using and adapting the code. We also thank Eran Assaf, Abbey Bourdon, Barry Mazur, Bjorn Poonen, and Chris Xu for helpful comments. 

M.D. and F.N. were financed by the Croatian
Science Foundation under the project no. IP-2022-10-5008 and by the  project ``Implementation of cutting-edge research and its application as part of the Scientific Center of Excellence for Quantum and Complex Systems, and Representations of Lie Algebras'', PK.1.1.10.0004, European Union, European Regional Development Fund. 
S.H. was partially supported by the National Science Foundation under Award No.\ DMS-2501658.  A.S. was funded by the European Research
Council (ERC, CurveArithmetic, 101078157).

\section{Definitions, conventions and notation}\label{sec: definitions}

We will mostly follow the conventions and notation by Zywina in \cite{zywina_possible_images, zywina_possible_indices,zywina_explicit_images, zywina_open_image_computations}. For a group $G$, we denote its commutator subgroup by $[G,G]$. For a number field $K$, we denote its absolute Galois group by $\Gal_K:=\Gal(\overline K/K)$.

\begin{definition}\label{def:goursat_groups}
     For an integer $M$, let $\Z_M \colonequals \prod_{p\mid M} \Z_p$ and $\Z'_M \colonequals \prod_{p\nmid M} \Z_p$. If $G\leq \GL_2(\Zhat)$ is an open subgroup, then we define $G_M$  to be the image of $G$ under the canonical projection $\GL_2(\Zhat) \to \GL_2(\Z_M)$. 
\end{definition}

\begin{definition}[Level]
    The \textit{level} (or $\GL_2$-level) of an open subgroup $G$ of $\GL_2(\Zhat)$ is the smallest positive integer $N$ for which $G$ contains the kernel of the reduction modulo $N$ homomorphism $\pi_N \colon \GL_2(\Zhat)\rightarrow \GL_2(\torz{N})$. 
    The level of an open subgroup $G$ of $\GL_2(\Z_N)$ is defined as the level of $G \times \GL_2(\Z_N) \leq \GL_2(\Zhat)$. We define the level of open subgroups of $\SL_2(\Zhat)$ and $\SL_2(\Z_M)$ completely analogously. For an open subgroup $G\leq \GL_2(\Zhat)$, we define the $\SL_2$-level of $G$ to be the level of $G\cap \SL_2(\Zhat)$. For $G \leq \GL_2(\Zhat)$ we write $N(G)$ for the level of $G$.
\end{definition}

We briefly recall the definition of the modular curves $\mathcal{X}_G$ and $X_G$; see \cite[IV - 3]{DeligneRapoprt73} or \cite[Section 2]{RouseDZB} for more details.  For $N \geq 1$, let $\mathcal{Y}(N)$ be the moduli stack over $\Q$ of pairs $(E,\alpha)$ where $E$ is an elliptic curve and $\alpha$ is an isomorphism $(\Z/N\Z)^2 \simeq E[N]$. This is a 1-dimensional Deligne--Mumford $\Q$-stack, which
 is a scheme if $N \geq 3$. The group $\GL_2(\Z/N\Z)$ acts on $\mathcal{Y}(N)$ through the level structures $\alpha$. For $G \leq \GL(\torz{N})$, let $\mathcal{Y}_G$ be the stack quotient $[\mathcal{Y}(N)/G]$, which is in general only a Deligne--Mumford stack over $\Q$, and which parameterizes $G$-orbits of pairs $(E,\alpha)$. Let $Y_G$ be the coarse space of $\mathcal{Y}_G$, a smooth algebraic curve over $\Q$. (If $N \geq 3$, this is simply the scheme quotient $\mathcal{Y}(N)/G$ given by taking $G$-invariants of the coordinate ring.) 
The $K$-rational points of $\mathcal{Y}_G$ correspond to pairs $(E, \alpha)$ such that $\rho_{E,N}(\Gal_K)$ is conjugate to a subgroup of $G$. Finally, let $\mathcal{X}_G$ and $X_G$ be the smooth compactifications of $\mathcal{Y}_G$ and $Y_G$ respectively. For an open subgroup $G \leq \GL_2(\Zhat)$, we define $\mathcal X_{G}\colonequals \mathcal X_{\pi_{N(G)}(G)}$ and $X_G = X_{\pi_N(G)}$. 

The stack $\mathcal{X}_G$ is a finite cover of the moduli stack $\mathcal{X}(1)$ of generalized elliptic curves.  Similarly, the curve $X_G$ is a finite cover of the coarse moduli space $X(1)$ of generalized elliptic curves. The map $X_G \to X(1) \simeq \PP^1$ sends $(E,\alpha)$ to the $j$-invariant $j(E)$. 
\begin{remark}\label{rem: stacks}
    The main reason we use stacks throughout the paper is to maintain bijective correspondence with open subgroups $G\leq \GL_2(\Zhat)$. We always have $X_G \simeq X_{\pm G}$ (as schemes), but if $-I\notin G$ then $\mathcal X_G\not\simeq \mathcal X_{\pm G}$. To obtain a classification of all possible $G_E$, it is crucial to differentiate between the geometric objects corresponding to $\pm G$ and $G$.

    The exception to this are the results dealing with rational points on modular curves, where we may and often do work with coarse space,  since the rational points on $\mathcal X_G$, $\mathcal X_{\pm G}$, $X_G$ and $X_{\pm G}$ are all in bijection to each other.  
    Note also that if $-I \notin G$, then the open substack of $\mathcal{X}_G$ where $j \neq 0, 1728, \infty$ is a scheme 
    and is identified with its coarse space.
\end{remark}

\begin{remark}
    If $G\leq \GL_2(\Zhat)$ and $\det(G) \neq \Zhat^\times$, then $X_G$ is connected over $\Q$, but geometrically disconnected, and hence cannot have any rational points.
\end{remark}

For an elliptic curve $E/\Q$, we denote by $G_E$ its adelic image of Galois $G_E\colonequals \rho_E(\Gal_\Q).$ We view $G_E$ as a subgroup of $\GL_2(\widehat{\Z})$ well-defined up to conjugacy. Then $E$ gives a $\Q$-rational point on $X_{G_E}$. For a non-CM $j$-invariant $j_0\in \Q$, define $G_{j_0}\colonequals \pm G_{E}$, where $E$ is any elliptic curve with $j(E)=j_0.$ For a point $x\in X_G(\Q)$, we define $G_x\colonequals G_{j(x)}$. 

\begin{definition}[Modular maps]
We say the maps $\mathcal X_{H_1} \to \mathcal X_{H_2}$ and $X_{H_1} \to X_{H_2}$ are {\it modular} if they are induced by an inclusion $H_1 \leq H_2 \leq \GL_2(\widehat{\Z})$.   
\end{definition}

The modular map $\mathcal X_{H_1} \to \mathcal X_{H_2}$ is abelian if and only if $H_1$ is normal in $H_2$ and $H_2/H_1$ is abelian. The modular map $X_{H_1} \to X_{H_2}$ is abelian if and only if $\pm H_1$ is normal in $\pm H_2$ and $\pm H_2/{\pm{H_1}}$ is abelian.

\subsection{Agreeable groups}
The group-theoretic notion of agreeable closure plays an important role in Zywina's classification.

\begin{definition}[Agreeable groups]\cite[\S 4.1]{zywina_open_image_computations}
A subgroup $G$ of $\GL_2(\Zhat)$ is \textit{agreeable} if it is open, contains all the scalar matrices  in $\GL_2(\Zhat),$ and each prime dividing the $\GL_2$-level of $G$ also divides the $\SL_2$-level of $[G,G]$.
\end{definition}

\begin{definition}[Agreeable closure]\cite[\S 4.1]{zywina_open_image_computations}
\label{def:agreeableclosure}
Let $G \leq \GL_2(\Zhat)$ be an open subgroup and let $M$ be the $\SL_2$-level of $[G,G]$, then the  \textit{agreeable closure of $G$} is defined to be the subgroup $G^{\ag} \leq \GL_2(\Zhat)$ given by
$$G^{\ag} :=(\Z_M^\times G_M) \times \GL_2(\Z_M').$$
\end{definition}

\begin{lemma}\cite[Lemma 4.1]{zywina_open_image_computations} \label{lem:agreeable}
Let $G \leq \GL_2(\Zhat)$ be an open subgroup. Its agreeable closure $G^{\ag}$ satisfies the following properties:
\begin{enumerate}
    \item $G \leq G^{\ag}$,
    \item $G^{\ag}$ is agreeable,
    \item If $H\leq \GL_2(\Zhat)$ is agreeable and $G \leq H$, then $G^{\ag} \leq H$,
    \item $G$ is agreeable if and only if $G=G^{\ag}$,
    \item $[G,G]=[G^{\ag},G^{\ag}]$.
    \item If $M'$ is the $\SL_2$-level of $[G,G]$, then the $\GL_2$-level of $G^{\ag}$ divides $2\lcm(M',4)$.
\end{enumerate}
\end{lemma}

\begin{lemma}
Let $H \leq G^{\ag}$ such that $\SL_2(\Zhat) \cap G= \SL_2(\Zhat) \cap H$. Then $G^{\ag}/H$ is an abelian group.
\end{lemma}
\begin{proof} Note 
$$[G^{\ag},G^{\ag}]=[G,G]\leq \SL_2(\Zhat) \cap G= \SL_2(\Zhat) \cap H\leq H.$$
Hence $G^{\ag}/H$ is an abelian group, since it is isomorphic to $(G^{\ag}/[G^{\ag},G^{\ag}])/(H/[G^{\ag},G^{\ag}])$.
\end{proof}

\begin{proposition}\cite[Lemma 1.7 and \S 1.4]{zywina_explicit_images}
\label{prop:com_im}
    Let $E$ be an elliptic curve over $\Q$ and $G_E$ its adelic image, then $[G_E^{\ag},G_E^{\ag}]=[G_E,G_E]=G_E \cap \SL_2(\Zhat)$.
\end{proposition}

\subsection{Toric operators and agreeable modular curves}

We translate the notions above to modular curves. For an open subgroup $G \leq \GL_2(\Zhat)$ and integer  $M \geq 1$, define $G_M \colonequals \pi_M(G) \leq \GL_2(\Z/M\Z)$.
 If $H \leq G$ is normal of level $M$, then there is a strict action, in the sense of \cite{Romagny}, of $G/H$ on the stack $\mathcal{X}_H = [\mathcal{X}(M)/H_M]$ and the quotient $\mathcal{X}_H/(G/H)$ is isomorphic to $\mathcal{X}_G$ \cite[Rem.\ 2.4]{Romagny}.  We sometimes write $\Aut(\mathcal{X}_H/\mathcal{X}_G)$ for  $G/H$ in this situation.

\begin{definition} The group of {\it toric operators} is the abelian group $\Aut(\mathcal{X}_G/\mathcal{X}_{G^{\mathrm{ag}}}) \simeq G^\mathrm{ag}/G$. 
\end{definition}
 For completeness, we spell out the moduli interpretation of this group.  
Let $N$ be the level of $G$. The scalar matrices in $\GL_2(\Z/N\Z)$ are central, hence induce automorphisms of $\mathcal{X}_G = \mathcal{X}(N)/G_N$. This defines an action of $(\Z/N\Z)^\times$ on $\mathcal{X}_G$, which is trivial if and only $\Zhat^\times \leq G$.
The quotient $G \Zhat^\times/G$ acts faithfully on $\mathcal{X}_G$, and we call these the \emph{diamond operators}. 
\begin{example}
{   For $\mathcal{X}_G = \mathcal{X}_1(N)$, the diamond operators are $\Aut(\mathcal{X}_1(N)/\mathcal{X}_0(N)) \simeq (\Z/N\Z)^ \times$, consistent with the classical nomenclature.}
\end{example}

For every $m\mid N$ and $(m,N/m) = 1$, the forgetful map $\mathcal{X}_G \to \mathcal{X}_{G_m}$ sends a $G$-level structure $(E,\alpha \colon (\Z/N\Z)^2 \simeq E[N])$, to the pair  $(E,\alpha')$ where $\alpha'$ is the restriction of $\alpha$ to $(\Z/m\Z)^2$. 

Now let $M$ be the level of $[G,G] \leq \SL_2(\Zhat)$ and let $m = (N,M^\infty)$, i.e.\ $m$ is the largest divisor of $N$ divisible only by primes dividing $M$. Then $G$ is normal in $G_m$ and $G_m/G$ is abelian, since $\GL_2(\Zhat)/\SL_2(\Zhat)$ is abelian. We call $\Aut(\mathcal{X}_G/\mathcal{X}_{G_m}) \simeq G_m/G$ 
the group of \emph{similitude operators}. Note that the similitude operators commute with the diamond operators. Let $T_G$ be the group generated by the diamond and similitude operators, viewed as a finite group acting on $\mathcal{X}_G$.

\begin{lemma}\label{lem:toric}
$T_G$ is the group of toric operators
$\Aut(\mathcal{X}_G/\mathcal{X}_{G^{\mathrm{ag}}})$.       
\end{lemma}

\begin{proof}
This follows from our definition of $G^{\mathrm{ag}}$, after translating through the moduli interpretation. 
\end{proof}

Let $\overline{T}_G = T_G/\{\pm I\}$, so that $\overline{T}_G \simeq \Aut(X_G/X_{G^{\mathrm{ag}}})$.
We say that $\mathcal{X}_G$ is {\it agreeable} if $T_G$ is trivial. Of course $\mathcal{X}_G$ is agreeable if and only if $G$ is agreeable.
Let $\Pi_G \colon \mathcal{X}_G \to \mathcal{X}_{G^{\mathrm{ag}}}$ and $\pi_G \colon X_G \to X_{G^{\mathrm{ag}}}$ be the natural maps.

\begin{lemma}\label{lem: agreeable quotient properties}
     The groups $T_G$ and $\overline{T}_G$ satisfy the following properties:
     \begin{enumerate}
         \item $T_G$ and $\overline{T}_G$ are abelian.
         \item $T_G \simeq G^\mathrm{ag}/G$ and $\overline{T}_G \simeq G^\mathrm{ag}/{\pm{G}}$.
         \item There is an isomorphism $\mathcal{X}_G/T_G \simeq \mathcal{X}_{G^{\mathrm{ag}}}$ of $\Q$-stacks over $\mathcal{X}(1)$
         \item There is an isomorphism $X_G/\overline{T}_G \simeq X_{G^{\mathrm{ag}}}$ of $\Q$-schemes over $X(1)$.
         \item $\mathcal{X}_{G^{\mathrm{ag}}}$ is agreeable.
        \item The map $\Pi_G$ is universal for modular maps $\mathcal{X}_G \to \mathcal{X}_H$ to agreeable modular curves $\mathcal{X}_H$. 
         \item The map $\pi_G$ is universal for modular maps $X_G \to X_H$ with $H$ agreeable. 
     \end{enumerate}
  
\end{lemma}

\begin{proof}
    This follows from Lemma \ref{lem:toric} and Lemma \ref{lem:agreeable}.
\end{proof}

\begin{proposition}\label{prop: bounded index}
    For every $B \geq 1$, there are finitely many agreeable modular curves $\mathcal X_G$ with index at most $B$. 
\end{proposition}

\begin{proof}  
Since the determinant of an agreeable $G\leq \GL_2(\Zhat)$ surjects onto $\Zhat^\times$, the index of $G$ in $\GL_2(\Zhat)$ is equal to the index of $G\cap \SL_2(\Zhat)$ in $\SL_2(\Zhat)$. Note that $\SL_2(\Zhat)$ is a finitely generated profinite group; such groups have finitely many subgroups of bounded index \cite[Proposition 2.5.1 (a)]{RibesZalesskii}. It follows from \cite[Lemma 4.1 and Theorem 2.8]{zywina_open_image_computations} that for each subgroup of $H\leq \SL_2(\Zhat)$ there are finitely many agreeable subgroups $G\leq \GL_2(\Zhat)$ such that $G \cap \SL_2(\Zhat)=H$, completing the proof. 
\end{proof}

\section{Twist-parameterized and twist-isolated curves}
\label{sec:twistisolatedparam}

We now introduce the key concepts of twist-parameterized and twist-isolated groups. 

 \begin{definition}[Twist-parameterized and twist-isolated curves] \label{def:twist_parameterized} Let $H_1 \leq H_2 \leq \GL_2(\Zhat)$ be open subgroups.
We say that $\mathcal X_{H_1}$ is \textit{twist parameterized over} $\mathcal X_{H_2}$ if the modular map $\pi:\mathcal X_{H_1}\rightarrow \mathcal X_{H_2}$ is  abelian  and $\# \mathcal  X_{H_2}(\Q)=\infty$.   For a group $G\leq \GL_2(\Zhat)$, we say that $\mathcal X_G$ is \textit{twist parameterized} if there exists a group $G \leq H \leq \GL_2(\Zhat)$ such that $\mathcal X_G$ is twist parameterized over $\mathcal X_{H}$.  If $\mathcal X_G$ is  not twist parameterized, we say that it is \textit{twist isolated}. \end{definition}

\begin{definition}[Twist-parameterized and twist-isolated points] 
 
If $\mathcal X_G$ is twist isolated, we say that any point $x\in \mathcal X_G(\Q)$ is a \textit{twist-isolated point} on $\mathcal X_G$. Similarly, if $\mathcal X_G$ is twist parameterized, then we say that any point $x\in \mathcal X_G(\Q)$ is a \textit{twist-parameterized point} on $\mathcal X_G$.

A non-CM $j$-invariant $j_0 \in \Q$ is \textit{twist isolated} if there exists an open $G\leq \GL_2(\Zhat)$ such that $x\in \mathcal X_G(\Q)$ with $j(x)=j_0$, and $\mathcal X_G$ is twist isolated.
A non-CM $j$-invariant $j_0 \in \Q$ is \textit{twist parameterized} if it is not twist isolated. 
\end{definition}

\begin{remark}\label{rem: twist isolated} \hfill
\begin{itemize}
    \item[a)] Taking $H_1= H_2$, it follows that if $\#\mathcal  X_G(\Q) = \infty$, then $\mathcal X_G$ is twist parameterized.
    \item[b)] Being a twist-isolated point is a stronger property than being an isolated point (in the sense of \cite[Definition~4.1]{BELOV}), since it follows from a) that every twist-isolated point is isolated.

    \item[c)] On the other hand, if $H$ is a maximal subgroup of $\GL_2(\Zhat)$, then a rational isolated point on $\mathcal X_H$ is also twist isolated. 

    \item[d)] Denote by $B_0(N)$ the inverse image in $\GL_2(\Zhat)$ of the upper triangular matrices in $\GL_2(\torz{N})$. This is a maximal subgroup of $\GL_2(\Zhat)$. Elliptic curves with $j=-9317$ give isolated points on $\mathcal X_0(37)=X_0(37)=X_{B_0(37)}$, so these points are twist isolated. 
\end{itemize}
\end{remark}

The following lemma motivates these terms, as it shows that if $\mathcal{X}_G$ is twist-parameterized over $\mathcal{X}_H$ then any non-special rational point in $X_H(\Q)$ lifts to a rational point on a unique twist of $X_G$. (We discuss twists in detail in Section \ref{subsec: twists of modular curves} and show that abelian twists of modular curves are again modular curves.)

\begin{lemma}\label{lem: abelian twist lifts}
Let $\pi \colon X_{H_1} \to X_{H_2}$ be an abelian modular map with $T:=\Aut(X_{H_1}/X_{H_2})$, and let $x \in X_{H_2}(\Q)$ be a non-special rational point. Then $x$ lifts to a rational point on the unique twist $X_{H_1}^\psi$ of $X_{H_1}$ by the character $\psi \colon \Gal_\Q \to T$ corresponding to the $T$-torsor $\pi^{-1}(x)$. 
\end{lemma}
\begin{proof}
    The fiber $\pi^{-1}(x)$ is a $T$-torsor, as the branch points of $\pi$ have $j$-invariant $0,1728,$ or $\infty$. For the rest, see \cite[8.4.1]{Poonen_book} 
\end{proof}

\begin{proposition} \label{prop:tp_group_condition}
    Let $G$ and $H$ be open subgroups of $\GL_2(\Zhat)$. 
    
    \begin{enumerate}
        \item $\mathcal X_G$ is a abelian cover of $\mathcal X_H$ if and only if $H$ satisfies  
        \begin{equation} \label{eq_def:param}
            [H,H] \leq G \leq H \leq \GL_2(\Zhat).
        \end{equation} 
        $\mathcal X_G$ is twist parameterized over $\mathcal X_H$ if and only if \eqref{eq_def:param} holds and $\#\mathcal X_H(\Q)=\infty$.
        
        \item If $\mathcal X_G$ is twist parameterized over $\mathcal X_H$ then $\mathcal X_G$ is twist parameterized over $\mathcal X_{H^{\ag}}$.

        \item Let $G=G_E$ for some elliptic curve $E /\Q$ and suppose $H$ is agreeable. Then $\mathcal X_G$ is twist parameterized over $\mathcal X_H$ if and only if $G^{\ag}\leq H$,
        $[H,H]= G \cap \SL_2(\Zhat)=[G^{\ag},G^{\ag}]$, and $\mathcal X_H(\Q)=\infty.$

        \item Let $G=G_E$ for some elliptic curve $E /\Q$. Then $\mathcal X_G$ is twist parameterized if there exists an agreeable group $H$ with $G^{\ag}\leq H$ such that 
        $[H,H]= G \cap \SL_2(\Zhat)=[G^{\ag},G^{\ag}] $ and $\mathcal X_H(\Q)=\infty.$

        \item If $\mathcal X_{G^{\ag}}(\Q)$ is infinite, then $\mathcal X_G$ is twist parameterized over $\mathcal X_{G^{\ag}}$.

        \item Let $E/\Q$ be an elliptic curve. Then $\mathcal X_{G_E^{\ag}}=\mathcal X_{G_j(E)^{\ag}}$, and if $\mathcal X_{G_E}$ is twist parameterized over $\mathcal X_{G_{j(E)}^{\ag}}$, then $j(E)$ is not twist isolated. 

        \item If $\mathcal X_G$ is not twist parameterized over $\mathcal X_{G^{\ag}}$, then  $\mathcal X_{G^{\ag}}$ is not twist parameterized over itself.
        \item Let $G^{\ag} \leq H^{\ag}$. Then $\mathcal X_G$ is twist parameterized over $\mathcal X_{H^{\ag}}$ if and only if $\mathcal X_{G^{\ag}}$ is twist parameterized over $\mathcal X_{H^{\ag}}$.
        
    \end{enumerate}
\end{proposition}
\begin{proof}
Part (1) follows immediately from Definition \ref{def:twist_parameterized}. Note that $\mathcal X_G \rightarrow \mathcal X_H$ is abelian if and only if $G$ is normal in $H$ and $H/G$ is abelian (which is equivalent to $[H,H]\leq G$).

Part (2) follows from (1). Indeed, if $\mathcal X_G$ is twist parameterized over $\mathcal X_H$, then the inclusion $H \leq H^{\ag}$ gives a map $\mathcal  X_H \to \mathcal X_{H^{\ag}}$. Since $\# \mathcal X_H(\Q)=\infty$, we have $\# \mathcal X_{H^{\ag}}(\Q)=\infty$ as well. Finally, the equality $[H,H]=[H^{\ag},H^{\ag}]$ implies $[H^{\ag},H^{\ag}] \leq G$.

For (3), suppose $G=G_E$ for some elliptic curve $E /\Q$. By Lemma \ref{lem:agreeable} (3) and (4) we have that $G^{\ag}\leq H=H^{\ag}$ is equivalent to $G\leq H$. Since $[\GL_2(\Zhat),\GL_2(\Zhat)]\leq \SL_2(\Zhat)$ it follows that $[H,H]\leq G\cap \SL_2(\Zhat)$ is equivalent to $[H,H]\leq G$. Note that 
$$G \cap \SL_2(\Zhat)=[G^{\ag},G^{\ag}]$$
when $G=G_E$ by \cite[Lemma 2.1 (iii)]{zywina_open_image_computations}, from which the equality now follows. 

Part (4) follows directly from (3). Part $(5)$ follows from Lemma \ref{lem:agreeable}, (1) and (5). In part (6), $G_E^{\ag} = G_{j(E)}^{\ag}$ follows immediately from the properties of the agreeable closure. The second part of (6), and (7) follow directly from the definitions.

To prove (8), note that we have 
\begin{equation}\label{eq:pf}
    [G,G]=[G^{\ag},G^{\ag}]\leq [H^{\ag},H^{\ag}].
\end{equation}
Both $\mathcal X_G$ and $\mathcal X_G$ will be twist parameterized over $\mathcal X_{H^{\ag}}$ if and only if equality holds in \eqref{eq:pf} and $\mathcal X_{H^{\ag}}(\Q)=\infty.$
\end{proof}

\subsection{Families of groups}
The following is \cite[Defintition 14.1]{zywina_explicit_images}.
\begin{definition}\label{def:twist_family}
Let 
$G \leqslant \operatorname{GL}_2(\widehat{\mathbb{Z}})$
be an open subgroup with
$\det(G)=\widehat{\mathbb{Z}}^\times$ and $-I \in G$.
Fix a closed subgroup $B$ of $G$ with $[G,G]\leqslant B$.
The \emph{family of groups} associated to the pair $(G,B)$ is
\[
\mathcal{F}(G,B)
  := \{\, H \leqslant G \mid
      \det(H)=\widehat{\mathbb{Z}}^\times
      \text{ and }
      H\cap \operatorname{SL}_2(\widehat{\mathbb{Z}})
      = B\cap \operatorname{SL}_2(\widehat{\mathbb{Z}})
   \,\}.
\]
\end{definition}

Note that if $G_0 \in \mathcal F(G,B)$ then $G_0$ is normal in $G$ and the quotient $G/G_0$ is finite abelian.

\begin{lemma} \label{lem: G_Eappears}
Let $G_E$ be the adelic image of an elliptic curve over $\Q$ then $G_E$ is an element of the family  $\mathcal F (G_E^{\ag}, [G_E^{\ag},G_E^{\ag}]).$ 
\end{lemma}

\begin{proof}
     By Proposition \ref{prop:com_im} we have $[G_E^{\ag},G_E^{\ag}]=[G_E,G_E]=G_E \cap \SL_2(\Zhat)$. Since $[G_E^{\ag},G_E^{\ag}]\leq  \SL_2(\Zhat)$, it follows that $[G_E^{\ag},G_E^{\ag}]\cap  \SL_2(\Zhat)= G_E \cap \SL_2(\Zhat)$, and hence $G_E \in \mathcal F (G_E^{\ag}, [G_E^{\ag},G_E^{\ag}]).$
\end{proof}

In particular, the above means that for every adelic image of an elliptic curve over $\mathbb Q$ there is an agreeable group $G$ such that $G_E \in \mathcal F(G, [G,G])$.

There is some ambiguity in the literature about the definition of the modular curve $X_G$: in particular, whether $\GL_2(\widehat{\Z})$ acts on the right or on the left. See \cite[Remark 2.1]{RSZB22} for a discussion of this. To avoid this issue with conventions, we prove the following lemma, which demonstrates that the ambiguity makes no difference to the results of this paper. 
\begin{lemma}
Let $\mathcal F(G,B)$ be a family of groups and assume that $G$ is agreeable, then family of groups $\mathcal F(G^\top,B^\top)$ is conjugate to $\mathcal F(G,B)$. 
\end{lemma}
\begin{proof}
Note that $F(G,B)=F(G,B\cap \SL_2(\Zhat))$, so we can restrict to the case $B \subseteq\SL_2(\Zhat)$. Letting $M = \left(\begin{smallmatrix} 0 & -1\\ 1 & 0\end{smallmatrix}\right)$, 
it is sufficient to show that $G^\top = MGM^{-1}$ and $B^\top = MBM^{-1}$.
The identity $G^\top = MGM^{-1}$ follows from the fact for all $g \in \GL_2(\Zhat)$ one has  $g^\top = \det(g) Mg^{-1}M^{-1}$ and the fact that $G$ contains all scalar matrices by the definition of agreeable groups. The identity $B^\top = MBM^{-1}$ follows similarly, except that $B$ might not contain all scalar matrices, so instead one uses $\det(g) =1$ for all $g \in B \subseteq \SL_2(\Zhat)$.
\end{proof}

\subsection{Twists of a subgroups of \texorpdfstring{$\GL_2(\Zhat)$}{GL2Zhat} by a character}

\begin{definition}\label{def:modular_twisting_character}
Let $H \leq G^{\ag}$ be such that $\SL_2(\Zhat) \cap G= \SL_2(\Zhat) \cap H$. 
The homomorphism \begin{align*}\gamma_G:\det(G)\to& G^{\ag}/H\\
 \det(g) \mapsto& gH, 
\end{align*} is well defined since $\SL_2(\Zhat) \cap G\leq H$ and hence $G \to G^{\ag}/H$ factors through the determinant map.
\end{definition}

\begin{lemma}\label{lem:reconstruct_group_from_twisting_character}
If $H \leq G^{\ag}$ is an open subgroup such that $\SL_2(\Zhat) \cap G= \SL_2(\Zhat) \cap H$, then  $$G = \set{g \in G^{\ag} | \det(g) \in \det(G),  \gamma_G(\det(g))=gH}.$$
\end{lemma}

\begin{proof}
    This is \cite[(1.1)]{zywina_open_image_computations}.
\end{proof}

\begin{definition}[Twists in terms of characters]\label{def:twist_character}Let $G,B$ be as in Definition \ref{def:twist_family}.
Assume $\mathcal{F}(G,B)\neq\emptyset$ and fix a group $H\in\mathcal{F}(G,B)$. For a continuous homomorphism
$\chi:\widehat{\mathbb{Z}}^\times \to G/H$ define
the subgroup $H^\chi\leq G$ by
\[
H^\chi := \{\, g\in G \mid gH= \chi(\det g) \,\}.
\]

\end{definition}

The following is just a rephrasing of Lemma \ref{lem:reconstruct_group_from_twisting_character}.
\begin{proposition}
Let $G,B$ and $H$ be as in Definitions \ref{def:twist_family} and \ref{def:twist_character}
Then the map
\begin{align*}
\Phi:\Hom_{\rm cts}\big(\widehat{\mathbb{Z}}^\times,\,G/H\big)&\rightarrow
\mathcal{F}(G,B)\\
\chi&\mapsto H^\chi
\end{align*}
is a bijection. Its inverse is given by $G \mapsto \gamma_G$ where $\gamma_G$ is as in Definition \ref{def:modular_twisting_character}.
\end{proposition}

\subsection{Twists of modular curves}\label{subsec: twists of modular curves}

Suppose for this subsection that $H \trianglelefteq G \leq \GL_2(\Zhat)$, and that $G$ and $H$ are open with surjective determinant. Recall that $G/H$ acts naturally on $\mathcal X_H$, and $\mathcal X_{G} \cong \mathcal X_H/(G/H)$.

For  $\psi \in \mathrm{H}^1(\Gal_\Q, G/H) = \Hom(\Gal_\Q, G/H)$, we denote by $\mathcal X_H^\psi$ the twist of $\mathcal X_H$ by $\psi$. Explicitly, $\mathcal{X}_H^\psi$ is the contracted product $\mathcal{X}_H \times^{G/H} Z$, where $Z$ is the $(G/H)$-torsor corresponding to $\psi$; because the action on $\mathcal{X}_G$ is strict, the construction is as in \cite[5.12.5.2]{Poonen_book}, but see also \cite[\S2]{etalebrauermanin2025}. We similarly define $X^\psi_H$. These twists are in general not necessarily modular curves themselves. However, when $G/H$ is abelian, they \emph{are} modular curves:

\begin{proposition} \label{prop:gps-curves}
    Assume that $G/H$ is abelian and let $\psi:  \Gal_\Q\to G/H$ be a homomorphism. The curve $\mathcal X_H^\psi$ as defined above is isomorphic to the modular curve corresponding to the group
    $$H^\psi\colonequals H^\gamma,$$
    where $H^\gamma$ is as defined in Definition \ref{def:twist_character} and $\gamma: \Zhat^\times \to G/H$ is the unique homomorphism satisfying 
$$\psi(\sigma)=\gamma (\chi_{cyc}(\sigma^{-1})) \quad \text{for all} \quad \sigma \in \Gal_\Q.$$
Moreover, $H^\psi \trianglelefteq G$ and $G/H^\psi$ is abelian.
\end{proposition}
\begin{proof}
The fact that $\mathcal X_H^\psi=\mathcal X_{H^\psi}$ follows from \cite[Proposition 2.10]{zywina_open_image_computations}. The fact that $\mathcal X_H^\psi$ is an abelian cover of $\mathcal X_G$ now follows from the fact that $H^\psi\leq G$ and $G/H^\psi$ is abelian, which can be seen from the proof of \cite[Proposition 2.10]{zywina_open_image_computations}. 
\end{proof}

The level of the modular curve $\mathcal X_H^\psi$ divides the least common multiple of the level of $G$ and the conductor of $\psi$. 
This construction motivates the following definition.

\begin{definition}
    Let $H\trianglelefteq G\leq \GL_2(\Zhat)$ be open subgroups such that $G/H$ is abelian. We call 
    $$\mathcal F(\mathcal X_G, \mathcal X_H) := \{\mathcal X_H^\psi\mid \Gal_\Q\stackrel{\psi}{\longrightarrow} G/H \text{ a homomorphism}  \}$$
    a \textit{twist family of modular curves over $\mathcal X_G$}. 
\end{definition}

By Proposition \ref{prop:gps-curves}, there is an obvious bijection between $\mathcal F(\mathcal X_G, \mathcal X_H)$ and $\mathcal F(G, H)$ sending $\mathcal X_H^\psi$ to $H^\psi.$

\begin{lemma}\label{lem:fin_over_base}
Let $G\leq \GL_2(\Zhat)$ be an agreeable group. Then there exist only finitely many twist families $\mathcal F(\mathcal X_G, \mathcal X_H)$.
\end{lemma}
\begin{proof}
     Given a fixed open $G\leq \GL_2(\Zhat)$, the twist families $\mathcal F(\mathcal X_G,\mathcal X_H)$ are uniquely determined by $H\cap \SL_2(\Zhat)$. Furthermore, $H\cap \SL_2(\Zhat)$ has to satisfy $[G,G]\cap \SL_2(\Zhat)\leq H\cap \SL_2(\Zhat) \leq G\cap \SL_2(\Zhat)$. 
        It follows that there are finitely many possibilities for $H \cap \SL_2(\Zhat)$ since $[G,G]\cap \SL_2(\Zhat)$ is of finite index in $G\cap \SL_2(\Zhat)$. 
\end{proof}

\subsection{Examples}
\label{sec:examples}
We illustrate the definitions of twist families and twist-parameterized points in the following examples.

\begin{example}[Elliptic curves with isogenies]  Let $x= (E,G)\in X_0(n)(\Q)$ for some elliptic curve $E$ over $\Q$ and $G$ a cyclic subgroup of $E$ over $\Q$ of order $n$. Let $\psi':\Gal_\Q \rightarrow \torz{n}^\times$ be the isogeny character, and let $m$ be its conductor.
    The character $\psi'$ factors through a character $\psi: \Zhat^\times \rightarrow \torz{n}^\times$, via the Kronecker--Weber isomorphism $\Gal(\Q^{\rm ab}/\Q)\simeq \Zhat^\times$.
    
Since $\chi^{\cyc}=\det \rho_{E}$ by the Weil pairing, we have  $\psi(\chi^{\cyc}(\sigma))=\psi(\det \rho_{E}(\sigma))$. Hence 
   $$\rho_{E}(\Gal_\Q)\leq B_{0,\psi}(n):=\left\{A=\begin{pmatrix}
    a& b\\ c&d
\end{pmatrix} \ \middle|\  c \equiv 0 \hspace{-0.8em}\pmod n \text{ and }    a\equiv\psi(\det A) \hspace{-0.8em} \pmod m\right\} 
\leq 
\GL_2(\Zhat).$$
 Note that $B_{0,\psi}(n)$ is a subgroup of $B_0(n)$ of index $\varphi(n)$ and level  $\lcm(m,n)$. We see that our rational point on $X_0(n)$ lifts to a rational point on $\mathcal X_{0,\psi}(n) :=\mathcal X_{B_{0,\psi}(n)}$. One can easily see that $X_{0,\psi}(n)$ is a twist of $\mathcal X_1(n)$, as they are isomorphic over $\Q(\zeta_m)$. In particular, we have $g(\mathcal X_{B_{0,\psi}(n)})=g(\mathcal X_1(n))$.

 Note that $\mathcal X_{0,\psi}(n)$ is a modular curve, of level $\lcm(n,m)$. If $n$ is such that $g(X_0(n))=0$ and $g(\mathcal  X_1(n))\geq 2$, that is $n=13,16,18$ or $25$, then every $x\in X_0(n)(\Q)$ lifts to a rational isolated point on $\mathcal X_{0,\psi}(n)$, with $\psi$ varying. Hence we have a ``parameterized family" of isolated points on modular curves. 
 \end{example}

\begin{example}[Twist-parameterized points on twists of $X_{\ns}(2)$]\label{ex: twists of Xns(2)}
We now explain the (well-known fact) that $[\GL_2(\Zhat):G_E]$ is always divisible by $2$ in terms of twist-parameterized points.  Let $\Ns(2)$ be the unique subgroup of index 2 of $\GL_2(\torz{2})$, and let $X_{\ns}(2)$ be the modular curve corresponding to this group. For any elliptic curve $E$ over $\Q$, either $\rho_{E,2}(\Gal_\Q)\leq \Ns(2) $ (in which case $E$ corresponds to a rational point on $X_{\ns}(2)(\Q)$), or since $X_{\ns}(2)\rightarrow X(1)$ is abelian, it follows that $E$ corresponds to a rational point on a nontrivial twist $X_{\ns}(2)^\psi$ of $X_{\ns}(2)$. This $\psi$ is the quadratic character associated to $\Q(\sqrt{\Delta_E})$, where $\Delta_E$ is the discriminant of $E$ (see e.g. \cite[Section 3.1.]{greicius10}). In particular, we have that $\rho_{E}(\Gal_\Q)\leq H_\psi$, where $H_\psi$ is the \textit{Serre subgroup of $\GL_2(\Zhat)$ with character $\psi$}, defined by 
 $$H_\psi= \set{g\in \GL_2(\Zhat):\sgn (g)=\psi(\det g)}, $$
 where $\sgn: \GL_2(\Zhat) \rightarrow \set{\pm 1}$ is the map obtained by composing reduction modulo $2$ and the sign map on $\GL_2(\torz{2})\simeq S_3$ (see \cite[Section 2.2.]{greicius10}). Note that $H_\psi = \Ns(2)$ if $\psi$ is the trivial character.
\end{example}

\begin{example}\label{ex: genus two example}
    $X = X_{K_{71}}$ is the genus two curve \href{https://beta.lmfdb.org/ModularCurve/Q/18.54.2.a.1/}{18.54.2.a.1} with Weierstrass model
    \[y^2 =  6x^6 - 9x^5 - 18x^4 + 33x^3 + 9x^2 - 36x - 12 =: f(x).\]
    The group $T_{71} = \Aut(X) = \overline{T}_{71}$ is generated by the hyperelliptic involution $(x,y) \mapsto (x,-y)$, so the quotient $X_{M_{71}}$ is  isomorphic to $\PP^1$. Characters $\psi \colon \Gal_\Q \to T_{71} \simeq \Z/2\Z$ are in bijection with classes $d \in \Q^\times/\Q^{\times 2}$, where $\Q(\sqrt{d})$ is the fixed field corresponding to $\ker(\psi)$. The twist family is therefore $X^d \colon dy^2 = f(x)$. Each $t \in  \PP^1(\Q)$ lifts to a rational point on the twist $X^d$, where $d = f(t)$. Explicitly, $t$ lifts to the $T_{71}$-orbit of points $\{(t,1), (t,-1)\}$ on the curve $f(t)y^2 = f(x)$.
\end{example}

\begin{example}
    The curve $X_K := X_{K_{91}}$ from Theorem \ref{thm 1+2} is the genus 5 curve 
    \[ X_{\mathrm{ns}}(2) \times_{X(1)}X_{\mathrm{sp}}^+(3)\times_{X(1)} X_{S_4}(5) = \href{https://beta.lmfdb.org/ModularCurve/Q/30.60.5.b.1/}{30.60.5.b.1},\] 
    which is a $(\Z/2\Z)^2$-cover of  
    \[X_M = X_{\mathrm{ns}}^+(3) \times_{X(1)} X_{S_4}(5) = \href{https://beta.lmfdb.org/ModularCurve/Q/15.15.1.a.1/}{15.15.1.a.1.}\] 
    The map $X_K \to X_M$ is induced from the double covers $X_{\mathrm{sp}}^+(3) \to X_{\mathrm{ns}}^+(3)$ and $X_{\mathrm{ns}}(2) \to X(1)$. We view $X_M$ as the elliptic curve $y^2 = x^3 - 2160$ with identity the unique cusp. Note that $X_M$ admits a $\mu_3$-action $(x,y) \mapsto (\zeta_3 x, y)$ inherited  from the $\mu_3$-cover $X_{\mathrm{ns}}^+(3) \to X(1)$. 
    
    One of the intermediate covers $X_K \to Y \to X_M$ is a genus $3$ curve, namely \href{https://beta.lmfdb.org/ModularCurve/Q/30.30.3.a.1/}{30.30.3.a.1}. The double cover $Y \to X_M$ is branched along $\infty$ and a $\mu_3$-orbit of CM points with $j$-invariant $1728$. It follows that $Y$ is a bielliptic Picard curve,
    and indeed we find the affine model 
    \[Y \colon y^3 = x^4 + 11x^2 + 64.\] 
    Since $X_K \to Y$ is an \'etale double cover, \cite{BruinPryms} shows that $X_K$ has a model of the form 
    \[\{q_0 = w^2\}\cap\{q_1 = wt\} \cap \{q_2 = t^2\} \subset \PP^5,\] 
    for quadratic forms $q_0,q_1,q_2 \in \Q[x,y,z]$ such that $q_0q_2 - q_1^2$ is a multiple of $-y^3z + x^4 + 11x^2z^2 + 64z^4$. Indeed, we find the following affine model for $X_K$:
    \begin{align*}
        w^2 &= 4x^2 + y^2 + 4y + 16\\
        wt &=xy - 6x\\
        t^2 &=x^2 - 4y + 16
    \end{align*}
    The involutions generating $ \Aut(X_K) \simeq (\Z/2\Z)^2$ are 
    \[(x, y , w , t) \mapsto (-x , y , -w , t)\]
    \[(x, y , w, t) \mapsto (x , y , -w, -t)\] 
    The corresponding twist family $X^{d_1,d_2}_K$ is therefore given by 
    \begin{align*}
        d_1d_2w^2 &= 4d_1x^2 + y^2 + 4y + 16\\
        d_2wt &=xy - 6x\\
        d_2t^2 &=d_1x^2 - 4y + 16
    \end{align*}
    for $(d_1,d_2) \in \Q^\times$. The isomorphism class of $X_K^{d_1,d_2}$  depends only on the image of $(d_1,d_2)$ in 
    \[(\Q^{\times}/\Q^{\times 2}) \times (\Q^{\times}/\Q^{\times 2})  \simeq \Hom(\Gal_\Q, (\Z/2\Z)^2).\] 

    To parameterize the rational points in this family, note that $X_M(\Q)$ is free of rank $1$ generated by $P = (24 : 108 : 1)$. Given non-zero $n \in \Z$, we must specify which twist $X_K^{d_1,d_2}$ the point $nP = (x_0,y_0)$ lifts to, or equivalently, the biquadratic extension splitting the fiber of $X_K \to X_M$ above $nP$. The double cover $Y \to X_M$ is $(x,y) \mapsto (4y,8x^2+44)$ \cite[pg.\ 13]{LagaShnidman-BiellipticPicardCurves}, whereas the double cover $X_K \to Y$ is obtained by taking the root of the function $4x^2 - 16y + 64 = \frac12 y_0-4x_0 +42$ on $X_M$. The point $nP$ therefore lifts to a rational point on  $X_K^{2y_0 - 88,\frac12y_0 - 4x_0 + 42}$. For example $-3P = (321,5751) \in X_M(\Q)$ lifts to the rational point $(x,y,w,t) = (2,642,3,8)$ on $X^{11414,6}_K$.
\end{example}

\begin{question}
    In the previous two examples, when we order the characters $\chi$ by conductor, what is the limiting proportion of twists $X_K^\chi$ that contain a rational non-special point? It follows from \cite[Corollary 1.2 (ii)]{Granville07} that the proportion in Example $\ref{ex: genus two example}$ is $0$, assuming the abc conjecture.
\end{question}

\section{Twist-isolated points}\label{sec: twist isolated points}
In this section we prove that there are only finitely many twist-isolated points.

\subsection{A finiteness result}
Our main finiteness result depends on the following conjecture (which was originally posed as a problem by Serre in \cite{serre72}).

\begin{conjecture}[Serre's uniformity question]\label{SUC}
There exists a bound $C$ such that for every non-CM elliptic curve $E/\Q$ and every $\ell>C$, the mod $\ell$ representation $\rho_{E,\ell}$ is surjective.     
\end{conjecture}

A stronger and more explicit version of the conjecture is given by Sutherland \cite[Conjecture 1.1.]{sutherland} and Zywina \cite[Conjecture 1.12]{zywina_possible_images}.

\begin{conjecture}[Explicit version of Serre's uniformity question]\label{ESUC}
    If $E$ is a non-CM elliptic curve over $\Q$ and $\ell>13$ is a prime, then either $\rho_{E,\ell}(\Gal_\Q)=\GL_2(\Z/\ell \Z)$ or
    $$(\ell, j(E))\in \left\{ (17,-17^2\cdot 101^3/2), (17, -17\cdot 373^3/2^{17}), (37, -7\cdot 11^3), (37, -7\cdot 137^3\cdot 2083^3) \right\}.$$
\end{conjecture}

We will need the following theorem. 

\begin{theorem}[{\cite[Theorem 1.9]{zywina_explicit_images}}]\label{thm:zywina-finite-Q} Let $\mathcal L =\{2,3,5,7,11,13,17,37\}$. There exist finite and explicitly computable sets $A_{ \infty}$ and $A_{\rm finite}$ consisting of agreeable subgroups of $\GL_2(\Zhat)$ such that
\begin{itemize}
    \item[(a)] for all $H \in A_{ \infty}$ one has $\#X_H(\Q) = \infty$, and for all $H \in A_{\rm finite}$ one has $\#X_H(\Q) <\infty$.
    \item[(b)] if $H$ is agreeable and $\#X_H(\Q)=\infty$ then $H$ is conjugate to a unique group in $A_{ \infty}$.
    \item[(c)] if $H$ is agreeable, the level of $H$ is divisible only by primes in $\mathcal L$, and $X_H(\Q)<\infty$, then $H$ is conjugate to a subgroup of a group in $A_{\rm finite}$.
\end{itemize} 
\end{theorem}

\begin{remark} \label{rem:exc_j} Zywina (see \cite[\S 10.2]{zywina_possible_images} and \cite[p.4]{zywina_explicit_images}) defines \textit{exceptional} $j$-invariants as those that are images under the $j$-map of some non-CM non-cuspidal $x\in X_G(\Q)$, where $G$ is agreeable and $X_G(\Q)$ is finite.

He found a set $J_{\rm exc}$ of 77 $j$-invariants\footnote{Zywina in \cite[p.4 and Section 12.3]{zywina_explicit_images} states there are 81 such values, but for 4 of the values $j\in\{68769820673/16, 78608, 2048, 16974593/256\}$, we find $\# X_{G_j^{\ag}}(\Q)=\infty$.} 
corresponding to non-special $x\in X_H(\Q)$ for $ H \in A_{\rm finite}$, and conjectured that there are no other, see the last paragraph of \cite{zywina_explicit_images}. See Table \ref{table:exc}.
\end{remark}

\begin{table}[h]
\centering
\begin{tabular}{llll}
$\frac{-5^{2}}{2}$ & $-2^{2} \cdot 3^{2}$ & $\frac{-2^{9} \cdot 3^{3} \cdot 5^{3} \cdot 13 \cdot 71^{3} \cdot 181^{3}}{43^{11}}$ & $\frac{3^{3} \cdot 5^{3}}{2^{6}}$ \\
$2^{6}$ & $\frac{3^{3} \cdot 29^{3} \cdot 599^{3}}{2^{24} \cdot 7^{6}}$ & $\frac{3^{3} \cdot 13}{2^{2}}$ & $-11^{2}$ \\
$\frac{11^{3}}{2^{3}}$ & $-2^{3} \cdot 3^{3}$ & $\frac{-2^{18} \cdot 3 \cdot 5^{3} \cdot 13^{3} \cdot 41^{3} \cdot 107^{3}}{17^{16}}$ & $2^{4} \cdot 3^{3}$ \\
$\frac{-3^{2} \cdot 5^{3} \cdot 101^{3}}{2^{21}}$ & $\frac{2^{15} \cdot 3^{3} \cdot 5^{3} \cdot 31^{3} \cdot 41^{3} \cdot 47^{3} \cdot 83^{3} \cdot 293^{3}}{23^{24}}$ & $\frac{-2^{6} \cdot 71^{3} \cdot 179^{3}}{3^{6} \cdot 5^{12}}$ & $\frac{5 \cdot 59^{3}}{2^{10}}$ \\
$\frac{2^{6} \cdot 5^{3} \cdot 17^{3} \cdot 5059^{3}}{11^{15}}$ & $\frac{5 \cdot 211^{3}}{2^{15}}$ & $\frac{3^{3} \cdot 5^{3}}{2}$ & $\frac{3^{2} \cdot 23^{3}}{2^{6}}$ \\[3pt]
$\frac{2^{6} \cdot 31^{3}}{3^{6}}$ & $\frac{-3^{3} \cdot 5^{3} \cdot 47^{3} \cdot 1217^{3}}{2^{8} \cdot 31^{8}}$ & $\frac{-5 \cdot 29^{3}}{2^{5}}$ & $2^{12}$ \\
$2 \cdot 3^{7}$ & $17^{3}$ & $-2^{3} \cdot 5^{4}$ & $\frac{-17 \cdot 373^{3}}{2^{17}}$ \\
$\frac{2^{4} \cdot 5 \cdot 13^{4} \cdot 17^{3}}{3^{13}}$ & $\frac{-3^{3} \cdot 11^{3}}{2^{2}}$ & $-7 \cdot 11^{3}$ & $\frac{2^{16} \cdot 3^{3} \cdot 17}{5^{5}}$ \\
$-2^{9} \cdot 3^{3}$ & $\frac{3^{3} \cdot 227^{3} \cdot 1367^{3}}{2^{12} \cdot 7^{12}}$ & $\frac{-3^{2} \cdot 5^{6}}{2^{3}}$ & $\frac{3^{3} \cdot 5 \cdot 7^{5}}{2^{7}}$ \\
$\frac{-29^{3} \cdot 41^{3}}{2^{15}}$ & $\frac{-3^{3} \cdot 37^{3} \cdot 47}{2^{10}}$ & $-2^{10} \cdot 3^{4}$ & $\frac{-2^{12} \cdot 5^{3} \cdot 11 \cdot 13^{4}}{3^{13}}$ \\
$2^{12} \cdot 3^{3}$ & $\frac{-2^{6} \cdot 5^{4} \cdot 19^{3} \cdot 263 \cdot 5113^{3}}{7^{20}}$ & $-2^{10} \cdot 3^{3} \cdot 5$ & $-2^{4} \cdot 3^{2} \cdot 13^{3}$ \\
$\frac{-3^{3} \cdot 5^{6} \cdot 199^{3} \cdot 809^{3} \cdot 5059^{3}}{61^{15}}$ & $2^{6} \cdot 3^{8}$ & $\frac{-5^{2} \cdot 41^{3}}{2^{2}}$ & $\frac{2^{10} \cdot 3^{2} \cdot 79^{3}}{5^{5}}$ \\
$\frac{2^{6} \cdot 223^{3} \cdot 277^{3}}{3^{12} \cdot 5^{6}}$ & $\frac{-3^{3} \cdot 13 \cdot 479^{3}}{2^{14}}$ & $\frac{2^{6} \cdot 109^{3}}{3^{3}}$ & $\frac{3^{3} \cdot 5^{6} \cdot 13^{3} \cdot 23^{3} \cdot 41^{3}}{2^{16} \cdot 31^{4}}$ \\[3pt]
$\frac{-2^{21} \cdot 3^{3} \cdot 5^{3} \cdot 7 \cdot 13^{3} \cdot 23^{3} \cdot 41^{3} \cdot 179^{3} \cdot 409^{3}}{79^{16}}$ & $257^{3}$ & $-2^{6} \cdot 3^{3} \cdot 23^{3}$ & $-11 \cdot 131^{3}$ \\
$\frac{-5^{2} \cdot 241^{3}}{2^{3}}$ & $-2^{12} \cdot 3^{7} \cdot 5$ & $\frac{2^{9} \cdot 3^{3} \cdot 13^{3} \cdot 167^{3} \cdot 1151^{3}}{11^{15}}$ & $\frac{-3^{3} \cdot 5^{4} \cdot 11^{3} \cdot 17^{3}}{2^{10}}$ \\
$\frac{-17^{2} \cdot 101^{3}}{2}$ & $\frac{2^{18} \cdot 3^{3} \cdot 13^{4} \cdot 127^{3} \cdot 139^{3} \cdot 157^{3} \cdot 283^{3} \cdot 929}{5^{13} \cdot 61^{13}}$ & $2^{15} \cdot 7^{5}$ & $-3^{3} \cdot 5^{4} \cdot 11 \cdot 19^{3}$ \\
$\frac{-3^{3} \cdot 5^{3} \cdot 383^{3}}{2^{7}}$ & $\frac{2^{12} \cdot 3^{3} \cdot 5^{7} \cdot 29^{3}}{7^{5}}$ & $2^{3} \cdot 3^{8} \cdot 67^{3}$ & $-2^{6} \cdot 719^{3}$ \\[3pt]
$2^{12} \cdot 211^{3}$ & $2^{4} \cdot 3^{2} \cdot 5^{7} \cdot 23^{3}$ & $2^{6} \cdot 3^{10} \cdot 17 \cdot 73^{3}$ & \\
$-2^{2} \cdot 3^{7} \cdot 5^{3} \cdot 439^{3}$ &
$\frac{3^{3} \cdot 11^{3} \cdot 107^{3} \cdot 521^{3}}{2^{15}}$ & $\frac{-3^{3} \cdot 29^{3} \cdot 107^{3} \cdot 199^{3}}{2^{15}}$ &\\
$-7 \cdot 137^{3} \cdot 2083^{3}$ & 
$2^{6} \cdot 11^{3} \cdot 23^{3} \cdot 149^{3} \cdot 269^{3}$ &
$2^{9} \cdot 17^{6} \cdot 19^{3} \cdot 29^{3} \cdot 149^{3}$ &
\end{tabular}
\caption{The 77 exceptional $j$-invariants in $J_{\rm exc}$}
\label{table:exc}
\end{table}

\begin{conjecture}[Zywina]\label{conj: Zywina}
    Let $H \in A_{\rm finite}$. If $x\in X_H(\Q)$  is a non-special point, then $j(x)\in J_{\rm exc}.$ 
\end{conjecture}

\begin{remark}
Proving Zywina's conjecture, under the assumption of Conjecture \ref{ESUC}, amounts to determining the rational points on a finite set of curves of genus $\geq 2$. 
\end{remark}

Our results will depend on both Conjectures \ref{ESUC} and \ref{conj: Zywina}. For easier reference, we combine them into one conjecture. 
\begin{conjecture}[Galois images conjecture] 
\label{conj: GIC}
     Let $H\leq \GL_2(\Zhat)$ be an agreeable subgroup such that $\#X_H(\Q)<\infty$, and let $x\in X_H(\Q)$ be a non-special point. Then $j(x)\in J_{\rm exc}$.
\end{conjecture}

Conjecture \ref{conj: GIC} is equivalent to Conjecture \ref{ESUC} and Conjecture \ref{conj: Zywina} combined. 
We first show these conjectures imply Conjecture \ref{conj: GIC}, then show the converse.
Conjecture \ref{ESUC} implies that if \(H\) is agreeable and \(X_H\) has a non-special rational point \(x\) then the level of $H$ is divisible only by primes $\ell \in \{2,3,5,7,11,13, 17, 37\}$. If additionally \(\#X_H(\Q)<\infty\), then \(H\in A_{\rm finite}\). Conjecture \ref{conj: Zywina} then implies that \(j(x)\in J_{\rm exc}\).  

To see that Conjecture \ref{conj: GIC} implies Conjecture \ref{ESUC}, we prove the contrapositive. 
Assume that there exists a non-CM elliptic curve $E$ over $\Q$ with non-surjective mod $\ell$ image for some $\ell>13$ and $j(E) \notin \{-17^2\cdot 101^3/2,  -17\cdot 373^3/2^{17},  -7\cdot 11^3, -7\cdot 137^3\cdot 2083^3\}$. This would induce a rational point on an agreeable $X_H$ where $H = G_E^{\ag}$. 
Furthermore, $\#X_H(\Q)<\infty$ because $H$ is contained in a maximal subgroup $M$ with $\ell \mid N(M)$, so either $M$ is the normalizer of a split or non-split Cartan with level $>13$, an exceptional group of level $>13$, or Borel of level $>13$ and  $\neq 17, 37$ (because of the restrictions on $j(E)$).
Thus we have $H \notin A_{\rm finite}$.
This  implies that $j(E)\notin J_{\rm exc}$.

The set $A_{\infty}$ has been explicitly computed in \cite{zywina_explicit_images} and it  contains $454$ groups. For exactly $315$ of these groups, the modular curve $X_H$ is isomorphic to $\PP^1$, and for the other $139$ groups, the modular curve $X_H$ is isomorphic to a positive rank elliptic curve.

\section{Algorithm for checking twist isolatedness}
\label{sec:algorithm}

From what has been proven so far, it is not hard to see that there exists an algorithm to determine whether a given $\mathcal X_G$ is twist parameterized. Note that to show that a $j$-invariant $j$ is twist isolated, it is not enough to check that $\mathcal X_{G_j}$ is twist isolated. It is possible that $\mathcal X_{G_j}$ is twist parameterized over some $\mathcal X_H$, while for some $G_j\leq G\leq \GL_2(\Zhat)$, the curve $\mathcal X_{G}$ is twist isolated. 

\begin{theorem}
    There exist algorithms that
    \begin{itemize}
        \item[a)] for a given open subgroup $G\leq \GL_2(\Zhat)$ check whether $\mathcal X_G$ is twist parameterized. 
        \item[b)] for a given $j$-invariant $j_0\in \Q$ test whether $j_0$ is a twist-parameterized $j$-invariant. 
    \end{itemize}
\end{theorem}
\begin{proof}
    To check whether $\mathcal X_G$ is twist parameterized, by Proposition \ref{prop:tp_group_condition} (1), (2) and (8), we need to check whether there exists an agreeable overgroup $H$ of $G^{\ag}$, with $[H,H]=[G^{\ag},G^{\ag}]$ and $\mathcal X_H(\Q)=\infty$. This can be done by explicitly checking whether there exists an $H$ from $A_{\inf}$ from Theorem \ref{thm:zywina-finite-Q} satisfying the condition $[H,H]=[G^{\ag},G^{\ag}]$.

    Now let $j_0 \in \Q$ be a $j$-invariant. By definition, $j_0$ is twist parameterized if and only if $\mathcal X_{G_{j_0}}$ is twist parameterized and all images of $\mathcal X_{G_{j_0}}$ under modular maps are twist parameterized. This is again a finite computation, and by applying Proposition \ref{prop:tp_group_condition} (8) we can see that it is enough to check whether all modular curves $\mathcal X_G$ for $G_{j_0}^{\ag} \leq G \leq \GL_2(\Zhat)$ are twist parameterized.
    
    is twist parameterized and all images of $\mathcal X_{G_{j_0}^{\ag}}$ under modular maps are twist parameterized.    
\end{proof}

We have implemented these algorithms, and they work very efficiently in practice. The key observation that makes this algorithm usable in practice is Proposition \ref{prop:tp_group_condition} (8) which allows one to work with the agreeable closure $G^{\ag}$ of a group $G$ instead of $G$.

\subsection{Determining the twist-isolated \texorpdfstring{$j$}{j}-invariants} \label{sec:twist-isolated_j}

 Any twist-isolated $j$-invariant will necessarily be exceptional (as defined in Remark \ref{rem:exc_j}). However, being twist isolated is a stronger property than being exceptional: an exceptional point can be twist parameterized. As an example, consider the $j$-invariant $\frac{-631595585199146625}{218340105584896}$. It is exceptional, but twist parameterized on all $\mathcal X_{G}$ with $G_j\leq G\leq\GL_2(\Zhat)$. On the other hand, if $x \in \mathcal X_{G_j}(\Q)$  is twist isolated then $\mathcal X_{G_j^{\ag}}(\Q)$ is necessarily finite (otherwise, the point on $\mathcal X_{G_j}$ would be twist parameterized over $\mathcal X_{G_j^{\ag}}$).

 Hence, assuming Conjecture \ref{conj: GIC}, to find all the twist-isolated $j$-invariants we need to determine which $j\in J_{\rm exc}$ are twist isolated.
 We \gitlink{twist_isolated_points.m}{compute} that out of these 77 $j$-invariants, 41 are twist isolated, while the rest are twist parameterized.

Table \ref{table:twist isolated} lists all the twist-isolated $j$-invariants and the genera of their agreeable closures.

\begin{longtable}{llllll}
\hline
$j$-invariant &  LMFDB link &  Agreeable curve &Agree. gen. & Adelic gen. & $i$   \\
\hline
$-2^{10}\cdot 3^{3}\cdot 5$ & \href{https://www.lmfdb.org/EllipticCurve/Q/1800.o1}{1800.o1} & \href{https://beta.lmfdb.org/ModularCurve/Q/12.24.1.h.1/}{12.24.1.h.1} &  $1$ & $1$ & 139 \\ 
$-2^{10}\cdot 3^{4}$ & \href{https://www.lmfdb.org/EllipticCurve/Q/1944/c/1}{1944.c1} & \href{https://beta.lmfdb.org/ModularCurve/Q/12.24.1.g.1/}{12.24.1.g.1}& $1$ & $1$ & 140 \\ 
$-\frac{3^{3}\cdot 11^{3}}{2^{2}}$ & \href{https://www.lmfdb.org/EllipticCurve/Q/162.a1}{162.a1} & \href{https://beta.lmfdb.org/ModularCurve/Q/12.32.1.b.1/}{12.32.1.b.1} & $1$ & $1$ & 141 \\ 
$\frac{3^{2}\cdot 23^{3}}{2^{6}}$ & \href{https://www.lmfdb.org/EllipticCurve/Q/162.a2}{162.a2} & \href{https://beta.lmfdb.org/ModularCurve/Q/12.32.1.b.1/}{12.32.1.b.1} & $1$ & $1$ &141  \\
$\frac{3^{3}\cdot 5^{3}}{2^{6}}$ & \href{https://www.lmfdb.org/EllipticCurve/Q/6534.g1}{6534.g1} & \href{https://beta.lmfdb.org/ModularCurve/Q/12.48.1.q.1/}{12.48.1.q.1} & $1$ & $5$ & 142 \\ 
$-\frac{5^{2}\cdot 41^{3}}{2^{2}}$ & \href{https://www.lmfdb.org/EllipticCurve/Q/14450.y1}{14450.y1} & \href{https://beta.lmfdb.org/ModularCurve/Q/20.24.1.g.1/}{20.24.1.g.1} & $1$ & $5$ & 143 \\ 
$\frac{5\cdot 59^{3}}{2^{10}}$ & \href{https://www.lmfdb.org/EllipticCurve/Q/14450.j2}{14450.j2} & \href{https://beta.lmfdb.org/ModularCurve/Q/20.24.1.g.1/}{20.24.1.g.1} & $1$ & $5$ & 143 \\ 
$ -\frac{17^{2}\cdot 101^{3}}{2}$ & \href{https://www.lmfdb.org/EllipticCurve/Q/14450.b2}{14450.b2} & \href{https://beta.lmfdb.org/ModularCurve/Q/17.18.1.a.1/}{$X_0(17)$} &$1$ & $17$ & 144 \\ 
$-\frac{17\cdot 373^{3}}{2^{17}}$ & \href{https://www.lmfdb.org/EllipticCurve/Q/14450.o2}{14450.o2} & \href{https://beta.lmfdb.org/ModularCurve/Q/17.18.1.a.1/}{$X_0(17)$} &$1$ & $17$ & 144\\ 
$-2^{4}\cdot 3^{2}\cdot 13^{3}$ & \href{https://www.lmfdb.org/EllipticCurve/Q/324.b1}{324.b1} & \href{https://beta.lmfdb.org/ModularCurve/Q/60.40.2.a.1/}{60.40.2.a.1} & $2$ & $4$ & 145 \\ 
$2^{4}\cdot 3^{3}$ & \href{https://www.lmfdb.org/EllipticCurve/Q/324.b2}{324.b2} & \href{https://beta.lmfdb.org/ModularCurve/Q/60.40.2.a.1/}{60.40.2.a.1} & $2$ & $4$ & 145 \\ 
$-\frac{2^{21}\cdot 3^{3}\cdot 5^{3}\cdot 7\cdot 13^{3}\cdot 23^{3}\cdot 41^{3}\cdot 179^{3}\cdot 409^{3}}{79^{16}}$ & not in lmfdb & \href{https://beta.lmfdb.org/ModularCurve/Q/16.64.2.a.1/}{$X_{\mathrm{ns}}^+(16)$} & $2$ & $7$ & 146 \\ 
$-\frac{2^{18}\cdot 3\cdot 5^{3}\cdot 13^{3}\cdot 41^{3}\cdot 107^{3}}{17^{16}}$ & not in lmfdb &\href{https://beta.lmfdb.org/ModularCurve/Q/16.64.2.a.1/}{$X_{\mathrm{ns}}^+(16)$} & $2$ & $7$ & 146 \\ 
$-11\cdot 131^{3}$ & \href{https://www.lmfdb.org/EllipticCurve/Q/121.a2}{121.a2} & \href{https://beta.lmfdb.org/ModularCurve/Q/44.24.2.a.1/}{44.24.2.a.1} & $2$ & $16$ & 147 \\ 
$-11^{2}$ & \href{https://www.lmfdb.org/EllipticCurve/Q/121.c2}{121.c2} & \href{https://beta.lmfdb.org/ModularCurve/Q/44.24.2.a.1/}{44.24.2.a.1}  & $2$ & $16$ & 147 \\ 
$\frac{2^{12}\cdot 3^{3}\cdot 5^{7} 29^3}{7^{5}}$ & \href{https://www.lmfdb.org/EllipticCurve/Q/21175/bm/1}{21175.bm1} & \href{https://beta.lmfdb.org/ModularCurve/Q/25.75.2.a.1/}{25.75.2.a.1} & $2$ & $19$ & 148 \\ 
$-7\cdot 137^{3}\cdot 2083^{3}$ & \href{https://www.lmfdb.org/EllipticCurve/Q/1225.b1}{1225.b1} & \href{https://beta.lmfdb.org/ModularCurve/Q/37.38.2.a.1/}{$X_0(37)$} & $2$ & $97$ & 149 \\ 
$-7\cdot 11^{3}$ & \href{https://www.lmfdb.org/EllipticCurve/Q/1225.b2}{1225.b2} & \href{https://beta.lmfdb.org/ModularCurve/Q/37.38.2.a.1}{$X_0(37)$}& $2$ & $97$ & 149 \\ 
$-\frac{5^{2}\cdot 241^{3}}{2^{3}}$ & \href{https://www.lmfdb.org/EllipticCurve/Q/50.a1}{50.a1} & \href{https://beta.lmfdb.org/ModularCurve/Q/120.48.3.c.1/}{120.48.3.c.1} &$3$ & $9$ & 150 \\ 
$-\frac{5\cdot 29^{3}}{2^{5}}$ & \href{https://www.lmfdb.org/EllipticCurve/Q/50.b3}{50.b3} &\href{https://beta.lmfdb.org/ModularCurve/Q/120.48.3.c.1/}{120.48.3.c.1} &$3$ & $9$  & 150\\ 
$-\frac{5^{2}}{2}$ & \href{https://www.lmfdb.org/EllipticCurve/Q/50.a3}{50.a3} &  \href{https://beta.lmfdb.org/ModularCurve/Q/120.48.3.c.1/}{120.48.3.c.1}&$3$ & $9$ & 150 \\ 
$\frac{5\cdot 211^{3}}{2^{15}}$ & \href{https://www.lmfdb.org/EllipticCurve/Q/50.b4}{50.b4} & \href{https://beta.lmfdb.org/ModularCurve/Q/120.48.3.c.1/}{120.48.3.c.1} & $3$ & $9$ & 150 \\ 
$-\frac{2^{12}\cdot 5^{3}\cdot 11\cdot 13^{4}}{3^{13}}$ & \href{https://www.lmfdb.org/EllipticCurve/Q/61347/bb/1}{61347.bb1} & \href{https://beta.lmfdb.org/ModularCurve/Q/13.91.3.a.1/}{$X_{S_4}(13)$} & $3$ & $10$ & 151\\ 
$\frac{2^{4}\cdot 5\cdot 13^{4}\cdot 17^{3}}{3^{13}}$ & \href{https://www.lmfdb.org/EllipticCurve/Q/50700/z/1}{50700.z1} & \href{https://beta.lmfdb.org/ModularCurve/Q/13.91.3.a.1/}{$X_{S_4}(13)$} & $3$ & $10$ & 151 \\ 
$\frac{2^{18}\cdot 3^{3}\cdot 13^{4}\cdot 127^{3}\cdot 139^{3}\cdot 157^{3}\cdot 283^{3}\cdot 929}{5^{13}\cdot 61^{13}}$ & not in lmfdb & \href{https://beta.lmfdb.org/ModularCurve/Q/13.91.3.a.1/}{$X_{S_4}(13)$}& $3$ & $10$ & 151 \\
$-\frac{3^{3}\cdot 5^{4}\cdot 11^{3}\cdot 17^{3}}{2^{10}}$ & not in lmfdb & \href{https://beta.lmfdb.org/ModularCurve/Q/20.60.3.w.1/}{20.60.3.w.1} &$3$  & $15$ & 152 \\ 
$-\frac{29^{3}\cdot 41^{3}}{2^{15}}$ & \href{https://www.lmfdb.org/EllipticCurve/Q/338.d1}{338.d1} & \href{https://beta.lmfdb.org/ModularCurve/Q/120.72.3.gqp.1}{120.72.3.gqp.1}  &$3$ & $17$ & 153 \\ 
$\frac{11^{3}}{2^{3}}$ & \href{https://www.lmfdb.org/EllipticCurve/Q/338.d2}{338.d2} & \href{https://beta.lmfdb.org/ModularCurve/Q/120.72.3.gqp.1/}{120.72.3.gqp.1} &$3$ & $17$  & 153 \\ 
$-\frac{3^{3}\cdot 13\cdot 479^{3}}{2^{14}}$ & \href{https://www.lmfdb.org/EllipticCurve/Q/338.c1}{338.c1} & \href{https://beta.lmfdb.org/ModularCurve/Q/28.64.3.b.1/}{28.64.3.b.1} &$3$ & $21$ & 154 \\ 
$\frac{3^{3}\cdot 13}{2^{2}}$ & \href{https://www.lmfdb.org/EllipticCurve/Q/338.c2}{338.c2} & \href{https://beta.lmfdb.org/ModularCurve/Q/28.64.3.b.1/}{28.64.3.b.1} &$3$ & $21$ & 154 \\ 
$-\frac{3^{3}\cdot 5^{3}\cdot 383^{3}}{2^{7}}$ & \href{https://www.lmfdb.org/EllipticCurve/Q/162.b1}{162.b1} & \href{https://beta.lmfdb.org/ModularCurve/Q/168.64.3.a.1/}{168.64.3.a.1}&$3$ & $21$ & 155 \\
$-\frac{3^{2}\cdot 5^{6}}{2^{3}}$ & \href{https://www.lmfdb.org/EllipticCurve/Q/162.b3}{162.b3} & \href{https://beta.lmfdb.org/ModularCurve/Q/168.64.3.a.1/}{168.64.3.a.1} &$3$ & $21$ & 155 \\ 
$-\frac{3^{2}\cdot 5^{3}\cdot 101^{3}}{2^{21}}$ & \href{https://www.lmfdb.org/EllipticCurve/Q/162.b2}{162.b2} & \href{https://beta.lmfdb.org/ModularCurve/Q/168.64.3.a.1/}{168.64.3.a.1} &$3$ & $21$ & 155 \\
$\frac{3^{3}\cdot 5^{3}}{2}$ & \href{https://www.lmfdb.org/EllipticCurve/Q/162.b4}{162.b4} &\href{https://beta.lmfdb.org/ModularCurve/Q/168.64.3.a.1/}{168.64.3.a.1} &$3$ & $21$ & 155 \\ 
$2^4 \cdot 3^2\cdot 5^7 \cdot 23^3$ & \href{https://beta.lmfdb.org/EllipticCurve/Q/396900e1/}{396900.e1} & \href{https://beta.lmfdb.org/ModularCurve/Q/100.100.4.d.1/}{100.100.4.d.1} & 4 & 9 & 156\\
$\frac{2^{16}\cdot 3^{3}\cdot 17}{5^{5}}$ & \href{https://www.lmfdb.org/EllipticCurve/Q/5780.c1}{5780.c1} &\href{https://beta.lmfdb.org/ModularCurve/Q/20.120.6.a.1/}{20.120.6.a.1} &$6$ & $6$ & 157 \\ 
$\frac{2^{10}\cdot 3^{2}\cdot 79^{3}}{5^{5}}$ & \href{https://www.lmfdb.org/EllipticCurve/Q/3240.c1}{3240.c1} & \href{https://beta.lmfdb.org/ModularCurve/Q/20.120.6.a.1/}{20.120.6.a.1 }&$6$ & $6$ & 157 \\ 
$-2^{2}\cdot 3^{7}\cdot 5^{3}\cdot 439^{3}$ & not in lmfdb & \href{https://beta.lmfdb.org/ModularCurve/Q/36.108.6.g.1/}{36.108.6.g.1} &$6$ & $14$ & 158 \\ 
$-3^{3}\cdot 5^{4}\cdot 11\cdot 19^{3}$ & \href{https://www.lmfdb.org/EllipticCurve/Q/27225/n/1}{27225.n1} & \href{https://beta.lmfdb.org/ModularCurve/Q/30.120.7.e.1}{30.120.7.e.1}&$7$ & $7$ & 159 \\ 
$-2^{6}\cdot 719^{3}$ & \href{https://www.lmfdb.org/EllipticCurve/Q/43264/f/1}{43264.f1} & 360.108.7.?.?\footnote{The agreeable curves for $j=-2^6 \cdot 719^3$ and $j=2^6$ are the same curve. This curve has level 360, index 108, genus 7, and Cummins and Pauli label 90C7. However, it is not in the LMFDB at the time of writing.\label{foot:notinlmfdb}} &$7$ & $33$ & 160\\ 
$2^{6}$ & \href{https://www.lmfdb.org/EllipticCurve/Q/43264.f2}{43264.f2} & 360.108.7.?.?\footref{foot:notinlmfdb} &$7$ & $33$ & 160 \\ 
\hline
\caption{Table of twist-isolated $j$-invariants. For each $i$, the group $M_i$ is the one corresponding to the agreeable modular curve, while $K_i$ is the adelic Galois image $G_E$ for any choice of $E$ in the second column.}
\label{table:twist isolated}
\end{longtable}

We \gitlink{twist_isolated_j_families.m}{compute} that the twist-isolated points live in 22 families of groups $\mathcal{F}(M_i, K_i)$ and no fewer. To check that no two families $\mathcal{F}(M_i, K_i)$ and $\mathcal{F}(M_j, K_j)$, for some $139\leq i<j\leq 160$, can be merged, we check that either $K_i \cap \SL_2(\Zhat) \neq K_j \cap \SL_2(\Zhat)$, or that there does not exist a $M_k$, containing, up to conjugacy, both $M_i$ and $M_j$, such that $M_k/K_i$ and $M_k/K_j$ are both abelian. 

\section{Proof of Theorem \ref{thm 1+2}} \label{sec: proof}

Let $E/\Q$ be a non-CM elliptic curve.
There are two cases: $j(E)$ is twist isolated and $j(E)$ is twist parameterized.

First, suppose $j(E)$ is twist isolated. 
Let $S_{\rm finite}$ be the set of $j$-invariants coming from the rational points on $X_H(\Q)$ where $H \in A_{\rm finite}$.
We claim that any twist-isolated non-CM $j$-invariant is contained in $S_{\rm finite}$. 
Let $j_0$ be a twist-isolated $j$-invariant. 
Then we can find a subgroup $G\leq \GL_2(\Zhat)$ and a point $x\in X_G(\Q)$ with $j(x)=j_0$ such that $x$ is twist isolated.  
By Proposition \ref{prop:tp_group_condition} (5), we have $\# X_{G^{\ag}}(\Q)<\infty$. 
First consider the case where there exists a prime $\ell>13$ dividing the level of $G^{\ag}$. 
Then for any elliptic curve $E$ over $\Q$ with $j(E)=j_0$ we have $\rho_{E,\ell^\infty}(\Gal_\Q)\subsetneq \GL_2(\Z_\ell)$ by \cite[Lemma 8.1. (i)]{zywina_explicit_images}. Since $\ell>13$, it follows by \cite[Chapter IV, Lemma 3]{serre_elladic} that $\rho_{E,\ell}(\Gal_\Q)\subsetneq \GL_2(\torz{\ell})$. Conjecture \ref{ESUC} implies that  $j \in \{-17^2\cdot 101^3/2,  -17\cdot 373^3/2^{17},  -7\cdot 11^3, -7\cdot 137^3\cdot 2083^3\}\subset S_{\rm finite}$, which deals with this case.
In the case when $G^{\ag}$ is divisible only by the primes $\leq 13$, then $\# X_{G^{\ag}}(\Q)<\infty$ together with Theorem \ref{thm:zywina-finite-Q} (c) implies that $j_0 \in S_{\rm finite}$.

By the results of Section \ref{sec:twist-isolated_j} there are 41 twist-isolated $j$-invariants that lie in $22$ twist families. For $i\in \{139, \ldots, 160\}$ let $M_i=G_{j_i}^{\ag}$ and $K_i=G_{E'}$, where $E'$ is the elliptic curve with $j(E')=j_i$ such that $E'$ is the quadratic twist with minimal conductor such that $j(E')=j_i$. (Note that for each $i$ in Table \ref{table:twist isolated} there can be multiple distinct $j_i$ and thus multiple ways to define $K_i$, but the algorithm in Section \ref{sec:algorithm} verifies that the family $\mathcal{F}(M_i, K_i)$ is well-defined, irrespective of this choice.)
By definition, the adelic image $G_E$  of $E$ belongs to the family of groups $\mathcal F(G_{j(E)}^{\ag}, G_{E^d})$ for any quadratic twist $E^d$ of $E$. 
Conjecture \ref{conj: GIC} implies that, since $\#X_{G_j(E)^{\ag}}(\Q)<\infty$ we have  $j(E) \in J_{\rm exc}$. Since $j(E)$ is twist isolated, it must be one of the 41 twist-isolated $j$-invariants in Table \ref{table:twist isolated}.
Hence,  $G_E\in \mathcal F(M_i, K_i)$, for some $i\in \{139, \ldots, 160\}$.

Now, let $E/\Q$ be an elliptic curve such that $j(E)$ is twist parameterized. Then we have $\#X_{G^{\ag}}(\Q)=\infty$ and by Lemma \ref{lem: G_Eappears} we have $G_E\in \mathcal F(G_E^{\ag},[G_E^{\ag}, G_E^{\ag}])$. Hence $G_E^{\ag}$ is conjugate to some $G \in A_{\inf}$, the finite set  of agreeable subgroups of $\GL_2(\Zhat)$ from Theorem \ref{thm:zywina-finite-Q} with $\#X_{G}(\Q) = \infty$. We see that the groups $G\in A_{\rm inf}$ are ``bases" for families of groups $\mathcal F(G,H)$ containing all the $G_E$ with twist-parameterized $j(E)$. Let $\mathcal{K}$ be the set of all $[G,G]$ with $G\in A_{\rm inf}$. 
 
 We can construct a smaller set $\mathcal A_{\rm base}\subset  A_{\inf}$ of agreeable groups with the property that for any $G\in A_{\rm inf}$ we have $\mathcal F(G,[G,G])\subset \mathcal F(B,[G,G])$ for some $B\in \mathcal A_{\rm base}$ (or in other words, $G$ is conjugate to a subgroup of $B$, and $[G,G]=[B,B]$). In particular, every $K\in \mathcal K$ is in $\mathcal F(M,K)$ for some $M\in \mathcal A_{\rm base}$. The values $M_i, K_i$ in the statement of the theorem are such that every such $\mathcal F(M,K)$ is contained in some $\mathcal F(M_i,K_i)$ and that $\det K_i=\Zhat$, i.e., such that $K_i\in \mathcal F(M_i,K_i)$ (note that, perhaps counterintuitively, this is not required by definition). This minimal set of families containing all the $G_E$ for $j(E)$ twist parameterized has likewise been computed by Zywina \cite[Remark 12.1]{zywina_explicit_images}, and contains 138 families; it consists of the 138 values of $M_i$ from the statement of the theorem. 
 Hence $G_E \in \mathcal F(M_i,[G_E^{\ag},G_E^{\ag}])$ for some $i\in \{1, \ldots, 138\}$, so $G_E \in \mathcal F(M_i,K_i)$ for some $i \in \{1, \ldots ,138\}$. 
 
 Thus, we have proved that 
 $G_E$ is conjugate to $K_i^\chi$ for some character $\chi:\Gal_\Q \rightarrow M_i/K_i = T_i$. Let $\overline{\chi} \colon \Gal_\Q \to \overline{T}_i = M_i/{\pm{K_i}}$ be the corresponding character of $\overline{T}_i$. By Lemma \ref{lem: abelian twist lifts}, $\overline{\chi}$ is the character $\chi_x \colon \Gal_\Q \to \overline{T}_i$ describing the action of $\overline{T}_i$ on $\pi_i^{-1}(x)$, where $x \in X_{M_i}(\Q)$ is the rational point in the base of the twist family below a rational point on $X_{G_E}(\Q)$ with $j$-invariant $j(E)$. This proves the theorem. 
 \qed

\begin{remark}
    If we assume only Conjecture \ref{SUC} instead of Conjecture \ref{conj: GIC}, there would still exist a finite set of triples $(K_i, M_i, T_i)$ for which Theorem \ref{thm 1+2} holds, but it would not be possible to make this set explicit. If Conjecture \ref{conj: Zywina} were to turn out to be false, then it might be necessary to add further triples $(K_i, M_i, T_i)$, all satisfying $\#X_{M_i}(\Q) < \infty$.

\end{remark}

In Theorem \ref{thm 1+2} we have 160 pairs $K_i \subset M_i$ among whose twists are all adelic images of Galois.
The following general lemma shows that there is always a twist $\mathcal{X}_{K_i}^\chi$ such that $\mathcal{X}_{K_i}^\chi \to \mathcal{X}_{M_i}$ is the agreeable closure morphism, i.e. $M_i=(K_i^\chi)^{\ag}$. 
\begin{lemma}\label{lem: ag_closure}
    Suppose $\mathcal{X}_H \to \mathcal{X}_G$ is a modular abelian map with group $T  = \Aut(\mathcal{X}_H/\mathcal{X}_G)$, and suppose that $\mathcal{X}_G$ is agreeable. Then there exists a character $\chi \colon \Gal_\Q\to T$ such that $G = (H^\chi)^\mathrm{ag}$.
\end{lemma}
\begin{proof}
    If $T$ is trivial, then $H = G$ is its own agreeable closure.
    When $T$ is non-trivial, as we vary over the infinitely many $\chi \colon \Gal_\Q \to T$, there are only finitely many agreeable modular curves $\mathcal{X}_{G'}$ that appear as intermediate covers 
    \[\mathcal{X}^\chi_H \to \mathcal{X}_{G'} \stackrel{f}{\to} \mathcal{X}_G\]
    with $\deg(f) > 1$, by Proposition \ref{prop: bounded index}. However, there exist characters $\chi$ which avoid these finitely many agreeable intermediate covers. Indeed, it is enough to show that given finitely many number fields $K_1,\ldots, K_n$, there exists a $T$-extension $K$ of $\Q$ not containing any of the $K_i$. This follows easily from the Kronecker--Weber theorem.  It follows that there exists a twist $\mathcal{X}^\chi_H$ such that all of its intermediate covers $\mathcal{X}_{G'}$ as above are {\it not} agreeable. 
    For such a $\chi$, the map $\mathcal{X}_H^\chi \to \mathcal{X}_G$ must be the agreeable quotient by Lemma \ref{lem: agreeable quotient properties}(6). 
    \end{proof}

\begin{remark}
    Lemma \ref{lem: ag_closure} shows that Theorem \ref{thm 1+2} could also have been formulated by specifying only certain groups \(K_i\) for which $M_i = K_i^\mathrm{ag}$ is the agreeable closure of $K_i$ (as it is enough to replace each group $K_i$ in Table \ref{table:bigtableofgroups} with some twist $K_i^\chi$), giving a statement more in line with Conjecture \ref{conj: mazur program B revised}.
This reformulation would not alter the mathematical content of the theorem, but it would instead require giving the groups \(K_i^\chi\) of larger level that likely do not appear in the LMFDB. 
\end{remark}

\section{Points explained by geometry}\label{sec: geometric points}

Having (conditionally) classified all adelic images $G_E$ of all elliptic curves $E$ over $\Q$, our next goal is to (conditionally) show that all rational points on all modular curves are ``explained by geometry''. In this section, we formalize this notion. As the focus in this section and the next is on rational points, we only consider the modular curves $X_H$ and not the stacks $\mathcal{X}_H$ (see Remark \ref{rem: stacks}).  

In the following definition we distinguish between a closed point $y \in X_H$ and a rational point $y \in X_H(\Q)$, which is by definition a closed point of degree $1$.

\begin{definition}[Explained points]
    The set of {\it explained points} (or \emph{points explained by geometry}) is the smallest set of closed points $y \in X_H$, as $H$ varies over all open subgroups of $\GL_2(\widehat{\Z})$, satisfying the following axioms: 
\begin{enumerate}
    \item\label{special} (Special) If $y \in X_H$ is a special closed point, then it is explained.
    \item\label{pushforward}(Push-forward) Let $\pi \colon X_{H'} \to X_H$ be a non-constant morphism and let $y \in X_{H'}(\Q)$ be an explained rational point. If either
    \begin{itemize}
        \item $\pi$ is modular, or
        \item there exists an explained point $y' \in X_{H'}(\Q)$ such that $\pi(y')$ is explained, or
        \item $g(X_H) > 1$,
    \end{itemize}
    then $\pi(y)$ is explained.
    \item\label{twistlift}(Abelian lift) Let $X_H \stackrel{\pi}{\to} X_{H'}$ be a modular abelian map and let $y \in X_{H}(\Q)$. If $\pi(y)$ is explained, then $y$ is explained. 
    \item\label{collinearity} (Collinearity) Let $y \in X_H(\Q)$, let $x_1,\ldots, x_n \in X_H$ be explained closed points, and let  
    \[D = \begin{cases} \text{an anticanonical divisor} & \text{if } g(X_H) = 0 \\
         \text{$3x$, for some explained point $x \in X_H(\Q)$} &  \text{if $g(X_H) = 1$}\\
        \text{a canonical divisor} & \text{if $g(X_H) > 1$.} 
    \end{cases}
    \]
    If $D$ is linearly equivalent to $y + \sum_{i = 1}^n x_i$, then $y$ is explained.
    \item\label{fibersplice} (Fiber splicing) Suppose $g(X_H) > 1$ and for $i \in \{1,2\}$, let $E_i$ be an elliptic curve over $\Q$ with $\Hom(X_H,E_i) = \Z f_i$ for some map $f_i \colon X_H \to E_i$.\footnote{Here, $\Hom(X_H,E_i)$ denotes the morphisms $X_H \to E_i$ factoring through the canonical Abel--Jacobi map $X_H \to \mathrm{Jac}(X_H)$ sending $P$ to $(2g-2)P - K$, where $K$ is a canonical divisor. Note that $\Hom(X_H,E_i)$ is rank $1$ if and only if $\Hom(\mathrm{Jac}(X_H),E_i)$ is rank $1$ if and only if $E_i$ has multiplicity $1$ as an isogeny factor of the Jacobian.}
    We say a point $x_i \in E_i(\Q)$ is explained if $x_i = f_i(x)$ for some explained point  $x \in X_H(\Q)$ or if $E_i$ is isomorphic to a modular curve $X_G$ and  all of $X_G(\Q)$ is explained. Let $x_1 \in E_1(\Q)$ and $x_2 \in E_2(\Q)$ be explained points, and suppose
    \begin{itemize}
        \item $f_1^{-1}(x_1) \cap f_2^{-1}(x_2) = \{y\}$ for some $y \in X_H(\Q)$, and 
        \item either $\#E_1(\Q) < \infty$ or $x_1 = f_1(x)$ for some explained point $x \in X_H(\Q)$.
    \end{itemize}
    Then $y$ is explained.  
\end{enumerate} 
\end{definition}

Intuitively, the explained rational points on modular curves are those that arise from special points via natural geometric operations that preserve rationality. The following lemmas and their proofs will help clarify the geometric significance of axiom (\ref{collinearity}).

\begin{lemma}\label{lem: genus 0 collinearity}
    If $g(X_H) = 0$ and $x_1 \in X_H(\Q)$ is explained, then all of $X_H(\Q)$ is explained. 
\end{lemma}

\begin{proof}
    This follows from axiom (\ref{collinearity}), which in this case says that when we view $X_H$ as a conic in $\PP^2$, then any line through $x_1$ intersects $X_H$ at an explained (and necessarily rational) point.  Since all rational points arise this way (by Euclid's argument), all rational points are explained.
\end{proof}

\begin{lemma}\label{lem: genus 1 collinearity}
    Suppose $g(X_H) = 1$ and $x \in X_H(\Q)$ is explained. View $X_H(\Q)$ as a group with identity element $x$. If $S \subset X_H(\Q)$ is a subset consisting entirely of explained points, then the subgroup $\langle S \rangle \subset X_H(\Q)$ generated by $S$ also consists of explained points.
\end{lemma}

\begin{proof}
    Let $E$ be the elliptic curve $(X_H,x)$. Denote the addition on $E(\Q)$ by $\oplus$. To prove the lemma it is enough to show that the inverse of an explained point is explained and the sum of two explained points is explained. We will use axiom (\ref{collinearity}) and the fact that $x_1 \oplus x_2 \oplus y = 0$ if and only if $x_1 + x_2 + y$ is linearly equivalent to $3x$.  If $s \in E(\Q)$ is explained, then $3x$ is linearly equivalent to $(-s) + s + x$, and so $-s$ is explained. If $s_1, s_2 \in E(\Q)$ are explained, $3x$ is linearly equivalent to $(-s_1\oplus -s_2) + s_1 + s_2$, and so $-s_1 \oplus-s_2$ is explained, and hence $s_1 \oplus s_2$ is explained.    
\end{proof}

\begin{lemma}\label{lem: non-hyperelliptic collinearity}
    Suppose $X_H$ is a geometrically non-hyperelliptic curve, and let $y+ \sum_{i = 1}^n x_i$ be a degree $2g-2$ hyperplane section on $X_H$ viewed in $\PP^{g-1}$ via the canonical embedding. If the closed points $x_i$ are all explained, then $y$ is rational and explained. In other words, if $y$ is coplanar with only explained points, then it is itself explained.      
\end{lemma}
\begin{proof}
    Again this follows from axiom (\ref{collinearity}). 
\end{proof}

We finish with some remarks on the motivation for the other axioms. 

\begin{remark}[Motivation for axioms of explained points]\label{rem: motivation for definition of explained} \hfill
\begin{itemize}
    \item Special points are explained by geometry since they arise from the boundary of the open modular curve or from fixed points of Atkin--Lehner involutions. When they are rational, their rationality is explained by Shimura reciprocity.  
    \item (\ref{pushforward}) is a generalization of Ogg's explanation of the rational points on $X_0(37)$. To avoid tautologies we impose additional rigidifying restrictions on the morphisms when $g(X_H)\leq 1$. 
    \item (\ref{twistlift}) is motivated by Lemma \ref{lem: abelian twist lifts}. 
    \item When $g(X_H) = 2$, axiom (\ref{collinearity}) says that the hyperelliptic involution sends explained points to explained points. Thus, it too generalizes Ogg's  $X_0(37)$ example.
    \item A crucial aspect of (\ref{collinearity}) is that given a set of known explained points in $X_H(\Q)$, there is a terminating algorithm to find all explained points on $X_H$ that can be found using this axiom.  Indeed, the degree of any non-rational explained closed point $x_i$ is bounded by $2g-3$ (or by $2$ if $g = 1$) and any such $x_i$ must be special. The claim then follows from the fact that there are finitely many special closed points on $X_H$ of bounded degree.
    \item (\ref{fibersplice}) says that if $\{y\}$ is the intersection of two rational fibers, then $y$ itself must be rational. The extra hypotheses on $E_1$ ensure that there is a terminating algorithm to find such points. 
    \item It is important in Ogg's example $X_0(37)$ that the hyperelliptic involution is defined over $\Q$. One might ask if there is a geometric reason that all morphisms $\pi$ and $f_i$ above are defined over $\Q$. Indeed there is, as Brunault \cite{Brunault} has proven that $\Hom(\mathrm{Jac}(X_H),\mathrm{Jac}(X_{H'})) \otimes \Q$ is generated by Hecke operators via the geometric action of $\GL_2(\Q)^+$ on the upper half-plane. Since we have sufficiently rigidified,  
    we may view any morphism $X_{H'} \to X_{H}$ that is allowed by the axioms as arising from the Hecke operators as well. Similarly, all maps $X_H \to E$ where $E$ is an $\Q$-isogeny factor of $\mathrm{Jac}(X_H)$ arise from this construction. 
\end{itemize}
\end{remark}

\section{Proof of Theorem \ref{thm: B}}
In this section we prove Theorem \ref{thm: B}. We will use Theorem \ref{thm 1+2} (and the accompanying Tables~\ref{table:twist isolated} and \ref{table:bigtableofgroups}) as well as Theorem \ref{thm: 41 twist isolated points} to show that all rational points $y \in X_H(\Q)$ on all modular curves $X_H$ are explained by geometry in the sense of the previous section. We only invoke Conjecture \ref{conj: GIC} later in the proof, so the first part concerning rational points whose $j$-invariants are twist-parameterized is unconditional.

\subsection{Reduction to the twist-isolated \texorpdfstring{$j$}{j}-invariants}

\begin{theorem}\label{thm: twist parameterized explained}
    Let $y \in X_H(\Q)$ be a rational point whose $j$-invariant is twist parameterized. Then $y \in X_H(\Q)$ is explained.
\end{theorem}
\begin{proof}
    If $y$ is special, then this is clear, so we may assume $y$ is non-special.  Let $x:=j(y)$. By axiom (\ref{pushforward}) it is enough to show that  all rational points on $X_{G_x}$ are explained (since we may then push them forward to explain $y$).  
    By assumption, $X_{G_x}$ is twist parameterized over its agreeable closure.  
    If we assume Conjecture \ref{conj: GIC}, then it follows from Theorem \ref{thm 1+2} that $X_{G_x}$ is an abelian cover of some $X_{M_i}$. However, this also follows unconditionally from Theorem \ref{thm:zywina-finite-Q}(b) and our proof of Theorem \ref{thm 1+2}. 
    By axiom (\ref{twistlift}), it is enough to show that all rational points on the $138$ twist-parameterized curves $X_{M_i}$ in Table~\ref{table:bigtableofgroups} are explained. 
    
    We begin with the 116 genus $0$ curves $X_{M_i}$ in Table \ref{table:bigtableofgroups}. By Lemma \ref{lem: genus 0 collinearity}, it is enough to show that each curve has at least one explained rational point (e.g.\ a rational  special point). We check that only five of these curves do not have rational special points, namely the curves \href{https://beta.lmfdb.org/ModularCurve/Q/8.12.0.h.1/}{8.12.0.h.1}, \href{https://beta.lmfdb.org/ModularCurve/Q/8.12.0.g.1/}{8.12.0.g.1}, \href{https://beta.lmfdb.org/ModularCurve/Q/16.16.0.b.1/}{16.16.0.b.1}, \href{https://beta.lmfdb.org/ModularCurve/Q/16.48.0.h.1/}{16.48.0.h.1}, and \href{https://beta.lmfdb.org/ModularCurve/Q/16.48.0.t.2/}{16.48.0.t.2}. 
    However all five of these curves are abelian covers of genus $0$ curves that do have rational special points. By axioms (\ref{special}) and (\ref{twistlift}) and Lemma \ref{lem: genus 0 collinearity}, all rational points on these five curves are explained as well. 

    Next consider the remaining curves $X_{M_i}$, which are all genus $1$ curves with rank 1. 
    Fourteen of these curves are  degree 2 abelian covers of modular curves of genus $0$ with rational special points, so all of their rational points are explained. The remaining curves are \href{https://beta.lmfdb.org/ModularCurve/Q/11.55.1.b.1/}{11.55.1.b.1}, \href{https://beta.lmfdb.org/ModularCurve/Q/15.15.1.a.1/}{15.15.1.a.1}, \href{https://beta.lmfdb.org/ModularCurve/Q/15.30.1.a.1/}{15.30.1.a.1}, \href{https://beta.lmfdb.org/ModularCurve/Q/20.20.1.c.1/}{20.20.1.c.1}, \href{https://beta.lmfdb.org/ModularCurve/Q/21.63.1.a.1/}{21.63.1.a.1}, \href{https://beta.lmfdb.org/ModularCurve/Q/24.48.1.mk.1/}{24.48.1.mk.1}, \href{https://beta.lmfdb.org/ModularCurve/Q/30.30.1.a.1/}{30.30.1.a.1}, \href{https://beta.lmfdb.org/ModularCurve/Q/36.36.1.e.1/}{36.36.1.e.1}. For each of these 8 curves we check that all rational points are generated by special points (where we view $X_H$ as a group with identity element some special point). Thus, their points are explained by Lemma \ref{lem: genus 1 collinearity}.
    \end{proof}

\begin{lemma}\label{lem: reduction to agreeable twist isolated}
If $y \in X_H(\Q)$ is not explained, then $y$ is non-special, $j_0:=j(y)$ is twist isolated, and there exists a rational point on $X_{G_{j_0}^{\mathrm{ag}}}$ which is not explained.
\end{lemma}
\begin{proof}
    That $y$ is not special is clear. That $j(y)$ is twist isolated follows from Theorem \ref{thm: twist parameterized explained}. If $X_{G_{j_0}^{\mathrm{ag}}}(\Q)$ is entirely explained, then $X_{G_{j_0}}(\Q)$ is explained by axiom (\ref{twistlift}) and hence $y$ is explained by axiom (\ref{pushforward}), which is a contradiction. 
\end{proof}

By Lemma \ref{lem: reduction to agreeable twist isolated}, to prove Theorem \ref{thm: B} it suffices to explain the rational points on the agreeable modular curves $X_{G_{j_0}^{\ag}}$  attached to the $41$ distinct twist-isolated $j$-invariants $j_0$ listed in Table \ref{table:twist isolated}. We consider these $22$ curves from low genus to high genus, in the following subsections. 

For later use we record the following corollary.
\begin{corollary}\label{cor: X_H infinite is explained}
    Assume Conjecture \ref{conj: GIC}. If $X_H(\Q)$ is infinite, then all rational points on $X_H(\Q)$ are explained.
\end{corollary}
\begin{proof}
    Note that $X_H$ has genus $0$ or $1$. By Lemma \ref{lem: reduction to agreeable twist isolated} and Theorem \ref{thm: 41 twist isolated points}, all but finitely many rational points on $X_H$ are explained. It now follows from Lemmas \ref{lem: genus 0 collinearity} and \ref{lem: genus 1 collinearity} that all of $X_H$ is explained. (In the genus $1$ case we are using that an infinite, finitely generated abelian group is generated by any cofinite subset.)  
\end{proof}

\subsection{Genus 1}
The following curves $X_{G_{j}^{\ag}}$ of genus $1$ from Table \ref{table:twist isolated} are explained by collinearity via Lemma \ref{lem: genus 1 collinearity}.
\begin{itemize}
\item  The group of rational points on \href{https://beta.lmfdb.org/ModularCurve/Q/12.32.1.b.1/}{12.32.1.b.1} is generated by the four rational cusps, so by Lemma \ref{lem: genus 1 collinearity}, all rational points are explained. This curve has two rational non-special points.

\item The four rational points of \href{https://beta.lmfdb.org/ModularCurve/Q/20.24.1.g.1/}{20.24.1.g.1} are generated by the two rational cusps. 
\item The four rational points of \href{https://beta.lmfdb.org/ModularCurve/Q/17.18.1.a.1/}{$X_0(17)$} are generated by the two cusps.

\end{itemize}
The following $j$-invariants are explained by push-forward.
\begin{itemize}
    \item $X =$ \href{https://beta.lmfdb.org/ModularCurve/Q/12.24.1.h.1/}{12.24.1.h.1} has exactly one rational special point and one rational non-special point. However it is isomorphic to the modular curve \href{https://beta.lmfdb.org/ModularCurve/Q/12.6.1.a.1/}{12.6.1.a.1} which has two rational special points. By axiom (\ref{pushforward}) the rational points in $X(\Q)$ are explained.     

\item $X =$ \href{https://beta.lmfdb.org/ModularCurve/Q/12.24.1.g.1/}{12.24.1.g.1} has exactly one rational special point and one rational non-special point. However it is isomorphic to the modular curve \href{https://beta.lmfdb.org/ModularCurve/Q/12.18.1.i.1/}{12.18.1.i.1} which has two rational cusps. By axiom (\ref{pushforward}) the two rational points in $X(\Q)$ are explained.

\end{itemize}

Finally, the curve $X =$ \href{https://beta.lmfdb.org/ModularCurve/Q/12.48.1.q.1/}{12.48.1.q.1} is an elliptic curve of rank $0$ where  $X(\Q) \simeq \Z/2 \Z \times \Z/2\Z$. 
The four rational points correspond to two rational cusps and two non-special points with $j= 2^{-6} \cdot 3^{3} \cdot 5^{3}$. The cusps generate a subgroup isomorphic to $\Z/2\Z \times \{0\}$, so the non-special points are not explained by collinearity. 
However, $X$ is a degree 2 abelian cover of \href{https://beta.lmfdb.org/ModularCurve/Q/12.24.1.o.1/}{12.24.1.o.1}, all of whose rational points are generated by rational special points, which explains the non-special rational points on $X$.

\subsection{Genus 2}
When $X_{G_j^{\ag}}$ has genus $2$, its rational points can all be explained either using the automorphism group of the curve (i.e.\ axiom (\ref{pushforward})) or by abelian lift of a genus 1 curve.

The following twist-isolated $j$-invariants are explained by abelian lift of genus 1 curves.
\begin{itemize}
\item $-2^{4}\cdot 3^{2}\cdot 13^{3}$ and $2^{4}\cdot 3^{3}$ on \href{https://beta.lmfdb.org/ModularCurve/Q/60.40.2.a.1/}{60.40.2.a.1} are abelian lifts from \href{https://beta.lmfdb.org/ModularCurve/Q/15.20.1.a.1/}{15.20.1.a.1}. The latter is a genus $1$ curve with five rational points which are generated by the three rational special points. 
\item $-11\cdot 131^{3}$ and $-11^2$ on \href{https://beta.lmfdb.org/ModularCurve/Q/44.24.2.a.1/}{44.24.2.a.1} are abelian lifts from  \href{https://beta.lmfdb.org/ModularCurve/Q/11.12.1.a.1/}{$X_0(11)$}. The latter has genus $1$ with Mordell--Weil group (taking the cusp $\infty$ to be the identity) of order $5$, generated by the cusp $0$. Hence all rational points are explained by (\ref{collinearity}).  
\end{itemize}
Meanwhile, the following $j$-invariants are explained by automorphisms.
\begin{itemize}
    \item The curve  \href{https://beta.lmfdb.org/ModularCurve/Q/25.75.2.a.1/}{25.75.2.a.1} (conjecturally) 
    has exactly two rational points. There is one special and at most one non-special point with $j = \frac{2^{12}\cdot 3^{3}\cdot 5^{7} 29^3}{7^{5}}$. These points are swapped by the hyperelliptic involution.
    \item The curve $X_{\rm ns}^+(16)$ has 10 rational points \cite{Baran} (c.f. \cite[Remark 1.5]{RouseDZB}), and 8 of them are special. 
The two non-special points are in the  $\Aut(X_{\rm ns}^+(16))$-orbit of CM points, so they are explained.  The curve is bielliptic, and each bielliptic involution takes the non-CM points to CM points.
\item As Ogg observed, the two non-special $j$-invariants on $X_0(37)$ are images under the hyperelliptic involution of rational cusps \cite{mazur78}.

\end{itemize}

\subsection{Genus 3}
\label{subsec:genus3}
For all but one curve $X_{G_j^{\ag}}$ in Table \ref{table:twist isolated} of genus 3, we can explain its rational points by abelian lifts, non-trivial automorphisms, or some combination of both.  In more detail: 
\begin{itemize}
\item \href{https://beta.lmfdb.org/ModularCurve/Q/120.48.3.c.1/}{120.48.3.c.1} is a double cover of \href{https://beta.lmfdb.org/ModularCurve/Q/15.24.1.a.1/}{$X_0(15)$}, an elliptic curve of rank 0 whose rational points are generated by cusps. 
This explains the $j$-invariants $-\frac{5^{2}\cdot 241^{3}}{2^{3}}$, $-\frac{5\cdot 29^{3}}{2^{5}}$, $-\frac{5^{2}}{2}$, and $\frac{5\cdot 211^{3}}{2^{15}}$.
\item \href{https://beta.lmfdb.org/ModularCurve/Q/120.72.3.gqp.1/}{120.72.3.gqp.1} is a double cover of \href{https://beta.lmfdb.org/ModularCurve/Q/15.36.1.b.1/}{15.36.1.b.1},  an elliptic curve of rank 0. We check that its rational points are generated by cusps and hence are explained by (\ref{collinearity}). This explains the $j$-invariants $-\frac{29^{3}\cdot 41^{3}}{2^{15}}$ and $\frac{11^3}{2^3}$.
\item \href{https://beta.lmfdb.org/ModularCurve/Q/168.64.3.a.1/}{168.64.3.a.1} is a double cover of \href{https://beta.lmfdb.org/ModularCurve/Q/21.32.1.a.1/}{$X_0(21)$}.  We check that $X_0(21)(\Q)$ is generated by cusps and hence its points are explained by Lemma \ref{lem: genus 1 collinearity}.

This explains the $j$-invariants $-\frac{3^{3}\cdot 5^{3}\cdot 383^{3}}{2^{7}}$, $-\frac{3^{2}\cdot 5^{6}}{2^{3}}$, $-\frac{3^{2}\cdot 5^{3}\cdot 101^{3}}{2^{21}}$, and $\frac{3^{3}\cdot 5^{3}}{2}$.
\item $X_3 =$ \href{https://beta.lmfdb.org/ModularCurve/Q/28.64.3.b.1/}{28.64.3.b.1} is a double cover of the genus 2 curve $X_2 =$ \href{https://beta.lmfdb.org/ModularCurve/Q/28.32.2.a.1/}{28.32.2.a.1}, which has two rational points with $j$-invariants $-\frac{3^{3}\cdot 13\cdot 479^{3}}{2^{14}}$ and  $\frac{3^{3}\cdot 13}{2^{2}}$, as well as two rational cusps. We check that $\#\Aut(X_2) = 4$ and $X_2(\Q)$ is a single $\Aut(X_2)$-orbit. Thus, the rational points with these $j$-invariants on $X_2$ are explained, and hence the corresponding rational points on the abelian cover $X_3$ are also explained. 
\item $X =$ \href{https://beta.lmfdb.org/ModularCurve/Q/20.60.3.w.1/}{20.60.3.w.1} has a rational point with $j$-invariant $-\frac{3^{3}\cdot 5^{4}\cdot 11^{3}\cdot 17^{3}}{2^{10}}$. We have $\Aut(X) \simeq \Z/2\Z$ and $\#X(\Q) = 6$, with $5$ of these rational points special. The involution takes the rational non-special point to a point with $j$-invariant $0$, so all rational points are explained.
\end{itemize}

There remains exactly one genus $3$ agreeable modular curve whose points need explaining: the curve $X = X_{S_4}(13) = $ \href{https://beta.lmfdb.org/ModularCurve/Q/13.91.3.a.1/}{13.91.3.a.1}. In the canonical model 
\[5x^4 - 7x^3y + 8x^3z + 3x^2y^2 - 7x^2yz + 4x^2z^2 + 2xy^3 - 2xy^2z - 5xyz^2 - 3xz^3 + 5y^3z + y^2z^2 + 2yz^3 = 0 \]
in $\PP^2$, the curve $X$ has the four rational points $y_0 = (0 : 1 : 0)$, $y_1 = (1: -2:1)$, $y_2 = (0 :0 : 1)$, and $y_3 = (-1: -1:1)$. The $j$-invariants are 
    \[j(y_0) = 2^4 \cdot 3^{-13} \cdot 5 \cdot 13^4  \cdot 17^3,\]
     \[j(y_1) = 2^{18} \cdot 3^3\cdot 5^{-13}\cdot 13^4 \cdot 61^{-13} \cdot 127^3\cdot  139^3 \cdot 157^3\cdot  283^3\cdot  929,\]
    \[j(y_2) =-2^{12}  \cdot 3^{-13}\cdot  5^3 \cdot 11\cdot  13^4,
\]
and $j(y_3) = 0$. The CM point $y_3$ is special and therefore explained.  Axiom $(\ref{collinearity})$ states in this case that if $y_i$ is $\Q$-collinear with three explained points, then it is itself explained.  We check that 
\[X \cap \{x+2z = 0\} = y_0 + x_1,\] 
where $x_1$ is a degree 3 CM point with $j$-invariant $287496 = 2^3\cdot3^3\cdot11^3$, so $y_0$ is explained.  Meanwhile, 
\[X \cap \{x -z = 0\} = y_0 + y_1 + x_2,\]
where $x_2$ is a quadratic CM point with $j$-invariant $1728$; hence $y_1$ is explained.  

It remains to explain $y_2$, for which collinearity relations do not suffice. To set things up, recall (Example \ref{ex: twists of Xns(2)}) that for each squarefree integer $d$,
the modular curve $X^d_{\mathrm{ns}}(2)$ parameterizes elliptic curves together with a choice of a square root of $d\Delta(E)$.  Then $y_2$ lifts to two rational points on 
\[\tilde{X} = X_{S_4}(13) \times_{X(1)} X^{-3}_{\mathrm{ns}}(2)\] since the discriminant of $E_{j(y_2)}$ is $-3$ times a square. This is the genus $10$ curve \href{https://beta.lmfdb.org/ModularCurve/Q/78.182.10.a.1/}{78.182.10.a.1}. Interestingly, the CM point $y_3$ also lifts to two rational points on $\tilde X$ (this can be seen by ``pure thought'' since $j(y_3) = 0$ and the discriminant of any polynomial $x^3 + k$ is $-27k^2$).  Let $x_2, x_2'$ be the lifts of the non-special point $y_2$ to $\tilde{X}$ and let $x_3, x_3'$ be the lifts of the CM-point $y_3$. We will explain $x_2,x_2'$ by fiber splicing.  Using that $\tilde X$ is a quadratic twist of $X_{S_4}(13) \times_{X(1)} X_{\mathrm{ns}}(2) =$ \href{https://beta.lmfdb.org/ModularCurve/Q/26.182.10.a.1/}{26.182.10.a.1}, we see from the newform decomposition that $\tilde X$ admits maps $f_i \colon \tilde{X} \to E_i$ to four different elliptic curves $E_i$, all of which have analytic rank $1$. We may take $E_1 =$ \href{https://www.lmfdb.org/EllipticCurve/Q/6084/f/1}{6084.f1} and $E_2 =$ \href{https://www.lmfdb.org/EllipticCurve/Q/468/b/1}{468.b1}. Using code of Mayle--Rouse \cite{MayleRouse}, we find that $f_1(x_2) = f_1(x_3)$ and $f_1(x_2') = f_1(x_3')$, while the images of the four points under $f_2$ are all distinct. It follows that $f_1(x_2)$ and $f_1(x_2')$ are explained on $E_1$. To see that $f_2(x_2)$ and $f_2(x_2')$ are explained, we check that $E_2$ is isomorphic to the  modular curve \href{https://beta.lmfdb.org/ModularCurve/Q/78.56.1.b.1/}{78.56.1.b.1}. 
Since $E_2$ has rank $1$, all of its rational points are explained by Corollary \ref{cor: X_H infinite is explained}, and so $f_2(x_2)$ and $f_2(x_2')$ are indeed explained. Finally, we compute that $f_1^{-1}(f_1(x_2)) \cap f_2^{-1}(f_2(x_2)) = \{x_2\}$ and similarly for $x_2'$, showing that $x_2$ and $x_2'$ are explained by (\ref{fibersplice}). Hence $y_2$ is explained by $(\ref{pushforward})$. 

To perform the above computations, we had to identify which pair of points  $x_2,x_2'$ and $x_3,x_3'$ are the lifts of $y_2$ and $y_3$ respectively. (This is not a priori clear since the code used to generate a model for the curve $\tilde X$ did not come with a map to $X_{S_4}(13)$.) For this, we found an involution $\iota \in \Aut(\tilde X)$ such that $\tilde{X}/\iota \simeq X_{S_4}(13)$. We now argue that there is a unique involution on $\tilde X$ with this property and hence the quotient $\tilde X \to \tilde X/\iota$ is the modular map. 
The involution on $\tilde{X}$ induces an involution on the Jacobian. 
Recalling that $\tilde{X}$ is a twist of the curve $X'= $ \href{https://beta.lmfdb.org/ModularCurve/Q/26.182.10.a.1/}{26.182.10.a.1}, we deduce that the  Jacobian of $\tilde{X}$ factors up to isogeny as $J_X \times A_3 \times E_1 \times E_2 \times E_3 \times E_4$, where $A_3$ is an abelian threefold, and where the factors are all simple and pairwise non-isogenous. 
It follows that the involution is given by a block diagonal matrix with $\pm 1$ in each block. But the $+1$-eigenspace is exactly $J_X$ so it must induce $-1$ on the other factors. In other words, the involution is uniquely determined in $\mathrm{End}^0(J_{\tilde{X}})$ and hence uniquely determined in $\Aut(\tilde X)$.

\subsection{Genus 4}

The genus 4 curve \href{https://beta.lmfdb.org/ModularCurve/Q/100.100.4.d.1/}{100.100.4.d.1} with a rational point with $j = 2^4 \cdot 3^2\cdot 5^7 \cdot 23^3$ is an abelian cover of the genus 2 curve \href{https://beta.lmfdb.org/ModularCurve/Q/25.50.2.a.1/}{25.50.2.a.1}.
This curve has at least two rational points, and exactly one non-special point corresponding to $j = 2^4 \cdot 3^2\cdot 5^7 \cdot 23^3$. The hyperelliptic involution sends the non-special point to a CM point with $j = 0$.

\subsection{Genus 6}\label{subsec: genus 6}

There are two genus $6$ agreeable curves $X$ in Table \ref{table:twist isolated}.  We explain their non-special rational points by fiber splicing two maps $f_1 \colon X \to E_1$ and $f_2 \colon X \to E_2$ as in axiom $(\ref{fibersplice})$. 

\begin{itemize}
    \item The curve $X = $ \href{https://beta.lmfdb.org/ModularCurve/Q/36.108.6.g.1/}{36.108.6.g.1} has a unique non-special point $x \in X(\Q)$ and two rational CM points $x_0,x_0' \in X(\Q)$. The Jacobian $J_X$ is isogenous to $E_1 \times E_2 \times E_3 \times E_4 \times S$ for an abelian surface $S$. Some of these  factors are modular curves, e.g.\ we may take the degree $9$ map $f_2 \colon X \to E_2$ where $E_2 = $ \href{https://beta.lmfdb.org/ModularCurve/Q/36.36.1.e.1/}{36.36.1.e.1}, a rank $1$ elliptic modular curve. There is exactly one factor, $E_1 = $ \href{https://beta.lmfdb.org/EllipticCurve/Q/432/f/1}{432.f1}, with rank $0$. In fact, we have $E_1(\Q) = \{0\}$ and the map $f_1 \colon X \to E_1$ has degree $6$. Since we have a rational CM point on $X$, the point $f_1(x)$ is explained regardless of whether $E_1$ is a modular curve or not. The point $f_2(x)$ is explained because $E_2$ is a modular curve all of whose rational points are explained by Corollary \ref{cor: X_H infinite is explained}.  Using code of Mayle--Rouse \cite{MayleRouse}, we check that $f_1^{-1}(f_1(x)) \cap f_2^{-1}(f_2(x)) = \{x\}$, so $x$ is explained by (\ref{fibersplice}). 
    \item The curve $X = $  \href{https://beta.lmfdb.org/ModularCurve/Q/20.120.6.a.1/}{20.120.6.a.1} has two non-special rational points $x,x' \in X(\Q)$ but no rational special points. However, $X$ admits a map $f_1 \colon X \to E_1$ where $E_1 = $ \href{https://www.lmfdb.org/EllipticCurve/Q/80/b/4}{80.b4} $=$ \href{https://beta.lmfdb.org/ModularCurve/Q/20.36.1.a.1/}{20.36.1.a.1}, which has two rational points, both cusps.   We find that $f_1(x) = f_1(x')$. For the curve $E_2$ we can take the modular curve \href{https://beta.lmfdb.org/ModularCurve/Q/20.20.1.c.1/}{20.20.1.c.1} = \href{https://beta.lmfdb.org/EllipticCurve/Q/400/a/1}{400.a1} of rank $1$, and we compute $f_2(x) \neq f_2(x')$. More precisely, $f^{-1}_1(f_1(x)) \cap f_2^{-1}(f_2(x)) = \{x\}$ and $f^{-1}_1(f_1(x')) \cap f_2^{-1}(f_2(x')) = \{x'\},$ 
    for this choice of $f_2$. 
    This shows that both $x$ and $x'$ are explained by fiber splicing.
\end{itemize}

\subsection{Genus 7}\label{subsec: genus 7}
The $j$-invariants $-2^{6}\cdot 719^{3}$ and $2^6$ have the same genus 7 group $G_j^{\ag}$. 
The corresponding modular curve is an abelian cover of the genus 2 curve $X_2 =$ \href{https://beta.lmfdb.org/ModularCurve/Q/45.54.2.c.1/}{45.54.2.c.1}.
There are two rational cusps on $X_2$ and two rational points with the aforementioned $j$-invariants. We have $\#\Aut(X_2) = 4$ and each rational point is in the $\Aut(X_2)$-orbit of the rational cusps. Thus, these two $j$-invariants are fully explained.

It remains to explain the rational points on the genus $7$ agreeable curve \href{https://beta.lmfdb.org/ModularCurve/Q/30.120.7.e.1}{30.120.7.e.1}. This is a degree 2 abelian cover of the genus 4 curve
$X =$ \href{https://beta.lmfdb.org/ModularCurve/Q/30.60.4.a.1/}{30.60.4.a.1}, and it suffices to explain the $j$-invariant $j=- 3^{3} \cdot 5^{4} \cdot 11 \cdot 19^{3}$ on $X$.
The curve $X$ has the canonical model in $\PP^3$
    \[ x^2 - yz = 20y^3 + 15xyz - 5z^3 - 30xyw + 6zw^2 + w^3 = 0.\]
Note the rational points $y_0 = (0 : 0 : -1 : 1)$ and $y_1 = (-1: 1 : 1 : 0)$. We have $j(y_0) = 0$ whereas $y_1$ is the unique non-special rational point on $X$ with $j$-invariant $-3^{3} \cdot 5^{4} \cdot 11 \cdot 19^{3}$. 
Since $X$ is contained in the singular quadric $x^2 - yz$, it admits a unique $g^1_3$, i.e.\ isomorphism class of degree $3$ map $f \colon X \to \PP^1$. We find that  $f = (x \colon y)$
and $f^{-1}(-1\colon 1)$ contains $y_1$ and a degree $2$ CM closed point $x_1$ (with CM  by $\Z[\sqrt{-7}]$ and $j$-invariant $16581375$).  This is already a geometric explanation of $y_1$ but we can relate it to our general rules. Indeed, $y_1$ is coplanar with five special points:
\[K_X \equiv \{y + x = 0\} \cap X = y_0 + y_1 + x_1 + x_2,\] 
where $x_2$ is the degree two cuspidal divisor. So $y_1$ is explained by rule $(4)$. 

\subsection{Summary}
Table \ref{table: explanations} summarizes how we explain the $j$-invariants in Table \ref{table:twist isolated} by geometry.  ``Genus'' refers to the genus of the curve on which the $j$-invariant is explained not including abelian covers. (It is not necessarily the agreeable genus since the agreeable quotient of $G_j$ may be an abelian cover of another curve).  
Therefore ``abelian lift'' never appears as a reason. 
The multiplicities of the given genus are listed as exponents.
\begin{table}[h]
\centering
\begin{tabular}{lll}
\hline
Reason & Count  & Genus \\
\hline
Push-forward & 13  & $1^2, 2^{10}, 3 $\\
Collinear  & 24  & $0^4, 1^{17}, 3^2, 4$ \\
Fiber splicing & 4 & $6^3, 10$\\
\hline
\end{tabular}
\caption{Summary of geometric explanations}
\label{table: explanations}
\end{table}
Interestingly, in the case of $X_{S_4}(13)$, not all of the $j$-invariants on this curve are explained by the same reason. 

One final observation, in the vein of  \cite{GrossZagier85}, is that the $j$-invariants in Table \ref{table:twist isolated} are all quite smooth. Galbraith \cite{GalbaraithExpMath} noticed similar phenomena for points on $X_0(p)^+ = X_0(p)/w_p$, and some results about cubic factors of $j$-invariants of $\Q$-curves were proven in \cite{GonzalezCubic}.

\section{Families of non-Serre curves}
\label{sec:nonserre}
We give an application of Theorem \ref{thm 1+2} that also shows how to navigate the 160 families of modular curves therein. 
Recall that for an elliptic curve $E$ over $\Q$, the index of $G_E$ satisfies  $[\GL_2(\Zhat) \colon G_E] \geq 2$, 
and $E$ is said to be a \emph{Serre curve} if equality holds.  Jones proved that $100\%$ of elliptic curves over $\Q$ are Serre curves,  when elliptic curves are ordered by a suitable height function \cite{Jones}. In the notation of Theorem \ref{thm 1+2}, Serre curves correspond to the non-special rational points on $X_{M_1}(\Q) = X(1)(\Q)$  that do not lift to a rational point on any other $X_{M_i}$. 

The following result shows that non-Serre curves come in 20 families, each one corresponding to a curve $X_{M_i}$ in Theorem \ref{thm 1+2} that happens to have a ``clean'' moduli interpretation.     

\begin{theorem}
\label{thm:nonSerre}
    Assume Conjecture \ref{conj: GIC} and let $E$ be an elliptic curve over $\Q$ that is not a Serre curve, i.e.\  $[\GL_2(\Zhat) \colon G_E] > 2$. Then at least one of the following holds:
    \begin{enumerate}
        \item $E$ has an $\ell$-isogeny for some  $\ell \in \{2,3,5,7,11,13,17,37\}$,
        \item $\rho_{E,\ell}(\Gal_\Q) \subset \GL_2(\Z/\ell\Z)$ lies in a normalizer of a non-split Cartan, for some $\ell \in \{3,4,5,7,11\}$,
        \item The image of $\rho_{E,\ell}(\Gal_\Q)$ in $\PGL_2(\Z/\ell\Z)$ lies in $S_4$, for some $\ell \in \{5,13\}$, 
        \item $\rho_{E,7}(\Gal_\Q) \subset \GL_2(\Z/7\Z)$ lies in the normalizer of a split Cartan, 
        \item $\mathrm{Disc}(E)$ or $-\mathrm{Disc}(E)$  is a square  in  $\Q^\times$,
        \item $\Q(E[2]) \subset \Q(E[3])$, or
        \item $\Q(E[9]) = \Q(E[3],\zeta_9)$.
    \end{enumerate}
    Conversely, each family above contains at least one non-CM elliptic curve $E$ over $\Q$. 
\end{theorem}
\begin{proof}
    First we observe that the properties in (1)-(7) are the moduli interpretations of twenty different curves $X_{M_i}$. Namely, the first four cases correspond to the curves $X_0(\ell)$, $X_{\mathrm{ns}}^+(\ell)$, $X_{S_4}(\ell)$, and $X_{\mathrm{sp}}^+(7)$, respectively. The final three are covered by the following lemmas.

     \begin{lemma}
        $\mathrm{Disc}(E) \in \Q^{\times2}$ if and only if $E$ gives rise to a rational point on $X_{\mathrm{ns}}(2)$. Similarly, $-\mathrm{Disc}(E) \in \Q^{\times 2}$ if and only if $E$ gives rise to a rational point on $X_H$, where $H = \href{https://beta.lmfdb.org/ModularCurve/Q/4.2.0.a.1/}{4.2.0.a.1}$.
    \end{lemma}
    \begin{proof}
    This follows from Example \ref{ex: twists of Xns(2)}, once we observe that $X_H$ is the $-1$-twist of $X_{\mathrm{ns}}(2)$.
    \end{proof}

    \begin{lemma}
        Suppose $\Q(E[2]) \subset \Q(E[3])$ and properties $(1)$ and $(5)$  do not hold. Then $E$ gives rise to a rational point on the curve $X_H$, where $H = \href{https://beta.lmfdb.org/ModularCurve/Q/6.6.0.b.1/}{6.6.0.b.1}$.
    \end{lemma}
    \begin{proof} This follows from \cite[Theorem 1.4.]{BrauJones16} (see \cite[Remark 1.9]{JonesMcMurdy22} for a correction).  We sketch the argument here.  
    If $\rho_{E,2}$ is not surjective, then $E$ gives rise to a point on either $X_0(2)$ or $X_{\mathrm{ns}}(2)$, which appear in $(1)$ and $(5)$. So we may assume $\rho_{E,2}$ is surjective.
        Let $f$ be the map $\GL_2(\F_3) \to S_4 \to S_3 \simeq \GL_2(\F_2)$, where the first map is the unique double cover. The lemma follows from the fact that $H$ is the pre-image of the subgroup $\{(f(h), h)\} \subset \GL_2(\F_2) \times \GL_2(\F_3)$ in $\GL_2(\Zhat)$.
    \end{proof}
    
    \begin{lemma}\label{lem: elkies curve}
        Let $\zeta_9$ be a primitive $9$-th root of unity. Suppose $\Q(E[9]) = \Q(E[3],\zeta_9)$ and $(1)$ and $(2)$ do not hold. Then $E$ gives rise to a rational point on the curve $X_H$, where $H = \href{https://beta.lmfdb.org/ModularCurve/Q/9.27.0.a.1/}{9.27.0.a.1}$. 
    \end{lemma}

\begin{proof}
If $\rho_{E,3}(\Gal_\Q) \neq \GL_2(\F_3)$, then $E$ gives rise to a point on  $X_{\mathrm{ns}}^+(3)$ or $X_0(3)$,
which appear in $(1)$ and $(2)$. Otherwise, $\rho_{E,9}(\Gal_\Q) \subset \GL_2(\Z/9\Z)$ is contained in a level 9 subgroup $H' \subset \GL_2(\Z/9\Z)$ such that $X_{H'}$ minimally covers $X(1)$. One checks in the LMFDB that $H$ is the unique such subgroup up to conjugacy, so $H' = H$ and $E$ gives rise to a rational point on $X_H$. 
\end{proof}

\begin{remark}
    {
    Equivalently, $X_H$ parameterizes elliptic curves $E$ with $\rho_{E,3}$ surjective but $\rho_{E,9}$ not surjective; see \cite{Elkies2006}. 
    }
\end{remark}

    Now suppose $E$ is any non-CM elliptic curve. By Theorem \ref{thm 1+2}, it gives rise to a rational point on one of the $160$ modular curves $X_{M_i}$. To prove Theorem \ref{thm:nonSerre}, it is enough to show that either $X_{M_i}$ maps to one of the $20$ modular curves described above or $E$ is a Serre curve. We \gitlink{NonSerre.m}{check} 
    that this is the case for all $X_{M_i}$ except $X_{M_1} = X(1)$. This is good enough since, as we remarked above, all curves in the family $X_{K_1}^\chi = X_{\mathrm{ns}}(2)^\chi$ of double covers of $X(1)$ parameterize Serre curves, except for the two exceptions in case $(5)$. 
\end{proof}

\section{Twist families of super-sporadic points}\label{sec: super-sporadic}

A special type of isolated point is a \emph{sporadic} point. In this section we discuss sporadic and super-sporadic points on modular curves.

\begin{definition}
    For a reduced curve $X$ over $\Q$, we define $\delta(X)$ to be the least $d \in \mathbb{Z}^{+}$ such that $X(\overline \Q)$ contains infinitely many points of degree $d$, and then call a point $p \in X(\overline  \Q)$ \textit{sporadic} if $\deg p < \delta(X)$. We call a point $p$ on a modular curve $X_H$ \textit{super-sporadic} if for every $G \leq H $, every point of $X_G(\Qbar)$ lying over $p$ is sporadic.
\end{definition}

In \cite[Section 7]{BELOV} the authors consider whether a sporadic point $x\in X_1(N)$ lifts to a sporadic point on a $X_1(dN)$ in the sense that any $y\in X_1(dN)$ such that $\pi(y)=x$ is sporadic, where 
$$\pi: X_1(dN)\rightarrow X_1(N)$$ is the natural projection map. They show that if $x$ is a CM point then this is true for all $N$ prime and sufficiently large, and all $d$ \cite[Theorem 7.1]{BELOV}.

One naturally wonders whether there exists a non-CM sporadic point $x\in X_1(N)$ for some integer $N$ that lifts to a sporadic point on infinitely many $X_1(dN)$. They show that this is true if $\deg x$ is small compared to the index of $\Gamma_1(N)$ in 
$\PSL_2(\Z)$, see \cite[Lemma 7.2]{BELOV}
\begin{equation} \label{BELOV:bound}
    \deg x <\frac {7}{1600}[\PSL_2(\Z):\Gamma_1(N)].
\end{equation}
However, no known non-CM points satisfy this bound (and we expect there to be no points with rational $j$-invariant that satisfy this bound). 

One can ask the same question for general modular curves. 
In \cite[Theorem 7.6]{CGPS22} it is proved that for $N$ large enough, the modular curves $X_0(N)$ and $X_1(M,N)$ have super-sporadic CM points. 

\begin{question}
    Do there exist non-CM (rational) super-sporadic points on modular curves $x\in X_G(\Q)$?
\end{question}

We give a positive answer to this question, and produce twist families of examples. For a subgroup $G\leq \GL_2(\Zhat)$ we denote by $\tilde G\leq \PSL_2(\Z)$, obtained by first reducing $G$ modulo its level $N$, taking the intersection with $\SL_2(\Z/N\Z)$, then taking the inverse image of it with respect to $\red_N:\SL_2(\Z)\rightarrow\SL_2(\torz{N})$, and finally taking the quotient of $\pm$ this group by $\diamondop{-I}$. The modular curve $X_G$ is isomorphic over $\C$ to $\tilde G \backslash  \mathcal G^*=:X_{\tilde G}$. Note that we have
$$\gon_\C X_{\tilde G }=\gon_\C X_{G}\leq \gon_\Q X_G,$$ for all modular curves $X_G$ defined over $\Q$.

\begin{lemma} \label{lem:index}
    Let $G\leq \GL_2(\Zhat)$ be an open group with surjective determinant. Then
    $$[\PSL_2(\Z):\tilde G]=\alpha [\GL_2(\Zhat):G], \quad \text{ where } \alpha =\begin{cases} 2 \text{ if } I \notin G\\ 1 \text{ if } I \in G
        
    \end{cases}.$$
\end{lemma}
\begin{proof}
Let $N$ be the level of $G_E$ and denote for any group $G$ in this proof by $G(N)$ the reduction modulo $N$ of $G$. 
\begin{align*}
[\PSL_2(\Z):\tilde G]&=\alpha [\SL_2(\Z): \red_N ^{-1}(G(N)\cap \SL_2(\Z/N\Z))]=\alpha [\SL_2(\Z/N\Z): G(N)\cap \SL_2(\Z/N\Z)]\\
&=\alpha [\GL_2(\Z/N\Z): G(N)]=\alpha[\GL_2(\Zhat): G],
\end{align*}
where the second equality follows from the surjectivity of $\red_N$ (see also \cite[\S 3.5]{zywina_possible_indices}), the third follows from the surjectivity of the determinant, and the fourth equality follows from the definition of the level.
\end{proof}

For an open subgroup $G\leq \GL_2(\Zhat)$ define
$$i(G):= \frac{325\alpha }{2^{15}}[\GL_2(\Zhat): G],$$
where $\alpha$ as in Lemma \ref{lem:index}. We now prove the following generalization (and sharpening) of \cite[Lemma 7.2]{BELOV} (see also \cite[Theorem 7.3]{CGPS22}).

\begin{lemma} \label{lem:super-sporadic}
 Let $G\leq \GL_2(\Zhat)$ be an open subgroup with surjective determinant and $ x\in X_G(\overline \Q)$ satisfying $\deg_k x <i(G).$ Then $x$ is super-sporadic. 
\end{lemma}
\begin{proof}
    The proof follows the same argumentation as \cite[Lemma 7.2]{BELOV}. First recall Abramovich's bound \cite[Theorem 0.1]{abramovich} (see \cite[Theorem 4.1]{DerickxNajman25} for the constant $\frac{325}{2^{15}}$):
\begin{equation}\gon_\C X_{G}<i(G).\end{equation}
    Now by \cite[Proposition 2]{frey}, it follows that every $x\in X_G(\Qbar)$ that satisfies 
    $$\deg_k x< \frac{1}{2}\gon_\C X_{\tilde G}\leq \frac 1 2 \gon_\Q X_G$$
    is sporadic. 
    For an open subgroup $H\leq G$, the natural morphism $X_H\rightarrow X_G$ is of degree $[G:H]$ or $[G:H]/2$ (depending whether $- I \in H$ and $G$). On the other hand, we have $\frac{i(H)}{i(G)}=[G:H]$. So if $x\in X_G$ satisfies $\deg x<i(G)$, then every $y\in X_H$ mapping to $x$ satisfies $\deg y<i(H)$, so is sporadic.   
\end{proof}

Twist families will produce infinite families of modular curves with rational super-sporadic points. Let $\mathcal F(G,H)$ be one of the twist families from Table \ref{table:bigtableofgroups} such that $X_G(\Q)=\infty$ and $i(H)>1$. Then any rational point on $X_G(\Q)$ lifts to a super-sporadic point on $X_{H}^\psi(\Q)$ for some character $\psi$ by Lemma \ref{lem:super-sporadic}.

\begin{lemma}\label{lem:SS}
Let $G\leq \GL_2(\Zhat)$ be an open subgroup with surjective determinant, and let $\alpha$ be as defined in Lemma \ref{lem:index}. If  $[\GL_2(\Zhat): G]>\frac{100}{\alpha}$, any $x\in X_G(\Q)$ is super-sporadic.
\end{lemma}
\begin{proof}
    We have
$$i(G)=\frac {325 \alpha}{2^{15}}[\GL_2(\Zhat): G]\geq 1 \quad \text{ when } \quad [\GL_2(\Zhat): G]>\frac{100}{\alpha}.$$
The result now follows from Lemma \ref{lem:super-sporadic}.
\end{proof}

\begin{remark}
We check that 95 out of 138 families $\mathcal F(M_i, K_i)$ from Table~\ref{table:bigtableofgroups} satisfy the assumptions of Lemma \ref{lem:SS}. This does not mean that the rational points on the $X_{K_i}$ for the remaining families are not super-sporadic -- it is usually much harder to prove that a point is not super-sporadic. 

However, the exception to this is when $g(X_{K_i})=0$, then the points will certainly not be sporadic (and hence not super-sporadic). There are 8 families (for $i\in \{1,3,10,21,27,38,76,96\}$) that satisfy this. In particular, the family $\mathcal F(\GL_2(\Zhat), \Ns(2))$ that contains the images of $100\%$ of elliptic curves over $\Q$ (by the argumentaion at begining of Section \ref{sec:nonserre}), will not produce super-sporadic points as $g(X_{\Ns(2)})=0$.
\end{remark}

As an example, consider the family $\mathcal F(M_{131}, K_{131}).$ We have $M_{131}=B_0(25)$, so $X_{B_0(25)}=X_0(25)$, and $g(X_0(25))=0$. On the other hand $i(K_{131})\approx 2.97.$ This shows that any $x\in X_0(25)(\Q)$ will lift to a super-sporadic rational point on $X_{K_{131}}^\psi$ (which is of genus 37) for some $\psi$. The family with the
largest value of $i(K_i)$ are the twist families $\mathcal F(M_i, K_i)$ for $i\in \{101,110,113,115,135,136\}$, with all of them having $ i(K_i)\approx 3.81,$ and $g(X_{K_i})=53.$

\section{Tables of groups}
In the table below we list the groups $M_i$ and $K_i$ that define our twist families with $\# X_{M_i}(\Q)=\infty$, as well as the invariants of the group $T_i=M_i/K_i$. The $M_i$ and $K_i$ are given by labels that are clickable links that lead to LMFDB pages of the corresponding modular curve $X_{M_i}$ or $X_{K_i}$. Much information can be read from these labels: the 1st number in the label is the level, the 2nd is the index, and the 3rd is the genus. The label contains a dash if and only if $-I$ is not in the group.

\begingroup
\small
\setlength{\tabcolsep}{3pt}
\begin{longtable}{lllc|lllc}
\caption{Groups} 
\label{table:bigtableofgroups}
\\
\toprule
$i$ & $M_i$ & $K_i$ & $M_i/K_i$ & $i$ & $M_i$ & $K_i$ & $M_i/K_i$ \\
\midrule
\endhead
\bottomrule
\endfoot
1 & \href{https://beta.lmfdb.org/ModularCurve/Q/1.1.0.a.1}{$X(1)$} & \href{https://beta.lmfdb.org/ModularCurve/Q/2.2.0.a.1}{$X_{\rm ns}(2)$} & $[ 2 ]$ & 2 & \href{https://beta.lmfdb.org/ModularCurve/Q/2.2.0.a.1}{$X_{\rm ns}(2)$} & \lmfdbmc{4.12.0-2.a.1.2} & $[ 6 ]$ \\
3 & \href{https://beta.lmfdb.org/ModularCurve/Q/2.3.0.a.1}{$X_0(2)$} & \lmfdbmc{4.12.0.d.1} & $[ 2, 2 ]$ & 4 & \href{https://beta.lmfdb.org/ModularCurve/Q/2.6.0.a.1}{$X(2)$} & \lmfdbmc{4.48.0-4.b.1.1} & $[ 2, 2, 2 ]$ \\
5 & \href{https://beta.lmfdb.org/ModularCurve/Q/3.3.0.a.1}{$X_{\rm ns}^+(3)$} & \lmfdbmc{6.12.1.b.1} & $[ 2, 2 ]$ & 6 & \href{https://beta.lmfdb.org/ModularCurve/Q/3.4.0.a.1}{$X_0(3)$} & \lmfdbmc{6.16.0-6.a.1.1} & $[ 2, 2 ]$ \\
7 & \href{https://beta.lmfdb.org/ModularCurve/Q/3.6.0.b.1}{$X_{\rm sp}^+(3)$} & \lmfdbmc{6.24.1.a.1} & $[ 2, 2 ]$ & 8 & \href{https://beta.lmfdb.org/ModularCurve/Q/3.12.0.a.1}{$X_{\rm sp}(3)$} & \lmfdbmc{18.144.3-18.b.1.1} & $[ 2, 6 ]$ \\
9 & \lmfdbmc{4.2.0.a.1} & \lmfdbmc{4.4.0-4.a.1.1} & $[ 2 ]$ & 10 & \href{https://beta.lmfdb.org/ModularCurve/Q/4.4.0.a.1}{$X_{\rm ns}^+(4)$} & \lmfdbmc{4.8.0.b.1} & $[ 2 ]$ \\
11 & \lmfdbmc{4.6.0.a.1} & \lmfdbmc{8.48.0-8.q.1.3} & $[ 2, 4 ]$ & 12 & \lmfdbmc{4.6.0.d.1} & \lmfdbmc{8.48.1-8.d.1.10} & $[ 2, 4 ]$ \\
13 & \lmfdbmc{4.6.0.b.1} & \lmfdbmc{8.48.0-4.b.1.9} & $[ 2, 4 ]$ & 14 & \lmfdbmc{4.6.0.e.1} & \lmfdbmc{8.48.1.o.1} & $[ 2, 2, 2 ]$ \\
15 & \href{https://beta.lmfdb.org/ModularCurve/Q/4.6.0.c.1}{$X_0(4)$} & \lmfdbmc{8.48.0-8.j.1.1} & $[ 2, 2, 2 ]$ & 16 & \lmfdbmc{4.8.0.b.1} & \lmfdbmc{4.16.0-4.b.1.1} & $[ 2 ]$ \\
17 & \href{https://beta.lmfdb.org/ModularCurve/Q/4.12.0.b.1}{$X_{\pm 1}(2,4)$} & \lmfdbmc{8.192.1-8.j.1.2} & $[ 2, 2, 2, 2 ]$ & 18 & \lmfdbmc{4.12.0.a.1} & \lmfdbmc{8.192.3-8.c.1.2} & $[ 2, 2, 4 ]$ \\
19 & \lmfdbmc{4.12.0.d.1} & \lmfdbmc{16.192.3-16.cl.1.6} & $[ 2, 2, 4 ]$ & 20 & \href{https://beta.lmfdb.org/ModularCurve/Q/4.24.0.b.1}{$X_{\rm sp}(4)$} & \lmfdbmc{16.768.13-16.bc.1.5} & $[ 2, 2, 2, 4 ]$ \\
21 & \href{https://beta.lmfdb.org/ModularCurve/Q/5.5.0.a.1}{$X_{S_4}(5)$} & \lmfdbmc{10.10.0.a.1} & $[ 2 ]$ & 22 & \href{https://beta.lmfdb.org/ModularCurve/Q/5.6.0.a.1}{$X_0(5)$} & \lmfdbmc{10.48.1-10.a.2.1} & $[ 2, 4 ]$ \\
23 & \href{https://beta.lmfdb.org/ModularCurve/Q/5.10.0.a.1}{$X_{\rm ns}^+(5)$} & \lmfdbmc{10.40.1.a.1} & $[ 2, 2 ]$ & 24 & \href{https://beta.lmfdb.org/ModularCurve/Q/5.15.0.a.1}{$X_{\rm sp}^+(5)$} & \lmfdbmc{10.60.3.b.1} & $[ 2, 2 ]$ \\
25 & \lmfdbmc{5.30.0.b.1} & \lmfdbmc{10.120.5.d.1} & $[ 2, 2 ]$ & 26 & \href{https://beta.lmfdb.org/ModularCurve/Q/5.30.0.a.1}{$X_{\rm sp}(5)$} & \lmfdbmc{50.1200.37-50.d.1.1} & $[ 2, 20 ]$ \\
27 & \lmfdbmc{6.6.0.b.1} & \lmfdbmc{6.6.0.b.1} & $[ ]$ & 28 & \lmfdbmc{6.8.0.a.1} & \lmfdbmc{12.96.2-12.a.1.6} & $[ 2, 6 ]$ \\
29 & \lmfdbmc{6.9.0.a.1} & \lmfdbmc{12.72.3.bf.1} & $[ 2, 2, 2 ]$ & 30 & \href{https://beta.lmfdb.org/ModularCurve/Q/6.12.0.a.1}{$X_0(6)$} & \lmfdbmc{12.96.1-12.h.1.7} & $[ 2, 2, 2 ]$ \\
31 & \lmfdbmc{6.18.0.b.1} & \lmfdbmc{12.144.5.n.1} & $[ 2, 2, 2 ]$ & 32 & \href{https://beta.lmfdb.org/ModularCurve/Q/6.24.0.a.1}{$X_{\pm 1}(2,6)$} & \lmfdbmc{12.384.5-12.b.1.4} & $[ 2, 2, 2, 2 ]$ \\
33 & \lmfdbmc{6.24.0.c.1} & \lmfdbmc{6.48.0-6.c.1.1} & $[ 2 ]$ & 34 & \href{https://beta.lmfdb.org/ModularCurve/Q/6.36.0.a.1}{$X_{\rm arith, \pm 1}(3,6)$} & \lmfdbmc{36.864.21-36.bo.1.2} & $[ 2, 2, 6 ]$ \\
35 & \href{https://beta.lmfdb.org/ModularCurve/Q/7.8.0.a.1}{$X_0(7)$} & \lmfdbmc{14.96.2-14.c.2.2} & $[ 2, 6 ]$ & 36 & \href{https://beta.lmfdb.org/ModularCurve/Q/7.21.0.a.1}{$X_{\rm ns}^+(7)$} & \lmfdbmc{14.84.5.b.1} & $[ 2, 2 ]$ \\
37 & \href{https://beta.lmfdb.org/ModularCurve/Q/7.28.0.a.1}{$X_{\rm sp}^+(7)$} & \lmfdbmc{14.112.5.b.1} & $[ 2, 2 ]$ & 38 & \lmfdbmc{8.8.0.a.1} & \lmfdbmc{16.16.0.b.1} & $[ 2 ]$ \\
39 & \lmfdbmc{8.12.0.v.1} & \lmfdbmc{16.96.3.r.2} & $[ 2, 2, 2 ]$ & 40 & \lmfdbmc{8.12.0.z.1} & \lmfdbmc{16.96.5.i.1} & $[ 2, 2, 2 ]$ \\
41 & \lmfdbmc{8.12.0.s.1} & \lmfdbmc{16.96.3.k.1} & $[ 2, 2, 2 ]$ & 42 & \lmfdbmc{8.12.0.t.1} & \lmfdbmc{8.96.1-8.w.1.1} & $[ 2, 4 ]$ \\
43 & \href{https://beta.lmfdb.org/ModularCurve/Q/8.12.0.n.1}{$X_0(8)$} & \lmfdbmc{16.192.1-16.a.2.2} & $[ 2, 2, 2, 2 ]$ & 44 & \lmfdbmc{8.12.0.o.1} & \lmfdbmc{16.192.3-16.w.2.4} & $[ 2, 2, 4 ]$ \\
45 & \lmfdbmc{8.12.0.h.1} & \lmfdbmc{32.96.0-32.c.1.1} & $[ 2, 4 ]$ & 46 & \lmfdbmc{8.12.0.q.1} & \lmfdbmc{16.96.1-8.g.2.3} & $[ 2, 4 ]$ \\
47 & \lmfdbmc{8.12.0.g.1} & \lmfdbmc{32.96.0-32.a.1.2} & $[ 2, 4 ]$ & 48 & \lmfdbmc{8.12.0.w.1} & \lmfdbmc{16.192.5-16.h.1.2} & $[ 2, 2, 4 ]$ \\
49 & \href{https://beta.lmfdb.org/ModularCurve/Q/8.16.0.a.1}{$X_{\rm ns}^+(8)$} & \lmfdbmc{8.32.1.c.1} & $[ 2 ]$ & 50 & \lmfdbmc{8.24.0.q.1} & \lmfdbmc{32.768.13-32.by.2.1} & $[ 2, 2, 2, 4 ]$ \\
51 & \lmfdbmc{8.24.0.bi.1} & \lmfdbmc{16.192.9.x.1} & $[ 2, 2, 2 ]$ & 52 & \lmfdbmc{8.24.0.t.1} & \lmfdbmc{32.768.21-32.jh.1.1} & $[ 2, 2, 8 ]$ \\
53 & \lmfdbmc{8.24.0.bj.1} & \lmfdbmc{16.192.5-16.bt.2.6} & $[ 2, 4 ]$ & 54 & \lmfdbmc{8.24.0.bp.1} & \lmfdbmc{16.192.3-16.bn.1.1} & $[ 2, 4 ]$ \\
55 & \lmfdbmc{8.24.0.f.1} & \lmfdbmc{16.768.21-16.ba.1.1} & $[ 2, 2, 8 ]$ & 56 & \lmfdbmc{8.24.0.bq.1} & \lmfdbmc{16.192.9.w.1} & $[ 2, 2, 2 ]$ \\
57 & \lmfdbmc{8.24.0.i.1} & \lmfdbmc{16.768.13-16.be.2.1} & $[ 2, 2, 2, 4 ]$ & 58 & \lmfdbmc{8.24.0.bf.1} & \lmfdbmc{8.96.1.e.1} & $[ 2, 2 ]$ \\
59 & \lmfdbmc{8.24.0.be.1} & \lmfdbmc{8.96.1-8.bb.1.1} & $[ 4 ]$ & 60 & \lmfdbmc{8.24.0.bt.1} & \lmfdbmc{16.192.9.bg.1} & $[ 2, 2, 2 ]$ \\
61 & \lmfdbmc{8.24.0.bn.1} & \lmfdbmc{8.96.1-8.h.2.1} & $[ 4 ]$ & 62 & \lmfdbmc{8.24.0.bh.1} & \lmfdbmc{16.192.5-16.i.2.6} & $[ 2, 4 ]$ \\
63 & \lmfdbmc{8.24.0.bc.1} & \lmfdbmc{16.192.3-16.cz.1.2} & $[ 2, 4 ]$ & 64 & \lmfdbmc{8.24.0.bo.1} & \lmfdbmc{16.192.9.fp.1} & $[ 2, 2, 2 ]$ \\
65 & \lmfdbmc{8.24.0.bs.1} & \lmfdbmc{16.192.9.b.1} & $[ 2, 2, 2 ]$ & 66 & \lmfdbmc{9.9.0.a.1} & \lmfdbmc{18.36.3.e.1} & $[ 2, 2 ]$ \\
67 & \href{https://beta.lmfdb.org/ModularCurve/Q/9.12.0.a.1}{$X_0(9)$} & \lmfdbmc{18.144.3-18.g.1.2} & $[ 2, 6 ]$ & 68 & \lmfdbmc{9.12.0.b.1} & \lmfdbmc{18.144.2-18.b.2.2} & $[ 2, 6 ]$ \\
69 & \lmfdbmc{9.18.0.a.1} & \lmfdbmc{18.72.3.e.1} & $[ 2, 2 ]$ & 70 & \lmfdbmc{9.18.0.d.1} & \lmfdbmc{18.72.5.a.1} & $[ 2, 2 ]$ \\
71 & \lmfdbmc{9.27.0.a.1} & \lmfdbmc{18.54.2.a.1} & $[ 2 ]$ & 72 & \href{https://beta.lmfdb.org/ModularCurve/Q/9.27.0.b.1}{$X_{\rm ns}^+(9)$} & \lmfdbmc{18.108.7.h.1} & $[ 2, 2 ]$ \\
73 & \lmfdbmc{9.36.0.b.1} & \lmfdbmc{54.1296.43-54.p.1.1} & $[ 2, 18 ]$ & 74 & \lmfdbmc{10.10.0.a.1} & \lmfdbmc{20.60.2-10.a.1.3} & $[ 6 ]$ \\
75 & \href{https://beta.lmfdb.org/ModularCurve/Q/10.18.0.a.1}{$X_0(10)$} & \lmfdbmc{20.288.5-20.a.1.4} & $[ 2, 2, 4 ]$ & 76 & \lmfdbmc{10.30.0.a.1} & \lmfdbmc{10.30.0.a.1} & $[ ]$ \\
77 & \href{https://beta.lmfdb.org/ModularCurve/Q/11.55.1.b.1}{$X_{\rm ns}^+(11)$} & \lmfdbmc{22.220.13.b.1} & $[ 2, 2 ]$ & 78 & \lmfdbmc{12.8.0.a.1} & \lmfdbmc{12.32.0-12.b.1.4} & $[ 2, 2 ]$ \\
79 & \lmfdbmc{12.12.0.q.1} & \lmfdbmc{12.48.3.i.1} & $[ 2, 2 ]$ & 80 & \lmfdbmc{12.18.0.l.1} & \lmfdbmc{24.288.17.dcb.1} & $[ 2, 2, 2, 2 ]$ \\
81 & \lmfdbmc{12.18.0.k.1} & \lmfdbmc{24.288.9-24.cg.1.1} & $[ 2, 2, 4 ]$ & 82 & \lmfdbmc{12.24.0.d.1} & \lmfdbmc{24.384.9-24.ib.1.14} & $[ 2, 2, 4 ]$ \\
83 & \href{https://beta.lmfdb.org/ModularCurve/Q/12.24.0.g.1}{$X_0(12)$} & \lmfdbmc{24.384.5-24.es.4.8} & $[ 2, 2, 2, 2 ]$ & 84 & \lmfdbmc{12.24.0.f.1} & \lmfdbmc{24.384.9-24.bl.1.21} & $[ 2, 2, 4 ]$ \\
85 & \lmfdbmc{12.36.0.h.1} & \lmfdbmc{12.144.7.bh.1} & $[ 2, 2 ]$ & 86 & \lmfdbmc{12.36.0.j.1} & \lmfdbmc{12.144.3-12.bx.1.4} & $[ 4 ]$ \\
87 & \lmfdbmc{12.36.0.k.1} & \lmfdbmc{24.144.7.lj.1} & $[ 2, 2 ]$ & 88 & \lmfdbmc{12.36.0.e.1} & \lmfdbmc{12.144.3-12.o.1.4} & $[ 4 ]$ \\
89 & \href{https://beta.lmfdb.org/ModularCurve/Q/13.14.0.a.1}{$X_0(13)$} & \lmfdbmc{26.336.9-26.d.1.1} & $[ 2, 12 ]$ & 90 & \lmfdbmc{14.16.0.a.1} & \lmfdbmc{28.576.16-28.c.1.6} & $[ 6, 6 ]$ \\
91 & \lmfdbmc{15.15.1.a.1} & \lmfdbmc{30.60.5.b.1} & $[ 2, 2 ]$ & 92 & \lmfdbmc{15.18.0.a.1} & \lmfdbmc{30.288.9-30.e.2.1} & $[ 2, 2, 4 ]$ \\
93 & \lmfdbmc{15.30.1.a.1} & \lmfdbmc{30.240.17.e.1} & $[ 2, 2, 2 ]$ & 94 & \lmfdbmc{15.45.1.a.1} & \lmfdbmc{30.360.25.y.1} & $[ 2, 2, 2 ]$ \\
95 & \lmfdbmc{16.16.0.b.1} & \lmfdbmc{32.32.0-16.b.1.1} & $[ 2 ]$ & 96 & \lmfdbmc{16.16.0.a.1} & \lmfdbmc{32.32.0.a.1} & $[ 2 ]$ \\
97 & \lmfdbmc{16.24.1.l.1} & \lmfdbmc{16.384.9-16.n.1.1} & $[ 4, 4 ]$ & 98 & \lmfdbmc{16.24.0.j.1} & \lmfdbmc{32.768.21-32.do.2.2} & $[ 2, 2, 8 ]$ \\
99 & \href{https://beta.lmfdb.org/ModularCurve/Q/16.24.0.g.1}{$X_0(16)$} & \lmfdbmc{32.768.13-32.n.2.2} & $[ 2, 2, 2, 4 ]$ & 100 & \lmfdbmc{16.48.0.h.1} & \lmfdbmc{32.192.11.w.1} & $[ 4 ]$ \\
101 & \lmfdbmc{16.48.0.c.1} & \lmfdbmc{32.1536.45-32.cr.1.6} & $[ 2, 16 ]$ & 102 & \lmfdbmc{16.48.0.t.2} & \lmfdbmc{32.192.11.bj.1} & $[ 4 ]$ \\
103 & \lmfdbmc{16.48.0.t.1} & \lmfdbmc{32.192.3-16.cj.1.3} & $[ 4 ]$ & 104 & \lmfdbmc{16.48.1.ch.1} & \lmfdbmc{32.384.21.ce.2} & $[ 2, 2, 2 ]$ \\
105 & \lmfdbmc{16.48.0.n.1} & \lmfdbmc{32.192.7.ch.1} & $[ 2, 2 ]$ & 106 & \lmfdbmc{16.48.1.cs.1} & \lmfdbmc{32.384.25.v.1} & $[ 2, 2, 2 ]$ \\
107 & \lmfdbmc{16.48.1.bv.1} & \lmfdbmc{32.384.21.r.1} & $[ 2, 2, 2 ]$ & 108 & \lmfdbmc{16.48.1.df.1} & \lmfdbmc{32.384.25.bg.1} & $[ 2, 2, 2 ]$ \\
109 & \lmfdbmc{16.48.1.bq.1} & \lmfdbmc{32.384.21.ke.1} & $[ 2, 2, 2 ]$ & 110 & \lmfdbmc{16.48.0.m.2} & \lmfdbmc{64.1536.53-64.fl.1.1} & $[ 2, 16 ]$ \\
111 & \lmfdbmc{16.48.1.dc.1} & \lmfdbmc{32.384.25.u.1} & $[ 2, 2, 2 ]$ & 112 & \lmfdbmc{16.48.1.de.1} & \lmfdbmc{32.384.25.b.1} & $[ 2, 2, 2 ]$ \\
113 & \lmfdbmc{16.48.0.m.1} & \lmfdbmc{64.1536.45-64.dc.1.1} & $[ 2, 16 ]$ & 114 & \lmfdbmc{16.48.0.h.2} & \lmfdbmc{32.192.3-16.bv.1.5} & $[ 4 ]$ \\
115 & \lmfdbmc{16.48.0.c.2} & \lmfdbmc{32.1536.53-32.h.1.8} & $[ 2, 16 ]$ & 116 & \lmfdbmc{16.48.0.i.1} & \lmfdbmc{32.192.7.w.1} & $[ 2, 2 ]$ \\
117 & \lmfdbmc{16.48.1.bl.1} & \lmfdbmc{32.384.21.bw.1} & $[ 2, 2, 2 ]$ & 118 & \lmfdbmc{16.48.1.cg.1} & \lmfdbmc{32.384.21.b.2} & $[ 2, 2, 2 ]$ \\
119 & \lmfdbmc{16.48.1.cc.1} & \lmfdbmc{32.384.21.bj.1} & $[ 2, 2, 2 ]$ & 120 & \lmfdbmc{18.24.0.c.1} & \lmfdbmc{18.48.0-18.c.1.1} & $[ 2 ]$ \\
121 & \lmfdbmc{18.24.0.b.1} & \lmfdbmc{36.864.28-36.e.2.7} & $[ 6, 6 ]$ & 122 & \href{https://beta.lmfdb.org/ModularCurve/Q/18.36.0.a.1}{$X_0(18)$} & \lmfdbmc{36.864.21-36.u.1.2} & $[ 2, 2, 6 ]$ \\
123 & \lmfdbmc{20.10.0.a.1} & \lmfdbmc{20.20.0-20.a.1.2} & $[ 2 ]$ & 124 & \lmfdbmc{20.20.1.c.1} & \lmfdbmc{20.40.2.b.1} & $[ 2 ]$ \\
125 & \lmfdbmc{21.63.1.a.1} & \lmfdbmc{42.504.37.o.1} & $[ 2, 2, 2 ]$ & 126 & \lmfdbmc{24.36.0.cj.1} & \lmfdbmc{48.576.41.bde.2} & $[ 2, 2, 2, 2 ]$ \\
127 & \lmfdbmc{24.36.0.cg.1} & \lmfdbmc{48.1152.41-48.bbp.1.1} & $[ 2, 2, 2, 4 ]$ & 128 & \lmfdbmc{24.36.1.gl.1} & \lmfdbmc{48.576.37.yl.1} & $[ 2, 2, 2, 2 ]$ \\
129 & \lmfdbmc{24.36.1.fw.1} & \lmfdbmc{48.576.37.ly.1} & $[ 2, 2, 2, 2 ]$ & 130 & \lmfdbmc{24.48.1.mk.1} & \lmfdbmc{24.192.13.y.1} & $[ 2, 2 ]$ \\
131 & \href{https://beta.lmfdb.org/ModularCurve/Q/25.30.0.a.1}{$X_0(25)$} & \lmfdbmc{50.1200.37-50.h.1.2} & $[ 2, 20 ]$ & 132 & \lmfdbmc{27.36.0.a.1} & \lmfdbmc{54.1296.43-54.bn.2.2} & $[ 2, 18 ]$ \\
133 & \lmfdbmc{28.16.0.a.1} & \lmfdbmc{28.192.6-28.j.2.1} & $[ 2, 6 ]$ & 134 & \lmfdbmc{30.30.1.a.1} & \lmfdbmc{30.30.1.a.1} & $[ ]$ \\
135 & \lmfdbmc{32.48.0.f.2} & \lmfdbmc{64.1536.53-64.ba.3.2} & $[ 2, 16 ]$ & 136 & \lmfdbmc{32.48.0.f.1} & \lmfdbmc{32.1536.45-32.d.1.4} & $[ 2, 16 ]$ \\
137 & \lmfdbmc{36.24.0.b.1} & \lmfdbmc{36.288.6-36.b.1.3} & $[ 2, 6 ]$ & 138 & \lmfdbmc{36.36.1.e.1} & \lmfdbmc{36.144.11.o.1} & $[ 2, 2 ]$ \\
\bottomrule
\end{longtable}
\endgroup

\bibliographystyle{alpha}
\bibliography{bibliography1}

@article {BalakrishnanMazur,
    AUTHOR = {Balakrishnan, Jennifer S. and Mazur, Barry},
     TITLE = {Ogg's torsion conjecture: fifty years later},
      NOTE = {With an appendix by Netan Dogra},
   JOURNAL = {Bull. Amer. Math. Soc. (N.S.)},
  FJOURNAL = {American Mathematical Society. Bulletin. New Series},
    VOLUME = {62},
      YEAR = {2025},
    NUMBER = {2},
     PAGES = {235--268},
      ISSN = {0273-0979,1088-9485},
   MRCLASS = {11G18 (14G05 14G35)},
  MRNUMBER = {4885858},
       DOI = {10.1090/bull/1851},
       URL = {https://doi.org/10.1090/bull/1851},
}

@article {Jones,
    AUTHOR = {Jones, Nathan},
     TITLE = {Almost all elliptic curves are {S}erre curves},
   JOURNAL = {Trans. Amer. Math. Soc.},
  FJOURNAL = {Transactions of the American Mathematical Society},
    VOLUME = {362},
      YEAR = {2010},
    NUMBER = {3},
     PAGES = {1547--1570},
      ISSN = {0002-9947,1088-6850},
   MRCLASS = {11G05 (11F80 11R45)},
  MRNUMBER = {2563740},
MRREVIEWER = {Ravi\ K.\ Ramakrishna},
       DOI = {10.1090/S0002-9947-09-04804-1},
       URL = {https://doi.org/10.1090/S0002-9947-09-04804-1},
}

@article {GonzalezCubic,
    AUTHOR = {Gonz\'alez, Josep},
     TITLE = {On cubic factors of {$j$}-invariants of quadratic {$\Bbb
              Q$}-curves of prime degree},
   JOURNAL = {J. Number Theory},
  FJOURNAL = {Journal of Number Theory},
    VOLUME = {128},
      YEAR = {2008},
    NUMBER = {2},
     PAGES = {377--389},
      ISSN = {0022-314X,1096-1658},
   MRCLASS = {11G18 (11G05)},
  MRNUMBER = {2380326},
MRREVIEWER = {Mihran\ Papikian},
       DOI = {10.1016/j.jnt.2007.05.011},
       URL = {https://doi.org/10.1016/j.jnt.2007.05.011},
}

@article {GalbaraithExpMath,
    AUTHOR = {Galbraith, Steven D.},
     TITLE = {Rational points on {$X_0^+(p)$}},
   JOURNAL = {Experiment. Math.},
  FJOURNAL = {Experimental Mathematics},
    VOLUME = {8},
      YEAR = {1999},
    NUMBER = {4},
     PAGES = {311--318},
      ISSN = {1058-6458,1944-950X},
   MRCLASS = {11G18 (11F11)},
  MRNUMBER = {1737228},
MRREVIEWER = {Conjeeveram\ S.\ Rajan},
       URL = {http://projecteuclid.org/euclid.em/1047262354},
}

@article {OggDiophantine,
    AUTHOR = {Ogg, A. P.},
     TITLE = {Diophantine equations and modular forms},
   JOURNAL = {Bull. Amer. Math. Soc.},
  FJOURNAL = {Bulletin of the American Mathematical Society},
    VOLUME = {81},
      YEAR = {1975},
     PAGES = {14--27},
      ISSN = {0002-9904},
   MRCLASS = {14G05 (10D05)},
  MRNUMBER = {354675},
MRREVIEWER = {Kuang-yen\ Shih},
       DOI = {10.1090/S0002-9904-1975-13623-8},
       URL = {https://doi.org/10.1090/S0002-9904-1975-13623-8},
}

@incollection {Mazur77RatPoints,
    AUTHOR = {Mazur, B.},
     TITLE = {Rational points on modular curves},
 BOOKTITLE = {Modular functions of one variable, {V} ({P}roc. {S}econd
              {I}nternat. {C}onf., {U}niv. {B}onn, {B}onn, 1976)},
    SERIES = {Lecture Notes in Math.},
    VOLUME = {Vol. 601},
     PAGES = {107--148},
 PUBLISHER = {Springer, Berlin-New York},
      YEAR = {1977},
   MRCLASS = {14G25 (10D15 14H25 14K15)},
  MRNUMBER = {450283},
MRREVIEWER = {M.\ A.\ Kenku},
}

@article {Baran,
    AUTHOR = {Baran, Burcu},
     TITLE = {Normalizers of non-split {C}artan subgroups, modular curves,
              and the class number one problem},
   JOURNAL = {J. Number Theory},
  FJOURNAL = {Journal of Number Theory},
    VOLUME = {130},
      YEAR = {2010},
    NUMBER = {12},
     PAGES = {2753--2772},
      ISSN = {0022-314X,1096-1658},
   MRCLASS = {11G05 (11G18 11R29)},
  MRNUMBER = {2684496},
MRREVIEWER = {Siman\ Wong},
       DOI = {10.1016/j.jnt.2010.06.005},
       URL = {https://doi.org/10.1016/j.jnt.2010.06.005},
}

@unpublished{MayleRouse,
      title={Rational points on modular curves via maps to elliptic curves with rank zero}, 
      author={Jacob Mayle and Jeremy Rouse},
note = {preprint, available at \url{https://arxiv.org/abs/2601.17202}}
}

@unpublished {vanHoeij,
    AUTHOR = {Mark van Hoeij},
     TITLE = {Low degree places on the modular curve ${X}_1(N)$},
   NOTE = {preprint, available at \url{https://arxiv.org/abs/1202.4355}}
}

@article {mazur77,
    AUTHOR = {Mazur, B.},
     TITLE = {Modular curves and the {E}isenstein ideal},
   JOURNAL = {Inst. Hautes \'Etudes Sci. Publ. Math.},
  FJOURNAL = {Institut des Hautes \'Etudes Scientifiques. Publications
              Math\'ematiques},
    NUMBER = {47},
      YEAR = {1977},
     PAGES = {33--186 (1978)},
      ISSN = {0073-8301},
     CODEN = {PMIHA6},
MRREVIEWER = {M. Ohta},
       URL = {http://www.numdam.org/item?id=PMIHES_1977__47__33_0},
}

@article {mazur78,
    AUTHOR = {Mazur, B.},
     TITLE = {Rational isogenies of prime degree (with an appendix by {D}.
              {G}oldfeld)},
   JOURNAL = {Invent. Math.},
  FJOURNAL = {Inventiones Mathematicae},
    VOLUME = {44},
      YEAR = {1978},
    NUMBER = {2},
     PAGES = {129--162},
      ISSN = {0020-9910},
     CODEN = {INVMBH},
   MRCLASS = {14K07 (10D35 14G25)},
MRREVIEWER = {V. V. Shokurov},
       DOI = {10.1007/BF01390348},
       URL = {http://dx.doi.org/10.1007/BF01390348},
}

@article {KM88,
    AUTHOR = {Kenku, M. A. and Momose, F.},
     TITLE = {Torsion points on elliptic curves defined over quadratic
              fields},
   JOURNAL = {Nagoya Math. J.},
  FJOURNAL = {Nagoya Mathematical Journal},
    VOLUME = {109},
      YEAR = {1988},
     PAGES = {125--149},
      ISSN = {0027-7630},
     CODEN = {NGMJA2},
MRREVIEWER = {Bert van Geemen},
       URL = {http://projecteuclid.org/euclid.nmj/1118780896},
}

@article {serre72,
    AUTHOR = {Serre, Jean-Pierre},
     TITLE = {Propri\'et\'es galoisiennes des points d'ordre fini des
              courbes elliptiques},
   JOURNAL = {Invent. Math.},
  FJOURNAL = {Inventiones Mathematicae},
    VOLUME = {15},
      YEAR = {1972},
    NUMBER = {4},
     PAGES = {259--331},
      ISSN = {0020-9910},
   MRCLASS = {14G25 (14K15)},
MRREVIEWER = {J. W. S. Cassels},
}

@preamble{
   "\def\cprime{$'$} "
}

@article {greicius10,
    AUTHOR = {Greicius, Aaron},
     TITLE = {Elliptic curves with surjective adelic {G}alois
              representations},
   JOURNAL = {Experiment. Math.},
  FJOURNAL = {Experimental Mathematics},
    VOLUME = {19},
      YEAR = {2010},
    NUMBER = {4},
     PAGES = {495--507},
      ISSN = {1058-6458},
   MRCLASS = {11G05 (11F80)},
MRREVIEWER = {N{\'u}ria Vila},
       DOI = {10.1080/10586458.2010.10390639},
       URL = {http://dx.doi.org/10.1080/10586458.2010.10390639},
}

@article{zywina_open_image_computations,
 author = {Zywina, David},
 title = {Open image computations for elliptic curves over number fields},
 fjournal = {Research in Number Theory},
 journal = {Res. Number Theory},
 issn = {2522-0160},
 volume = {11},
 number = {1},
 pages = {24},
 note = {Id/No 1},
 year = {2025},
 language = {English},
 doi = {10.1007/s40993-024-00599-2},
 zbMATH = {7958639}
}

@unpublished {zywina_possible_images,
    AUTHOR = {Zywina, David},
     TITLE = {On the possible images of the mod $\ell$ representations associated to elliptic curves over {$\mathbb{Q}$}},
   NOTE = {preprint, \url{https://arxiv.org/abs/1508.07660}
}
}

@unpublished{zywina_explicit_images,
 author = {Zywina, David},
 title = {Explicit open images for elliptic curves over {$\mathbb Q$}},
 note = {preprint, \url{https://arxiv.org/abs/2206.14959}}
}

@unpublished{zywina_possible_indices,
 author = {Zywina, David},
 title = {Possible indices for the {Galois} image of elliptic curves over {$\mathbb Q$}},
 note = {preprint, \url{https://arxiv.org/abs/1508.07663}}
}

@article {BELOV,
    AUTHOR = {Bourdon, Abbey and Ejder, \"{O}zlem and Liu, Yuan and Odumodu,
              Frances and Viray, Bianca},
     TITLE = {On the level of modular curves that give rise to isolated
              {$j$}-invariants},
   JOURNAL = {Adv. Math.},
  FJOURNAL = {Advances in Mathematics},
    VOLUME = {357},
      YEAR = {2019},
     PAGES = {106824, 33},
      ISSN = {0001-8708},
   MRCLASS = {14G35 (11G05)},
  MRNUMBER = {4016915},
       DOI = {10.1016/j.aim.2019.106824},
       URL = {https://doi.org/10.1016/j.aim.2019.106824},
}

@article {najman16,
    AUTHOR = {Najman, F.},
     TITLE = {Torsion of rational elliptic curves over cubic fields and
              sporadic points on {$X_1(n)$}},
   JOURNAL = {Math. Res. Lett.},
  FJOURNAL = {Mathematical Research Letters},
    VOLUME = {23},
      YEAR = {2016},
    NUMBER = {1},
     PAGES = {245--272},
      ISSN = {1073-2780},
   MRCLASS = {14H52 (11G05 11G18 14G25)},
MRREVIEWER = {John L. Boxall},
       DOI = {10.4310/MRL.2016.v23.n1.a12},
       URL = {http://dx.doi.org/10.4310/MRL.2016.v23.n1.a12},
}

@article {sutherland,
    AUTHOR = {Sutherland, Andrew V.},
     TITLE = {Computing images of {G}alois representations attached to
              elliptic curves},
   JOURNAL = {Forum Math. Sigma},
  FJOURNAL = {Forum of Mathematics. Sigma},
    VOLUME = {4},
      YEAR = {2016},
     PAGES = {e4, 79},
      ISSN = {2050-5094},
   MRCLASS = {11G05 (11F80 11G20 11Y16)},
MRREVIEWER = {Susan L. Schmoyer},
       DOI = {10.1017/fms.2015.33},
       URL = {https://doi.org/10.1017/fms.2015.33},
}

@article {RouseDZB,
    AUTHOR = {Rouse, Jeremy and Zureick-Brown, David},
     TITLE = {Elliptic curves over {$\Bbb Q$} and 2-adic images of {G}alois},
   JOURNAL = {Res. Number Theory},
  FJOURNAL = {Research in Number Theory},
    VOLUME = {1},
      YEAR = {2015},
     PAGES = {Art. 12, 34},
      ISSN = {2363-9555},
   MRCLASS = {11G05 (11F80)},
MRREVIEWER = {\'{A}lvaro Lozano-Robledo},
       DOI = {10.1007/s40993-015-0013-7},
       URL = {https://doi.org/10.1007/s40993-015-0013-7},
}

@article {frey,
    AUTHOR = {Frey, Gerhard},
     TITLE = {Curves with infinitely many points of fixed degree},
   JOURNAL = {Israel J. Math.},
  FJOURNAL = {Israel Journal of Mathematics},
    VOLUME = {85},
      YEAR = {1994},
    NUMBER = {1-3},
     PAGES = {79--83},
      ISSN = {0021-2172},
   MRCLASS = {11G30 (11G05 14G25 14H25)},
MRREVIEWER = {Takeshi Ooe},
       DOI = {10.1007/BF02758637},
       URL = {https://doi.org/10.1007/BF02758637},
}

@article {abramovich,
    AUTHOR = {Abramovich, Dan},
     TITLE = {A linear lower bound on the gonality of modular curves},
   JOURNAL = {Internat. Math. Res. Notices},
  FJOURNAL = {International Mathematics Research Notices},
      YEAR = {1996},
    NUMBER = {20},
     PAGES = {1005--1011},
      ISSN = {1073-7928},
   MRCLASS = {11G18 (11F32 14G35)},
MRREVIEWER = {M. Ram Murty},
       DOI = {10.1155/S1073792896000621},
       URL = {https://doi.org/10.1155/S1073792896000621},
}

@article {Balakrishnan,
    AUTHOR = {Balakrishnan, Jennifer and Dogra, Netan and M\"{u}ller, J. Steffen
              and Tuitman, Jan and Vonk, Jan},
     TITLE = {Explicit {C}habauty-{K}im for the split {C}artan modular curve
              of level 13},
   JOURNAL = {Ann. of Math. (2)},
  FJOURNAL = {Annals of Mathematics. Second Series},
    VOLUME = {189},
      YEAR = {2019},
    NUMBER = {3},
     PAGES = {885--944},
      ISSN = {0003-486X},
   MRCLASS = {14G05 (11G18 11G50 11Y50)},
  MRNUMBER = {3961086},
       DOI = {10.4007/annals.2019.189.3.6},
       URL = {https://doi.org/10.4007/annals.2019.189.3.6},
}

@article {Ogg74,
    AUTHOR = {Ogg, Andrew P.},
     TITLE = {Hyperelliptic modular curves},
   JOURNAL = {Bull. Soc. Math. France},
  FJOURNAL = {Bulletin de la Soci\'{e}t\'{e} Math\'{e}matique de France},
    VOLUME = {102},
      YEAR = {1974},
     PAGES = {449--462},
      ISSN = {0037-9484},
   MRCLASS = {14G05 (10D05 14H45)},
  MRNUMBER = {364259},
MRREVIEWER = {Kuang-yen Shih},
       URL = {http://www.numdam.org/item?id=BSMF_1974__102__449_0},
}

@article {GrossZagier85,
    AUTHOR = {Gross, Benedict H. and Zagier, Don B.},
     TITLE = {On singular moduli},
   JOURNAL = {J. Reine Angew. Math.},
  FJOURNAL = {Journal f\"{u}r die Reine und Angewandte Mathematik. [Crelle's
              Journal]},
    VOLUME = {355},
      YEAR = {1985},
     PAGES = {191--220},
      ISSN = {0075-4102},
   MRCLASS = {11F03 (11G05 14K15)},
  MRNUMBER = {772491},
MRREVIEWER = {R. W. K. Odoni},
}

@article {BP11,
    AUTHOR = {Bilu, Yuri and Parent, Pierre},
     TITLE = {Serre's uniformity problem in the split {C}artan case},
   JOURNAL = {Ann. of Math. (2)},
  FJOURNAL = {Annals of Mathematics. Second Series},
    VOLUME = {173},
      YEAR = {2011},
    NUMBER = {1},
     PAGES = {569--584},
      ISSN = {0003-486X},
   MRCLASS = {11G15 (11G05)},
  MRNUMBER = {2753610},
MRREVIEWER = {Damian R\"{o}ssler},
       DOI = {10.4007/annals.2011.173.1.13},
       URL = {https://doi.org/10.4007/annals.2011.173.1.13},
}

@article {BPR13,
    AUTHOR = {Bilu, Yuri and Parent, Pierre and Rebolledo, Marusia},
     TITLE = {Rational points on {$X^+_0(p^r)$}},
   JOURNAL = {Ann. Inst. Fourier (Grenoble)},
  FJOURNAL = {Universit\'{e} de Grenoble. Annales de l'Institut Fourier},
    VOLUME = {63},
      YEAR = {2013},
    NUMBER = {3},
     PAGES = {957--984},
      ISSN = {0373-0956},
   MRCLASS = {11G18 (11G05 11G16)},
  MRNUMBER = {3137477},
       DOI = {10.5802/aif.2781},
       URL = {https://doi.org/10.5802/aif.2781},
}

@Article{CGPS22,
 Author = {Clark, Pete L. and Genao, Tyler and Pollack, Paul and Saia, Frederick},
 Title = {The least degree of a {CM} point on a modular curve},
 FJournal = {Journal of the London Mathematical Society. Second Series},
 Journal = {J. Lond. Math. Soc., II. Ser.},
 ISSN = {0024-6107},
 Volume = {105},
 Number = {2},
 Pages = {825--883},
 Year = {2022},
 Language = {English},
 DOI = {10.1112/jlms.12518},
 zbMATH = {7731379}
}

@Article{LeFournLemos,
 Author = {Samuel {Le Fourn} and Pedro {Lemos}},
 Title = {{Residual Galois representations of elliptic curves with image contained in the normaliser of a nonsplit Cartan}},
 FJournal = {{Algebra \& Number Theory}},
 Journal = {{Algebra Number Theory}},
 ISSN = {1937-0652},
 Volume = {15},
 Number = {3},
 Pages = {747--771},
 Year = {2021},
 Publisher = {Mathematical Sciences Publishers (MSP), Berkeley, CA},
 DOI = {10.2140/ant.2021.15.747},
 MSC2010 = {11G05 11G18}
}

@unpublished{LagaShnidman-BiellipticPicardCurves,
	Author = {Laga, Jef and Shnidman, Ari},
NOTE = { J. London Math. Soc., 112: e70347. https://doi.org/10.1112/jlms.70347} ,
	Title = {The geometry and arithmetic of bielliptic {P}icard curves},
	Year = {2023+},
	}

@book {Serre-GaloisCohomology,
    AUTHOR = {Serre, Jean-Pierre},
     TITLE = {Galois cohomology},
    SERIES = {Springer Monographs in Mathematics},
   EDITION = {English},
      NOTE = {Translated from the French by Patrick Ion and revised by the
              author},
 PUBLISHER = {Springer-Verlag, Berlin},
      YEAR = {2002},
     PAGES = {x+210},
      ISBN = {3-540-42192-0},
   MRCLASS = {12G05 (11R34)},
  MRNUMBER = {1867431},
}

@book{serre_elladic,
 author = {Serre, Jean-Pierre},
 title = {Abelian {{\(\ell\)}}-adic representations and elliptic curves},
 fseries = {Research Notes in Mathematics},
 series = {Res. Notes Math.},
 volume = {7},
 isbn = {1-56881-077-6},
 year = {1998},
 publisher = {Wellesley, MA: A K Peters},
 language = {English},
 zbMATH = {1101830},
 Zbl = {0902.14016}
}

@article{DerickxNajman25,
url = {https://doi.org/10.1515/crelle-2025-0069},
title = {Classification of torsion of elliptic curves over quartic fields},
author = {Maarten Derickx and Filip Najman},
pages = {123--156},
volume = {2025},
number = {829},
journal = {J. Reine Angew. Math.},
fjournal = {Journal für die reine und angewandte Mathematik (Crelles Journal)},
doi = {doi:10.1515/crelle-2025-0069},
year = {2025},
lastchecked = {2026-01-15}
}

@article {BruinPryms,
    AUTHOR = {Bruin, Nils},
     TITLE = {The arithmetic of {P}rym varieties in genus 3},
   JOURNAL = {Compos. Math.},
  FJOURNAL = {Compositio Mathematica},
    VOLUME = {144},
      YEAR = {2008},
    NUMBER = {2},
     PAGES = {317--338},
      ISSN = {0010-437X,1570-5846},
   MRCLASS = {11G30 (14H40)},
  MRNUMBER = {2406115},
MRREVIEWER = {Jordi\ Gu\`ardia},
       DOI = {10.1112/S0010437X07003314},
       URL = {https://doi.org/10.1112/S0010437X07003314},
}

@article {Romagny,
    AUTHOR = {Romagny, Matthieu},
     TITLE = {Group actions on stacks and applications},
   JOURNAL = {Michigan Math. J.},
  FJOURNAL = {Michigan Mathematical Journal},
    VOLUME = {53},
      YEAR = {2005},
    NUMBER = {1},
     PAGES = {209--236},
      ISSN = {0026-2285,1945-2365},
   MRCLASS = {14A20 (14H10)},
  MRNUMBER = {2125542},
MRREVIEWER = {Ivan\ S.\ Kausz},
       DOI = {10.1307/mmj/1114021093},
       URL = {https://doi.org/10.1307/mmj/1114021093},
}

@article {Brunault,
    AUTHOR = {Brunault, Fran\c{c}ois},
     TITLE = {On the modularity of endomorphism algebras},
   JOURNAL = {Bull. Pol. Acad. Sci. Math.},
  FJOURNAL = {Bulletin of the Polish Academy of Sciences. Mathematics},
    VOLUME = {71},
      YEAR = {2023},
    NUMBER = {1},
     PAGES = {23--33},
      ISSN = {0239-7269,1732-8985},
   MRCLASS = {11F41 (11F25 11F70 11F80 14G32)},
  MRNUMBER = {4622407},
MRREVIEWER = {Ivan\ Mati\'c},
       DOI = {10.4064/ba230310-29-3},
       URL = {https://doi.org/10.4064/ba230310-29-3},
}

@unpublished {BourdonEjder25,
    AUTHOR = {Bourdon, Abbey and Ejder, \"Ozlem},
     TITLE = {Rational isolated $j$-invariants from {$X_1(\ell^n)$ }and {$X_0(\ell^n)$}},
   NOTE = {preprint, \url{https://arxiv.org/abs/2506.19560} }
}

@unpublished {etalebrauermanin2025,
    AUTHOR = {Dhillon, Ajneet and Lemire, Nicole and   Martin, Jonathan and  Wang, Yidi},   
    TITLE = {The étale Brauer-Manin obstruction for classifying stacks},
   NOTE = {preprint, \url{https://arxiv.org/abs/2512.01014} }
}

@unpublished {Elkies2006,
    AUTHOR = {Elkies, Noam},
     TITLE = {Elliptic curves with $3$-adic {G}alois representation surjective mod $3$ but not mod $9$},
   NOTE = {preprint, \url{https://arxiv.org/abs/math/0612734v1}}}

@article{BHKKLMNS25,
 author = {Bourdon, Abbey and Hashimoto, Sachi and Keller, Timo and Klagsbrun, Zev and Lowry-Duda, David and Morrison, Travis and Najman, Filip and Shukla, Himanshu},
 title = {Towards a classification of isolated {{\(j\)}}-invariants},
 fjournal = {Mathematics of Computation},
 journal = {Math. Comput.},
 issn = {0025-5718},
 volume = {94},
 number = {351},
 pages = {447--473},
 year = {2025},
 language = {English},
 doi = {10.1090/mcom/3956},
 zbMATH = {7929953}
}

@misc{Terao24,
 author = {Terao, Kenji},
 title = {Isolated points on modular curves},
 year = {2024},
 howpublished = {Preprint, {arXiv}:2412.13108 [math.{NT}] (2024)},
 url = {https://arxiv.org/abs/2412.13108},
 arXiv = {arXiv:2412.13108}
}

@misc{DeligneRapoprt73,
 author = {Deligne, Pierre and Rapoport, M.},
 title = {Moduli schemes of elliptic curves},
 year = {1973},
 language = {French},
 howpublished = {Modular {Functions} of one {Variable} {II}, {Proc}. {Int}. {Summer} {School}, {Univ}. {Antwerp} 1972, {Lect}. {Notes} {Math}. 349, 143-316 (1973).},
 doi = {10.1007/978-3-540-37855-6_4},
 zbMATH = {3440554},
 Zbl = {0281.14010}
}

@book{Poonen_book,
 author = {Poonen, Bjorn},
 title = {Rational points on varieties},
 fseries = {Graduate Studies in Mathematics},
 series = {Grad. Stud. Math.},
 issn = {1065-7339},
 volume = {186},
 isbn = {978-1-4704-3773-2; 978-1-4704-4315-3},
 year = {2017},
 publisher = {Providence, RI: American Mathematical Society (AMS)},
 language = {English},
 doi = {10.1090/gsm/186},
 zbMATH = {6828731},
 Zbl = {1387.14004}
}

@book{RibesZalesskii,
 author = {Ribes, Luis and Zalesskii, Pavel},
 title = {Profinite groups.},
 edition = {2nd ed.},
 fseries = {Ergebnisse der Mathematik und ihrer Grenzgebiete. 3. Folge},
 series = {Ergeb. Math. Grenzgeb., 3. Folge},
 issn = {0071-1136},
 volume = {40},
 isbn = {978-3-642-01641-7; 978-3-642-01642-4},
 year = {2010},
 publisher = {Berlin: Springer},
 language = {English},
 doi = {10.1007/978-3-642-01642-4},
 zbMATH = {5662352},
 Zbl = {1197.20022}
}

@article{RSZB22,
 author = {Rouse, Jeremy and Sutherland, Andrew V. and Zureick-Brown, David},
 title = {{{\(\ell\)}}-adic images of {Galois} for elliptic curves over {{\(\mathbb{Q}\)}}(and an appendix with {John} {Voight})},
 fjournal = {Forum of Mathematics, Sigma},
 journal = {Forum Math. Sigma},
 issn = {2050-5094},
 volume = {10},
 pages = {63},
 note = {Id/No e62},
 year = {2022},
 language = {English},
 doi = {10.1017/fms.2022.38},
 zbMATH = {7577487},
 Zbl = {1499.14057}
}

@misc{FurioLombardo252,
 author = {Furio, Lorenzo and Lombardo, Davide},
 title = {On 7-adic {Galois} representations for elliptic curves over {$\mathbb Q $}},
 year = {2025},
 howpublished = {Preprint, {arXiv}:2507.17967 [math.{NT}] (2025)},
 url = {https://arxiv.org/abs/2507.17967},
 arXiv = {arXiv:2507.17967}
}

@misc{FurioLombardo251,
 author = {Furio, Lorenzo and Lombardo, Davide},
 title = {Serre's uniformity question and proper subgroups of {$C_{ns}^+(p)$}},
 year = {2025},
 howpublished = {Algebra Number Theory, to appear. {arXiv}:2305.17780 [math.{NT}] (2025)},
 url = {https://arxiv.org/abs/2305.17780},
 arXiv = {arXiv:2305.17780}
}

@misc{Baletal25,
 author = {Balakrishnan, Jennifer S. and Betts, L. Alexander and Hast, Daniel Rayor and Jha, Aashraya and M{\"u}ller, Jan Steffen},
 title = {Rational points on the non-split {Cartan} modular curve of level 27 and quadratic {Chabauty} over number fields},
 year = {2025},
 howpublished = {Preprint, {arXiv}:2501.07833 [math.{NT}] (2025)},
 url = {https://arxiv.org/abs/2501.07833},
 arXiv = {arXiv:2501.07833}
}

@article{BrauJones16,
 author = {Brau, Julio and Jones, Nathan},
 title = {Elliptic curves with 2-torsion contained in the 3-torsion field},
 fjournal = {Proceedings of the American Mathematical Society},
 journal = {Proc. Am. Math. Soc.},
 issn = {0002-9939},
 volume = {144},
 number = {3},
 pages = {925--936},
 year = {2016},
 language = {English},
 doi = {10.1090/proc/12786},
 zbMATH = {6549095},
 Zbl = {1333.11050}
}

@article{JonesMcMurdy22,
 author = {Jones, Nathan and McMurdy, Ken},
 title = {Elliptic curves with non-abelian entanglements},
 fjournal = {The New York Journal of Mathematics},
 journal = {New York J. Math.},
 issn = {1076-9803},
 volume = {28},
 pages = {182--229},
 year = {2022},
 language = {English},
 url = {nyjm.albany.edu/j/2022/28-9.html},
 zbMATH = {7474314},
 Zbl = {1497.11150}
}

@article{Granville07,
 author = {Granville, Andrew},
 title = {Rational and integral points on quadratic twists of a given hyperelliptic curve},
 fjournal = {International Mathematics Research Notices. IMRN},
 journal = {Int. Math. Res. Not. IMRN},
 issn = {1073-7928},
 volume = {2007},
 number = {8},
 pages = {rnm027, 24},
 year = {2007},
 language = {English},
 url = {https://doi.org/10.1093/imrn/rnm027},
 zbMATH = {5177891},
 Zbl = {1129.11035}
}

@article{Rakvi24,
 author = {Rakvi},
 title = {A classification of genus 0 modular curves with rational points},
 fjournal = {Mathematics of Computation},
 journal = {Math. Comput.},
 issn = {0025-5718},
 volume = {93},
 number = {348},
 pages = {1859--1902},
 year = {2024},
 language = {English},
 doi = {10.1090/mcom/3907},
 zbMATH = {7833097},
 Zbl = {1557.11052}
}
\end{document}